\documentclass{amsart}
\usepackage{amsmath,amsfonts,amssymb,amsthm, esint}
\usepackage[english]{babel}

\usepackage{enumerate}
\usepackage{dsfont}

\usepackage[colorlinks,citecolor=blue]{hyperref}

\newtheorem{theorem}{Theorem}[section]
\newtheorem{definition}[theorem]{Definition}

\newtheorem{lemma}[theorem]{Lemma}
\newtheorem{proposition}[theorem]{Proposition}
\newtheorem{corollary}[theorem]{Corollary}

\theoremstyle{remark}
\newtheorem{remark}[theorem]{Remark}
\newtheorem{example}[theorem]{Example}

\numberwithin{equation}{section}
\numberwithin{figure}{section}

\newtheorem{examples}[theorem]{Examples}

\newcommand{\R}{\mathbb{R}}
\newcommand{\bra}[1]{\left(#1\right)}
\newcommand{\cur}[1]{\left\{#1\right\}}
\newcommand{\sqa}[1]{\left[#1\right]}
\newcommand{\ang}[1]{\left<#1\right>}

\newcommand{\abs}[1]{\left\lvert#1\right\rvert}
\newcommand{\norm}[1]{\left\lVert#1\right\rVert}
\let\d\undefined
\newcommand{\d}{\mathop{}\!\mathrm{d}}

\renewcommand{\v}{\mathsf{v}}

\renewcommand{\i}{\mathfrak{i}}

\renewcommand{\u}{\mathsf{u}}

\renewcommand{\v}{\mathsf{v}}
\newcommand{\w}{\mathsf{w}}

\newcommand{\simp}{\operatorname{Simp}}
\newcommand{\chain}{\operatorname{Chain}}
\newcommand{\vol}{\operatorname{vol}}

\newcommand{\cochain}{\operatorname{Cochain}}
\newcommand{\germ}{\operatorname{Germ}}

\newcommand{\sew}{\operatorname{sew}}

\newcommand{\dya}{\operatorname{dya}}

\renewcommand{\d}{\mathsf{d}}

\renewcommand{\c}{\mathsf{c}}

\newcommand{\flip}{\operatorname{flip}}
\newcommand{\cut}{\operatorname{cut}}

\renewcommand{\O}{D}
\newcommand{\conv}{\operatorname{conv}}
\newcommand{\diam}{\operatorname{diam}}
\newcommand{\Id}{\operatorname{Id}}
\newcommand{\cwedge}{\cup}

\usepackage{pgf,tikz}
\usetikzlibrary{arrows}

\usepackage{caption,subcaption}
\newcommand{\drawchain}[9]{
  	\coordinate (p_0) at ({#1},{#2});
  	\coordinate (p_1) at ({#3},{#4});
  	 \coordinate (p_2) at ({#5},{#6});
      \coordinate (a) at ({#1*4/6+(#3+#5)/6},{#2*4/6+(#4+#6)/6});
      \coordinate (c) at ({(#1+#3+#5)/3+#7*cos(#9*#8)/3},{(#2+#4+#6)/3+#7*sin(#9*#8)/3});
     \draw[thick, -](p_0)  -- (p_1) -- (p_2) -- (p_0);
     \draw[gray, thick, dashed] (p_0)  -- (a);
\draw[thick, gray, ->] (c) arc ({#9*#8}:{(270+#9*#8)}:{#7/3});			
}
		
\newcommand{\drawchaininv}[9]{
  	\coordinate (p_0) at (#1,#2);
  	\coordinate (p_1) at (#3,#4);
  	 \coordinate (p_2) at (#5,#6);
      \coordinate (a) at ({#1*4/6+(#3+#5)/6},{#2*4/6+(#4+#6)/6});
      \coordinate (c) at ({(#1+#3+#5)/3+#7*cos(#9*#8)/3},{(#2+#4+#6)/3+#7*sin(#9*#8)/3});
     \draw[thick, -](p_0)  -- (p_1) -- (p_2) -- (p_0);
     \draw[gray, thick, dashed] (p_0)  -- (a);
\draw[thick, gray, ->] (c) arc ({#9*#8}:{(-270+#9*#8)}:{#7/3});	

     }

\begin{document}

%\author{Valentino Magnani}
%\address{Valentino Magnani, Dipartimento di Matematica, Universit\`a di Pisa \\
%Largo Bruno Pontecorvo 5 \\ I-56127, Pisa}
%\email{valentino.magnani@unipi.it}
%\date{\today}

\author{Eugene Stepanov}
\address{St.Petersburg Branch of the Steklov Mathematical Institute of the Russian Academy of Sciences,
Fontanka 27,
191023 St.Petersburg, Russia
\and
Faculty of Mathematics, Higher School of Economics, Usacheva 6, 119048 Moscow, Russia
}

\email{stepanov.eugene@gmail.com}

\author{Dario Trevisan}
\address{Dario Trevisan, Dipartimento di Matematica, Universit\`a di Pisa \\
Largo Bruno Pontecorvo 5 \\ I-56127, Pisa}
\email{dario.trevisan@unipi.it}

\thanks{This work has been partially supported by the GNAMPA projects 2017 ``Campi vettoriali, superfici e perimetri in geometrie singolari'' and 2019 ``Propriet\`a analitiche e geometriche di campi aleatori'' and by the University of Pisa, Project PRA 2018-49.
The work of the first author has been also partially financed by the
RFBR grant \#20-01-00630 A
}

%and by the Russian government grant \#074-U01, the Ministry of Education and Science of Russian Federation project
%\#14.Z50.31.0031.}
\date{\today}

\title{Towards geometric integration of rough differential forms}%{Some Remarks on a two-dimensional sewing lemma}

\begin{abstract}
We provide a draft of a theory of geometric integration of
``rough differential forms'' which are generalizations of classical
(smooth) differential forms to similar objects with very low regularity,
for instance,
involving H\"{o}lder continuous functions that may be nowhere
differentiable. Borrowing ideas from the theory of rough paths, we show
that such a geometric integration can be constructed substituting
appropriately differentials with more general asymptotic expansions. This
can be seen as the basis of geometric integration similar to that used in
geometric measure theory, but without any underlying differentiable
structure, thus allowing Lipschitz functions and rectifiable sets to be
substituted by far less regular objects (e.g.\ H\"{o}lder functions
and their images which may be  purely unrectifiable). Our
construction includes both the one-dimensional Young integral and
multidimensional integrals introduced recently by R.\ Z\"ust, and provides also an alternative
(and more geometric) view on the standard construction of rough paths. To simplify the
exposition, we limit ourselves to integration of rough $k$-forms with $k\leq 2$.
%forms of dimensions not
%exceeding two.
\end{abstract}
\maketitle

\setcounter{tocdepth}{1}
\tableofcontents
%Cose da fare
%
%\begin{enumerate}
%\item Esempi: determinante, modulo di determinante.
%\item Esempi: Cantor, Intervallo, Sierpinski, Curva di Peano?
%\item Frobenius?
%\end{enumerate}

\section{Introduction}

Integrating a differential $k$-form $\eta:=v\d x^{j_1}\wedge \ldots\wedge \d x^{j_k}$ in $\R^n$, where $x^j$ stands for the $j$-th cordinate, $j=1,\ldots, n$, over
a surface $S$ parameterized, say, over $[0,1]^k$ by the map 
$h=(h^1,\ldots h^n)\colon [0,1]^k\to \R^n$, amounts to integrating over $[0,1]^k$
the pull-back $k$-form 
$g^*\eta= f\d g^1\wedge \ldots\wedge \d g^k$, where $f:= v\circ h$,
$g^i:= h^{j_i}$. This fits within the classical integration theory of differential forms when all the functions above are smooth, i.e.\ so are both the form
$\eta$ and the surface $S$. But if, for instance, $S$ is nonsmooth, e.g.\
just H\"{o}lder continuous, then even with smooth $\eta$ one has to deal 
formally with integration of the ``differential $k$-form'' of the type
$f\d g^1\wedge \ldots\wedge \d g^k$ involving nondifferentiable functions $g^i$, so that even the ``differential'' notation $\d g^i$ has no classical meaning. When additionally $\eta$ (i.e.\ the function $v$) is only H\"{o}lder continuous, then the situation even worsens because then also $f$ may 
be nondifferentiable. It is tempting therefore to construct a theory of 
integration of such $k$-forms (called further ``rough'' to stress inherent
connections of the theory developed with \emph{rough paths} in the 
one-dimensional case $k=1$), 
which can be considered a natural extension of  classical integration 
of smooth differential forms. 

\subsection*{Known constructions}
\subsubsection*{Integration of H\"older one-forms} In~\cite{young_inequality_1936} and independently in~\cite{kondurar_1937} it has been shown that
for H\"{o}lder continuous functions $f\in C^\alpha([0,1])$, $g\in C^\beta([0,1])$ with $\alpha+\beta>1$
the Riemann-type integral
$\int_0^1  f(x)\d g(x)$, which we may view as an integral of  a H\"{o}lder one form $f\d g$,  may be defined similarly to
the classical case of smooth $f$ and $g$,
as a suitable limit of finite sums
$\sum_i f(x_i) (g(x_{i+1})- g(x_{i}))$ over a refining family of partitions $0=x_0\leq x_1\leq \ldots \leq x_m=1$
of the interval of integration. Much later, in the late 1980s, this construction which nowadays is commonly known as the \emph{Young integral}, has been understood as a particular case of \emph{rough path} construction~\cite{lyons_differential_1998} via so-called \emph{sewing lemma}~\cite{gubinelli_controlling_2004, feyel_curvilinear_2006, friz_course_2014}. 

\subsubsection*{Integration of H\"older $k$-forms}
The natural extension of the Young integral from one-dimensional to multidimensional case is even more recent. In fact,
R.\ Z\"ust showed in~\cite{zust_integration_2011} that a $k$-dimensional Young-type integral (which we therefore will further call
\emph{Z\"ust integral})
\begin{equation}\label{eq:zust-intro} \int_{[0,1]^k }f\, \d g^1 \wedge \ldots \wedge\d g^k\end{equation}
 can be defined provided that
$f \in C^\alpha$, $g^i \in C^{\beta_i}$, and $\alpha + \sum_{i=1}^k \beta_i >k$.
The definition provided in~\cite{zust_integration_2011} is quite ``manual'' and uses dimension reduction via the integration by parts type procedure. Namely, 
the Z\"ust integral is defined inductively over $k$ as a limit over 
the refining sequence of dyadic partitions of the original cube $[0,1]^k$ 
into subcubes $Q_i$ of the sums
\[
\sum_i f(a_i) \int_{\partial Q_i}g^1 \d g^2 \wedge \ldots \wedge\d g^k,
\]
where $a_i$ is the center of $Q_i$, 
the integrals of the $(k-1)$-dimensional form $g^1 \d g^2 \wedge \ldots \wedge\d g^k$ over $(k-1)$-dimensional sides of $Q_i$ involved in the above formula
being defined at the previous step of induction. 

Let us point out that other approaches to integration of higher dimensional irregular objects (beyond measures) appear both in the stochastic and geometric literature. In the former setting, we mention~\cite{towghi_multidimensional_2002, friz_multidimensional_2010, chouk_2014} for integration of Young and rough ``sheets'',~\cite{gubinelli_paracontrolled_2015} for paracontrolled calculus based on Fourier decomposition,~\cite{FlandGubGiaqTort05-currents} for stochastic currents, and, finally, the reconstruction theorem in the celebrated theory of Regularity Structures~\cite{hairer_introduction_2015}. To the authors' knowledge, even the definition of Z\"ust integrals~\eqref{eq:zust-intro} for $k>1$ falls outside the scopes of these approaches, essentially because they miss a key ``dimension reduction'' feature which clearly appears when one uses the geometric language of simplices and their boundaries. On the geometric side, besides~\cite{zust_integration_2011}, we mention also possible connections with the chainlet approach in~\cite{harrison-1, harrison_differential_2006, harrison-3}, motivated by variational problems (Plateau's problem \cite{harrison-2}) and numerical approximations, with possibly  very irregular integration domains:  however, also in this case, already Z\"ust integrals~\eqref{eq:zust-intro} seem to be outside the scope of the theory. It has to be noted that similar theories can be developed for more general germs, for instance $(x,y) \mapsto \bra{\fint_{[x,y]} f}  (g(y)- g(x))$, $\fint$ standing for the average. One dimensional integrals arising from such germs have been studied in~\cite{matsaev_1972} (see also~\cite{kats_1999}); here for simplicity we do not pursue this direction. %Although it  also in this case, it seems that already the definition of~\eqref{eq:zust-intro} is missing.

\subsection*{Towards a general machinery: multidimensional sewing lemma}
In this paper we provide a  sketch of a possible general theory of integration of \emph{``rough'' differential forms}
which is in a sense halfway between classical calculus of differential forms and its purely discrete analogue: in fact, the differentials will be substituted by appropriate asymptotic expansions of finite differences.

\subsubsection*{Recovering Young and Z\"ust integrals}
 Our aim is to show that all the above-mentioned constructions, i.e.\ Young integrals, integrals of one-dimensional rough paths (of H\"older regularity larger than $1/3$) and multidimensional Z\"ust integrals are particular cases of this general theory. For instance, we see the Young integral as an integral over an interval of a one-dimensional
 form $f\,\d g$, and Z\"ust integral as an integral over a $k$-dimensional cube of a
$k$-dimensional  rough form
$f \d g^1 \wedge \ldots \wedge \d g^k$. Both are constructed from their ``discrete versions'', which can be viewed as ``germs of rough forms''. In particular,
the one-dimensional
form $f\,\d g$ and its integral are built from the one-dimensional germ, which is a function of two variables (or, equivalently, of a one-dimensional simplex) $(x,y)\mapsto f(x) (g(y)-g(x))$, while $k$-dimensional forms and their integrals are built from
$k$-dimensional germs. The latter are just functions of $k+1$ variables
(or, equivalently, of an ordered $(k+1$)-uple of points which can be
identified with a $k$-dimensional oriented simplex 
with a chosen ``initial'' vertex). 

\subsubsection*{The sewing procedure} 
The integral of a rough $k$-form over a $k$-dimensional oriented 
simplex can be considered also a $k$-dimensional germ, but satisfying certain ``regularity'' properties. Namely it is \emph{nonatomic} in the sense that it vanishes over ``degenerate'' simplices (those having zero $k$-dimensional volume), 
and its restriction to each $k$-dimensional plane is additive 
(a property further better expressed in a more common geometric language as being \emph{closed} with respect to the natural coboundary operator $\delta$). 
Such germs will be further called \emph{regular}. 
The construction of rough forms and their integrals from germs is done by
the ``sewing'' procedure, which can be therefore 
seen as a natural regularization of germs. 
It is provided by a multidimensional \emph{sewing lemma}, which is a natural extension of
the respective one dimensional result in rough paths theory. 
Technically, for $k=1$ and $k=2$ (the only cases studied in this paper) 
it relies on a dyadic decomposition of the domain (one or two-dimensional simplex) into self-similar pieces, and summing the respective discrete germs over the resulting partitions. 

The sewing lemma provides a natural map  sending some $k$-germs
into regular ones that can be viewed as integrals of rough $k$-forms.
The germs in its domain of definition  will be called \emph{sewable}, and can be considered as discrete prototypes of the respective integrals. 

\subsubsection*{Stokes-Cartan theorem} 
Note that the natural boundary operator $\partial$ mapping any $k$-dimensional simplex $S$ into a $(k-1)$-dimensional 
simplicial chain that is a sum of boundary simplices of $S$ 
with alternating signs induces by duality the coboundary operator $\delta$ on
$k$-dimensional germs. The latter will be shown to preserve regularity of germs,
and will be denoted by $\d$ over regular germs.  
We will show that for integrals of rough $k$-forms 
the Stokes-Cartan theorem holds, namely,
\[
\int_S \d \eta = \int_{\partial S} \eta
\]
for every $k$-dimensional simplex $S$ and rough $k$-form $\eta$. 
The proof is trivial for one-dimensional germs, but actually requires a slight modification of the dyadic summation in the two-dimensional case.

\subsubsection*{Problems with higher dimensional forms}
To keep the exposition simple, we limit ourselves to $0$, $1$ and $2$-dimensional forms, i.e.\ $k\leq 2$. The case of general $k$ requires 
much more technical effort due to the fact that one and 
two-dimensional simplices (i.e.\ line segments and triangles) 
can be easily divided into a finite number of self-similar simplices
with the same similarity ratio, thus allowing for easy 
``dyadic divisions'', but this is not the case for generic $k$-simplices with 
$k>2$. Of course, dyadic division is easier for cubes  than for simplices, but
on the other hand the notion of a germ is natural 
as a function of a simplex rather than of a cube, at least because simplices are preserved by pointwise transformations. One of the possible strategies for general $k$ could be defining $k$-germs as functions of simplices, as it is done here, then extending them to cubes and making a sewing procedure using the dyadic decomposition of cubes, and finally return to integrals over simplices.
 
 \subsubsection*{Future perspectives}

As a far reaching perspective of the theory presented, one can think of the development of an intrinsic approach to  ``rough'' exterior differential systems, e.g.\ the study of highly irregular ``integral manifolds'' of non differentiable, say, purely H\"{o}lder
vector fields (such problems may arise in many contexts of weak geometric structures, for instance in sub-Riemannian geometry~\cite{MagSteTrev16}); this however will be studied in the forthcoming paper~\cite{frobenius-zust}. Other interesting directions could be integration of nonsmooth forms over fractal-type or irregular random domains.

\subsection*{Structure of the paper}
In Section~\ref{sec:simp-germ} we introduce notation and basic facts about simplicial chains and ``germs of forms''. In Section~\ref{sec:sewable} we state the main results, in particular the multidimensional \emph{sewing lemma}  (Theorem~\ref{thm:sew-existence}), the uniqueness and continuity of integrals as well as the ``rough'' version of Stokes-Cartan theorem. In Section~\ref{sec:young-zust} we show how Young and Z\"ust integrals fit into our framework, and in Section~\ref{sec:rough} we point out extensions to more irregular functions, borrowing ideas from rough paths theory. Most of the technical proofs are collected in the appendices.

\subsection*{Acknowledgments}
Both authors thank V.\ Magnani for many useful suggestions in a preliminary version of this work.
%\section{Notation}
%
%For $d\ge 1$, $A \subseteq \R^d$, let $\conv(A)$ denote the convex envelope of $A$.
%%\cur{ \sum_{i=0}^k t_i p_i \colon \text{for every $i \in \cur{0,1, \ldots, k}$, $t_i \in [0,1]$ and $\sum_{i=0}^k t_i =1$ }}

\section{Simplices and germs}\label{sec:simp-germ}

We %introduce some basic notions of simplicial geometry,
will use some basic notions related to simplices and polyhedral chains,
extending the point of view of~\cite[Section 2]{gubinelli_controlling_2004}. Let throughout the paper $d$, $\tilde{d}$, $h$, $k$ stand for nonnegative integers and $\O \subseteq \R^d$, $\tilde{\O} \subseteq \R^{\tilde{d}}$ be sets.  For a set $S \subseteq \R^d$, let $\conv(S)$ stand for its convex envelope.

\subsection{Polyhedral chains}\label{sec:pol-chains}

The basic object of our study are real polyhedral $k$-chains (which provide a discrete counterpart of $k$-currents). As already mentioned, for the sake of simplicity we provide in this paper the
results which use these notions only for small values of $k$ ($k\le 3$), but at this point general definitions are more convenient.  % hence we always provide examples for these values.

\begin{definition}[simplices and chains] Any $S=[p_0 p_1 \ldots p_k] \in D^{k+1}$ such that $\conv(S) \subseteq \O$ is called a $k$-simplex (in $\O$). A (real polyhedral) $k$-chain (in $\O$) is any element %$C$ belonging to
of the real vector space generated by $k$-simplices in $\O$.
\end{definition}

%\footnote{}
We write $\simp^k(\O)$ for the set of $k$-simplices in $\O$ and $\chain^k(\O)$ for the linear space of $k$-chains in $\O$. It is useful to define $\chain^{-1}(D)$ as the set containing the singleton $[]=\emptyset$.
 %, omitting the reference to $\O$ when the latter is clear from the context (most frequently, when $\O$ is the whole space).
%\footnote{when $\O$ is understood, we only write $\simp^k$ and $\chain^k$.}
We use throughout the paper the property that any map defined on $\simp^k(\O)$ with values in a vector space naturally extends to a linear map on $\chain^k(\O)$. In particular, this holds for maps taking values in $\chain^{h}(\tilde {\O})$. %, for some $\tilde{d} \subseteq \R^{\tilde{d}}$.

\begin{remark}[representation of simplices and chains]
A $k$-simplex can be thought as the ``geometric'' simplex $\conv(S)$ with a base point  $p_0$ and the  ``orientation'' given by the ordering of the points. Hence, $0$-simplices are points, $1$-simplices are oriented segments from the point $p_0$ to $p_1$ and  $2$-simplices are pointed oriented triangles, see Figure~\ref{fig:representation-simplices}. Similarly, chains are weighted families of these objects (for simplicity we omit to specify unit weights in the drawings). %These representations will be helpful in Appendix~\ref{sec:proof-sew}.
\end{remark}

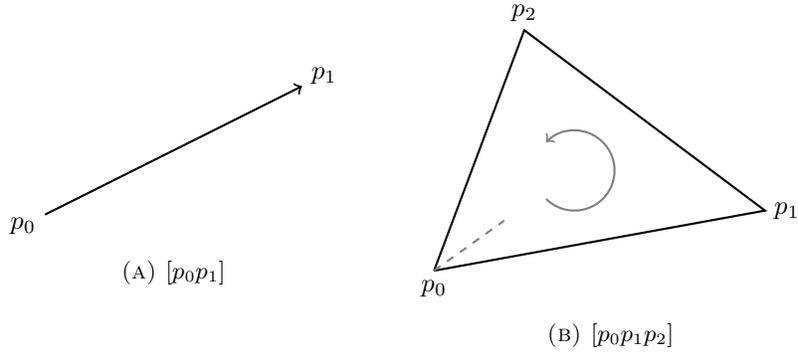
\begin{figure}[ht]
\begin{minipage}{0.45\textwidth}
\begin{center}
	\begin{tikzpicture}[scale=4]
	\node (0) at (-0.5, 0) {$p_0$};
	\node (1) at (0.5, 0.5) {$p_1$};
	\draw[thick, ->]   (0) edge (1);
		    \end{tikzpicture} \subcaption{ $[p_0 p_1]$}
    \end{center}
\end{minipage}
 \begin{minipage}{0.45\textwidth}
\begin{center}
     \begin{tikzpicture}[scale=4]
%	\coordinate (p_0) at (0,0);
%  	\coordinate (p_1) at (0.3,0.8);
%  	 \coordinate (p_2) at (1.1,0.2);
%     \draw[thick, -](p_0)  -- (p_1) -- (p_2) -- (p_0);
        \drawchain{0}{0}{0.3}{0.8}{1.1}{0.2}{0.4}{-135}{1}
	%\draw   (0,0) node[below left]{$p_0$}  (1.1,0.2) node[below]{$p_1$} (0.3,0.8) node[right]{$p_2$};
 %   	\draw   (0,0) node[below]{$p_0$}  (1.1,0.2) node[right]{$p_1$} (0.3,0.8) node[above]{$p_2$};
    \draw (p_0) node[below] {$p_0$};
    \draw (p_1) node[above] {$p_2$};
    \draw (p_2) node[right] {$p_1$};
    \end{tikzpicture}\subcaption{$[p_0p_1p_2]$}
    \end{center}
\end{minipage}
\caption{Representations of $1$-simplices  and  $2$-simplices .}\label{fig:representation-simplices}
\end{figure}

\subsubsection*{Push-forward} Any $\varphi: \O \subseteq \R^{d} \to \R^{\tilde{d}}$ induces a \emph{push-forward} operator (denoted with $\varphi_\natural$) mapping $k$-simplices in $\O$ to $k$-simplices in $\tilde{\O} := \conv( \varphi(\O))$,
\[  \varphi_\natural: \simp^{k}(\O)\to \simp^{k}(\tilde{\O}), \quad \quad S=[p_0 p_1 \ldots p_k] \mapsto \varphi_\natural S := [ \varphi(p_0) \varphi(p_1) \ldots \varphi(p_k)], \]
and extended by linearity to $\chain^k(\O)$. If $\varphi(x) = Ax+ q$ is affine, we  write $\varphi_\natural S = A_\natural S + q$ , and if $\varphi(x) = \lambda x$, for $\lambda \in \R$, we write $\varphi_\natural S =\lambda_\natural S$.

\begin{definition}[isometric simplices]
Given $S^i = [p_0^ip_1^i\ldots p_k^i] \in \simp^k$, $i \in \cur{1,2}$, we say that $S^1$ and $S^2$ are \emph{isometric} if there exists an isometry  $\i\colon\R^d\to \R^d$ such that $\i_\natural S^1 = S^2$, i.e.\ $\i (p_j^1) =  p_j^2$, for every $j \in \cur{0, 1, \ldots, k}$.
\end{definition}

Notice however that two $k$-simplices $S^i = [p_0^ip_1^i\ldots p_k^i]$, $i \in \cur{1,2}$, coincide  if and only if $p_j^1 = p_j^2$ for $j \in \cur{0,1, \ldots, k}$.

\subsubsection*{Boundary} The boundary operator is defined on $k$-simplices as
\[  \partial: \simp^k(\O) \to \chain^{k-1}(\O),  \quad \quad \partial [p_0 p_1 \ldots, p_k]  := \sum_{i=0}^k (-1)^i [p_0 \ldots \hat{p}_i \ldots p_k],\]
the notation $\hat{p}_i$ meaning that the element is removed from the list (Figure~\ref{fig:representation-simplices}), and extended by linearity over $k$-chains.

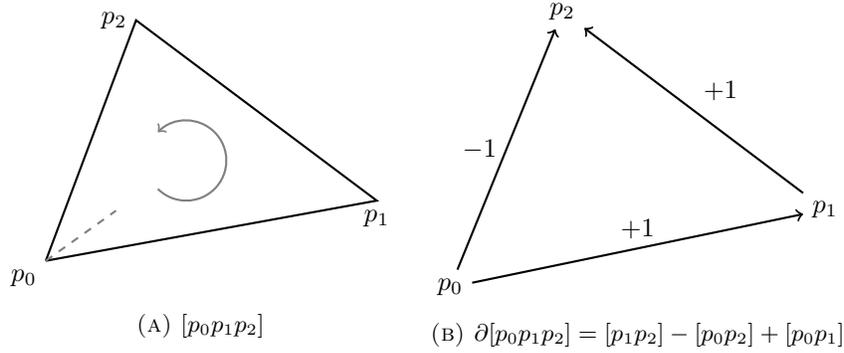
\begin{figure}[ht]
 \begin{minipage}{0.45\textwidth}
\begin{center}
     \begin{tikzpicture}[scale=4]
        \drawchain{0}{0}{0.3}{0.8}{1.1}{0.2}{0.4}{-135}{1}

	\draw   (0,0) node[below left]{$p_0$} (1.1,0.2) node[below]{$p_1$} (0.3,0.8) node[left]{$p_2$};
    \end{tikzpicture} \subcaption{$[p_0p_1p_2]$}
    \end{center}
\end{minipage}
 \begin{minipage}{0.45\textwidth}
\begin{center}
     \begin{tikzpicture}[scale=4]
	\node (0) at (0, 0)  [below left]{$p_0$};
	\node (2) at (0.3, 0.8) [above] {$p_2$};
		\node (1) at (1.1, 0.2) [right] {$p_1$};
	\draw[thick, ->]   (0) edge node[above]{$+1$}(1)
						(1) edge node[above right]{$+1$}(2)
						(0) edge node[left]{$-1$} (2);	
    \end{tikzpicture} \subcaption{$\partial [p_0p_1p_2] = [p_1p_2]- [p_0p_2] + [p_0p_1]$}
    \end{center}
\end{minipage}
\caption{A $2$-simplex and its boundary.}\label{fig:representation-boundary}
\end{figure}

Clearly, $\partial \varphi_\natural   = \varphi_\natural \partial$, for any $\varphi: \O \to \R^{\tilde{d}}$.  A simple argument gives $\partial \partial =0$, for in the double summation that one obtains, for any $i < j$ the $(k-1)$-simplex obtained removing the elements $\cur{p_i, p_j}$ appears once with sign $(-1)^{i+j}$ (when we remove first $j$ and then $i$) and once with the opposite sign  (when we remove first $i$ and then $j$).
%\[ \partial \partial [p_0p_1 \ldots p_k] = \sum_{i=0}^k \sum_{j=0}^{k-1} (-1)^{i+j} [p_0p_1 \ldots p_k] \]
%: in particular when $\varphi$ is affine, hence $\partial$ is a geometric map.

\subsubsection*{Permutations}
Any permutation $\sigma$ of $\cur{0,1, \ldots, k}$ induces an operator on germs (Figure~\ref{fig:representation-permutation})
for which we use the same notation
\[ \begin{split} \sigma: \simp^k(\O)&\to \simp^k(\O)\\
 S = [p_0 p_1\ldots p_k] & \mapsto \sigma S  := [p_{\sigma^{-1}(0)}  p_{\sigma^{-1}(1)} \ldots  p_{\sigma^{-1}(k)}].\end{split}\]
One has  $(\sigma \tau) S = \sigma  (\tau S) $,  for permutations $\sigma$, $\tau$,  and $\varphi_\natural \sigma   = \sigma \varphi_\natural$, for  $\varphi: \O \to  \R^{\tilde{d}}$.

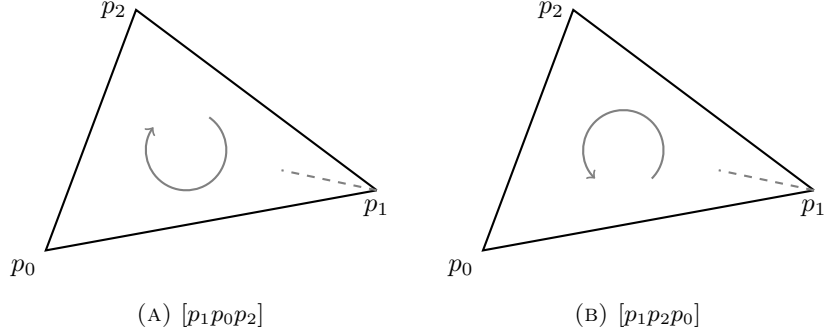
\begin{figure}[ht]
 \begin{minipage}{0.45\textwidth}
\begin{center}
     \begin{tikzpicture}[scale=4]
        \drawchaininv{1.1}{0.2}{0}{0}{0.3}{0.8}{0.4}{-305}{1}
	\draw   (0,0) node[below left]{$p_0$} (1.1,0.2) node[below]{$p_1$} (0.3,0.8) node[left]{$p_2$};
    \end{tikzpicture} \subcaption{$[p_1 p_0 p_2 ]$}
    \end{center}
\end{minipage}
 \begin{minipage}{0.45\textwidth}
\begin{center}
     \begin{tikzpicture}[scale=4]
        \drawchain{1.1}{0.2}{0.3}{0.8}{0}{0}{0.4}{-45}{1}
	\draw   (0,0) node[below left]{$p_0$} (1.1,0.2) node[below]{$p_1$} (0.3,0.8) node[left]{$p_2$};
    \end{tikzpicture} \subcaption{$[p_1 p_2 p_0]$}
    \end{center}
\end{minipage}

\caption{Permutations of the simplex $[p_0 p_1 p_2]$.}\label{fig:representation-permutation}
\end{figure}

\subsection{Germs of rough forms} Functions on simplices, called further \emph{germs of rough differential forms} or simply germs,  are discrete counterparts of differential forms.

\begin{definition}[germs and cochains] A \emph{$k$-germ} (of a rough differential $k$-form in $D$) is a function 
\[ 
\omega\colon \simp^k(D) \to \R \qquad\qquad S\mapsto \omega_{S}.
\]
We use also the notation $\omega_{p_0p_1\ldots p_k}$ for $\omega_S$ when $S=[p_0 p_1 \ldots p_k]$.  
A \emph{$k$-cochain} (in $D$) is a linear functional 
\[
\omega\colon \chain^k(D) \to \R,
 \qquad\qquad C\mapsto  \ang{C, \omega}.
\]
\end{definition}

We write $\germ^k(\O)$ for the set of $k$-germs and $\cochain^k(\O)$ for the linear space of $k$-cochains in $\O$. There is a natural correspondence between germs and cochains, and therefore we conveniently switch between the notations $\omega_S$ and $\ang{S, \omega}$, for $S \in \simp^k(\O)$, $\omega \in \germ^k(\O)$. For simplicity, we deal mainly with germs taking real values, but most of the notions and results extend verbatim to germs with values in a Banach space. % and we then  write $\germ^k(\O; B)$.

\begin{examples}
Since $0$-simplices are points, then $0$-germs are functions $f(p_0) = f_{p_0} = \ang{[p_0], f}$. $1$-germs are functions on oriented segments of $\O$ and $2$-germs are functions on pointed oriented triangles.
\end{examples}

We use further the following notions.
\begin{itemize}
%\begin{remark}[composition]
\item[(i)] \emph{[composition]}
The composition of $k$-germs $(\omega^i)_{i=1}^n$, $n>1$,  with a function $F\colon \R^n \to \R$ defines a $k$-germ $F\circ (\omega^i)_{i=1}^n : S \mapsto F\bra{\bra{\omega_S^i}_{i=1}^n}$. If $F$ is not linear one should be careful when dealing with the corresponding $k$-cochain, because it is not given by the composition of $F$ with the corresponding $k$-cochains.
%This notion however is useful e.g.\ to
 An easy example of the composition is the definition of the germ $\abs{\omega}$ by setting $\abs{\omega}_S := \abs{ \omega_S}$. % germ $\omega^2$ is $\omega^2_S := \bra{\omega_S}^2$.
%\end{remark}
%\begin{remark}
\item[(ii)]\emph{[pointwise product]}
As another example of composition, we  define the pointwise product of two  $k$-germs as $(\omega^1 \omega^2)_S := \omega^1_S \omega^2_S$, for $S \in \simp^k(\O)$.
%\end{remark}
%\begin{remark}
\item[(iii)] \emph{[inequalities]} Given $k$-germs $\omega$, $\tilde{\omega}$, we write $\omega \le \tilde{\omega}$ if, for every $S \in \simp^k(\O)$ one has $\omega_S \le \tilde{\omega}_S$.
\end{itemize}

Maps acting on simplices (or linear operators in the space of chains) induce by duality maps on germs (or linear operators on cochains): in general, given $\tau\colon\simp^{k}(\O) \to \chain^{h}(\tilde{\O})$, let $\tau^\prime\colon \cochain^{h}(\tilde{\O}) \to \cochain^k(\O)$ be the linear operator defined by
\[ (\tau ^\prime \omega)_{S} = \ang{S, \tau ^\prime\omega} := \ang{\tau S, \omega} \quad \text{for $S \in \simp^{k}(\O)$.}\]
%\footnote{Occasionally, we may write $\tau^\prime \omega_{S} = (\tau ^\prime \omega)_S$.}

\subsubsection*{(Algebraic) pull-back}
For $\varphi: \O \to \R^{\tilde{d}}$, $\tilde{\O} =\conv(\tilde{\O})$, the dual of $\varphi_\natural$ is the (algebraic) pull-back $\varphi^\natural := \varphi_\natural'$,
\[ \varphi^\natural: \germ^k(\tilde{\O}) \to \germ^k(\O), \quad  \quad  \ang{S, \varphi^\natural \omega} := \ang{ \varphi_\natural S, \omega}.\]

%If $k_1=k_2$, $\Omega_1 = \Omega_2$, we say that $\omega$ is $T$-invariant if $T^\prime \omega = \omega$. This is equivalent to require that any chain of the form $T_C - C$, for $C \in \chain^k(\Omega)$ belongs to the kernel of $\omega$.

\subsubsection*{Exterior derivative} The operator $\delta := \partial^\prime$ dual to the boundary operator $\partial$,
\[ \delta: \germ^k(\O) \to \cochain^{k+1}(\O), \quad  \quad \ang{S, \delta \omega} := \ang{ \partial S, \omega},\]  provides  a  discrete counterpart to the exterior derivative of differential forms.

\begin{examples} If $\omega \in \germ^0(\O)$, then for $[p_0 p_1] \in \simp^1(\O)$,
\[ \ang{[p_0 p_1], \delta f } =\ang{ [p_1] - [p_0], f}, \quad \text{or in other words} \quad \delta f_{p_0 p_1} = f_{p_1} - f_{p_0}.\]
If  $\omega \in \germ^1(\O)$, then for $[p_0 p_1 p_2] \in \simp^2(\O)$,
\[ \ang{[p_0 p_1 p_2], \delta \omega } = \ang{[p_1 p_2] - [p_0p_2]+[p_0 p_1], \omega}\]
 or, in a different notation,
$\delta \omega_{p_0p_1 p_2} =  \omega_{p_1 p_2} -  \omega_{p_0 p_2}+ \omega_{p_0 p_1}$.
\end{examples}

\begin{definition}[exact and closed germs]
We say that a $k$-germ $\omega$ is \emph{exact}, if $\omega = \delta \eta$ for some $(k-1)$-germ $\eta$, and is \emph{closed} if $\delta\omega = 0$.
\end{definition}

The identity $\partial \partial = 0$ gives $\delta \delta = 0$, i.e. every exact germ is closed.

 \begin{remark}[discrete Poincar\'{e} lemma]\label{rem:poincare}
If $D$ is convex, then all closed germs are exact. In fact, fix a $\bar{p} \in D$, and for $\omega \in \germ^{k}(\O)$ closed, define $\eta \in \germ^{k-1}(\O)$ via
\[ \ang{ [p_0 p_1 \ldots p_{k-1}], \eta} := \ang{[\bar p p_0 p_1 \ldots p_{k-1}], \omega},\]
for $[p_0p_1 \ldots p_{k-1}] \in \simp^{k-1}(\O)$.
%where $0 \in D$ denotes a chosen fixed point (not necessarily the origin of $\R^d$).
Given $S = [p_0 p_1 \ldots p_k] \in \simp^k(\O)$,  we have
\[ \begin{split} \ang{ S, \delta \eta} &= \sum_{i=0}^k (-1)^i \ang{ [p_0 p_1 \ldots \hat p _i \ldots p_k], \eta} = \sum_{i=0}^k (-1)^i \ang{ [\bar p p_0 p_1 \ldots \hat p _i \ldots p_k], \omega} \\
& = \ang{ [p_0 p_1  \ldots p_k], \omega} - \ang{ [\bar p  p_0 \ldots p_k], \delta \omega} =  \ang{S, \omega},
\end{split}\]
i.e., $\omega=\delta\eta$.
\end{remark}

\begin{definition}\label{def:alternating} We say that a $k$-germ is \emph{alternating}, if $\sigma^\prime \omega = (-1)^{\sigma} \omega$, i.e.,
\[  \ang{ \sigma S, \omega} = (-1)^{\sigma} \ang{ S, \omega} \quad \text{for every $S \in \simp^k(\O)$,}\]
for every permutation $\sigma$ of  $\cur{0,1, \ldots, k}$, where $(-1)^\sigma$ denotes its sign.
\end{definition}

It is not difficult to check that $\delta \omega$ is alternating when so is $\omega$. In particular,  $\delta f$ is alternating for every $f\in \germ^0(\O)$. %\begin{itemize}
%\item[(i)] $\delta f$ is alternating for every $f\in \germ^0(\O)$,
%\item[(ii)] \footnote{controllare}.
%\end{itemize}

\subsection{Cup products}

\label{sec:ext-product} We define the \emph{cup product} between a $k$-germ $\omega$ and an $h$-germ $\tilde{\omega}$ as the $(k+h)$-germ $\omega \cwedge \tilde{\omega}$ acting on $[p_0 p_1 \ldots p_k p_{k+1} \ldots p_{k+h}] \in \simp^{k+h}(\O)$ as
\[ \ang{[p_0 p_1\ldots p_k p_{k+1} \ldots p_{k+h} ], \omega \cwedge \tilde{\omega}} :=  \ang{[p_0 p_1\ldots p_k], \omega} \ang{ [p_k p_{k+1} \ldots p_{k+h} ], \tilde{\omega}}.\]
In~\cite{gubinelli_controlling_2004} this operation is called external product. In general, the cup product is not commutative, although when $f \in \germ^0(\O)$ we sometimes write with a slight abuse of notation $f \omega :=f\cwedge \omega$. For a $1$-form $\omega$, one has
\begin{equation}
\label{eq:identity-commutator}
\omega \cwedge f - f \cwedge \omega = (\delta f)\omega, \quad \text{i.e.} \quad  \omega_{p_0 p_1} f_{p_1} - f_{p_0}\omega_{p_0 p_1} = (\delta f_{p_0 p_1})\omega_{p_0 p_1}.
\end{equation}

The cup product is associative, and the following  Leibniz rule holds: for $\omega \in \germ^k(\O)$, $\tilde\omega \in \germ^{h}(\O)$ one has
\begin{equation}
\label{eq:identity-leibniz}
 \delta ( \omega \cwedge \tilde\omega) = (\delta \omega) \cwedge \tilde{\omega} + (-1)^{k} \omega \cwedge (\delta \tilde{\omega}). \end{equation}
Indeed, given $S= [p_0 p_1 \ldots p_{k} p_{k+1} \ldots p_{k+h}p_{k+h+1}]$, one has
\begin{align*} 
\langle S,  & \delta(\omega \cwedge \tilde{\omega})\rangle =
\langle\partial S,   \omega \cwedge \tilde{\omega}\rangle % &= \sum_{i=0}^{k+h+1} (-1)^i \ang{ [p_0\ldots \hat p_i \ldots p_{k+h+1}], \omega \cwedge \tilde{\omega}} \\
\\
& =  \sum_{i=0}^{k} (-1)^i \ang{ [p_0\ldots \hat p_i \ldots p_k p_{k+1} \ldots p_{k+h+1}], \omega \cwedge \tilde{\omega}} \\
& \quad + (-1)^k \sum_{i=1}^{h+1} (-1)^i \ang{ [p_0\ldots \ldots p_k p_{k+1} \ldots \hat p_{k+i} \ldots p_{k+h+1}], \omega \cwedge \tilde{\omega}}\\
& =  \sum_{i=0}^{k+1} (-1)^i \ang{ [p_0\ldots \hat p_i \ldots p_{k+1}], \omega} \ang{ [p_{k+1} \ldots p_{k+h+1}], \tilde{\omega}}  \\
%& \quad  +  (-1)^{k} \ang{ [p_0\ldots \hat p_i \ldots p_{k}], \omega} \ang{ [p_{k+1} \ldots p_{k+h+1}], \tilde{\omega}}\\
&  \quad + (-1)^k \sum_{i=0}^{h+1} (-1)^i \ang{ [p_0\ldots \ldots p_k], \omega} \ang{ [p_k  \ldots \hat p_{k+i} \ldots p_{k+h+1}],  \tilde{\omega}}\\
& = \ang{S, (\delta \omega) \cwedge \tilde{\omega} +  (-1)^k \omega \cwedge (\delta \tilde{\omega})}.
\end{align*}

\subsection{Continuity and convergence} \label{sec:gauges} We say that
\begin{enumerate}[(i)]
\item  $\omega \in  \germ^k(\O)$ is \emph{continuous}, if $S \mapsto \omega_S$ is a continuous function of $S \in \simp^k(\O) = \O^{k+1}$, i.e.\ given any sequence $S^n=[p^n_0 p^n_1 \ldots p^n_k]$ 
such that  that $p^n_i \to p_i$, as $n \to \infty$,  $i=0,  \ldots, k$,
 then $\omega_{S^n} \to \omega_S$, where  $S:=[p_0 p_1 \ldots p_k]$, 
\item a sequence $(\omega^n)_{n \ge 1} \subseteq \germ^k(\O)$ converge \emph{pointwise} to $\omega \in  \germ^k(\O)$, if $\lim_{n \to +\infty} \omega^n_S = \omega_S$ for every $S \in \simp^k(\O)$.
\end{enumerate}

 Quantitative aspects of continuity and convergence follow by introducing suitable \emph{gauges}, such as volumes.

\subsubsection*{Gauges} A $k$-gauge in $\O$ is a nonnegative $k$-germ $\v\in \simp^k(\O)$. We will call a $k$-gause \emph{uniform} if
\[ \ang{S, \u} = \ang{S', \u}, \quad \text{ whenever $S$ and $S'$ are isometric $k$-simplices in $\O$.}\]

Given a $k$-germ $\omega$ and a $k$-gauge $\v$,  we write $[\omega]_{\v} \in [0,\infty]$ for the infimum over all constants $\c$ such that $|\omega|\leq \c \v$, i.e.\
\[
\abs{ \omega _S } \le \c \ang{S, \v} \quad \text{for every $S \in \simp^k(\O)$.}
\]
In particular, we think of $[\omega]_{\v}<\infty$ as a condition on a ``size'' of $\omega$, compared to $\v$. Notice that, if $\v \le \tilde{\v}$, then $[\cdot]_{\tilde{\v}} \le [\cdot]_{\v}$ and that the triangle inequality holds:
\[
[\omega + \tilde{\omega}]_{\v} \le [\omega]_\v + [\tilde{\omega}]_\v.
\]
In fact, $\sqa{\cdot}_\v$ defines a norm over the space of all $\omega \in \germ^k(D)$ such that $\sqa{\omega} _\v<\infty$. %If 
For $\{\omega,\tilde{\omega}\} \subset \germ^k(\O)$, and a $k$-gauge $\v$, we write
\[ \omega \approx_\v \tilde{\omega} \quad \text{if and only if} \quad [\omega - \tilde{\omega}]_\v < \infty.\]
The relation $\approx_\v$ is an equivalence relation. The properties of gauges imply that, if $\omega \approx_{\v} \tilde{\omega}$ and $\eta \approx_{\w} \eta'$, then $\omega+ \eta \approx_{\max\cur{\v, \w}} \tilde{\omega} + \eta'$.

\begin{remark}[gauges and pull-back]\label{rem:moduli-pullback}
Given $\omega \in \germ^{k}(\tilde{\O})$ and $\varphi: \O \to \tilde{\O}$, we have
\[
[\varphi^\natural \omega]_{\varphi^\natural \v} \le [\omega]_{\v}
\]
for every $k$-gauge $\v$ in $\tilde{\O}$.
Indeed, for every $S \in \simp^k(\O)$ we have
\[ \abs{\ang{ S, \varphi^\natural \omega} } = \abs{\ang{ \varphi_\natural S, \omega} } \le  [\omega]_{\v} \ang{  \varphi_\natural S, \v} =  [\omega]_{\v} \ang{ S, \varphi^\natural \v}.\]
\end{remark}

\subsubsection*{Volumes}
Given $S=[p_0 \ldots p_k] \in \simp^k(\O)$, its $k$-volume is defined by
\[ \vol_k(S) :=  \frac {1}{k!}\sqrt{ \det( P P^T)}\]
where $P_{ij} := (p_i -p_0)^j$, $i, j =1,  \ldots, k$, $P^T$ stand for the transposition of the matrix $P$. 
More generally, for $h\le k$, we define the $h$-volume of $S$ as  %$h\le k$, write
\[ \vol_h(S) := \max_{ S' \in \simp^h(\O), S'\subset S } \vol_h(S'),\]
the maximum in the above formula being taken among all possible choices of $(h+1)$-uples $S'$ from the $(k+1)$-uple $S$. We are particularly interested in the cases $h=1,2$. For $h=1$, we have that
%where the maximum is taken among all possible ..
\[ \vol_1(S) = \max_{0\le i,j \le k} \abs{p_i -p_j } =: \diam(S)\]
is just the diameter of $S$.
Notice that $\vol_h$ (on $k$-simplices) defines a uniform $k$-gauge for every $h\leq k$, so in what follows we write expressions such as $\omega \le \diam$ or $\omega \le \vol_h$ and $\varphi^\natural \diam$, $\varphi^\natural \vol_h$.

\begin{example} Using inequalities between germs, we can restate $\alpha$-H\"older continuity for a $f \in \germ^0(\O)$ as $\abs{\delta f} \le \c \, \diam^{\alpha}$, that is
\[ \abs{(\delta f)_{p_0 p_1}} =  \abs{ f_{p_1} - f_{p_0} } \le \c \abs{p_1 -p_0}^{\alpha}\quad \text{for $[p_0 p_1] \in \simp^1(\O)$,}\]
where  $\c \ge 0$ is some constant. The smallest such constant is denoted $\sqa{\delta f}_{\alpha}$.
\end{example}

% Given $\u \in \chain^{k+1}(\O)$, a condition $[\delta \omega]_{\u}<\infty$ gives a sort of  ``modulus of closedness'' for $\omega$ in terms of $\u$.
Motivated by the previous example, when $\v= \diam^\gamma = \vol_1^\gamma$, for some $\gamma \in \R$, we simply write $\sqa{\cdot}_\gamma := [\cdot]_\v$ and $\approx_\gamma$ instead of $\approx_\v$.  Since, for some constant $\c > 0$, one has the isodiametric inequality (on triangles)
\[ \vol_2 \le \c \diam^{2},\]
then for some (different) constant $\c =\c(\gamma_2) \ge 0$, one has \[[\cdot]_{\diam^{\gamma_1+2\gamma_2} } \le \c [\cdot]_{\diam^{\gamma_1} \vol_2^{\gamma_2}},\] for any $\gamma_2 \ge 0$.% \le \c [\cdot]_{\gamma_1 + 2 \gamma_2}
%which can be shown arguing on triangles.

%\subsubsection*{Size} For $k\ge 0$, given a $k$-germ $\omega$, we say that a function $\s: \simp^{k+1}(\O) \to [0,\infty)$ is a \emph{size} of $\omega$ if
%\[  \abs{ \ang{S, \omega}} \le \s(S), \quad \text{for every $S \in \simp^{k}(\O)$}.\]

\subsubsection*{Dini gauges} We introduce the following generalization of the ``control functions'' from~\cite{feyel_curvilinear_2006}.  %, $\int_{0}^1 \eta(r) \frac{\d r}{r} < \infty$

\begin{definition}
For an $r \ge 0$,  a $k$-gauge $\u \in \germ^k(\O)$ is called \emph{$r$-Dini} if it is uniform and for every $S \in \simp^k(\O)$
one has
\[ \lim_{n \to \infty} 2^{n r}\ang{ \bra{2^{-n}}_\natural S, \u}  = 0.\]
If, in addition, $\u$ is such that
\begin{equation}\label{eq:v-dini} \ang{S, \tilde{\u}} := \sum_{n=0}^{+\infty}2^{n r} \ang{ \bra{2^{-n}}_\natural S, \u} < \infty \quad \text{for every $S \in \simp^k(\O)$},  \end{equation}
then $\u$ is called \emph{strong $r$-Dini}.
\end{definition}

Notice that, in general, $(2^{-n})_\natural S$ may not belong to $\simp^{k}(\O)$, but since $\u$ is the same on isometric simplices (because $\u$ is uniform), it is sufficient to consider any translated copy of it which  belongs to $\simp^{k}(\O)$.  The terminology ``Dini'' is reminiscent of the one, closely related, on moduli of continuity of functions.

\begin{remark}\label{rem_Dini_gauge1}
The following immediate observations are worth being made.
\begin{enumerate}[(i)]
\item Every $r$-Dini gauge is strong $h$-Dini for every $h< r$.
\item For a $k$-gauge $\u \in \germ^k(\O)$ to be $r$-Dini (resp.\ strong $r$-Dini), it is sufficient that
$\u(S)=o(\diam (S)^r)$ (resp.\ $\u(S)=o(\diam (S)^{r+\alpha})$ for some $\alpha>0$) as
$\diam S\to 0$, $S\in \germ^k(\O)$. For instance, if $h\in \{2,\ldots,  k\}$, then
the $k$-gauge $\vol_h$ is strong  $r$-Dini for every $r\in (0,h)$.
%\item[(i')] If $\u$ is a continuous $r$-Dini gauge $k$-gauge, then $\u_S =o(\diam(S)^r)$ for every $S \in \simp^k(
\item
Sums and pointwise maxima of a finite number of (strong) $r$-Dini gauges are (strong) $r$-Dini.
\item
If $\u$ is strong $r$-Dini, then $\tilde{\u}$ in~\eqref{eq:v-dini} is $r$-Dini, because
\[\begin{split} \lim_{n \to \infty}  2^{nr} \ang{ (2^{-n})_\natural S, \tilde{\u}}  & = \lim_{n \to \infty}  \sum_{m=0}^{+\infty}  2^{mr} 2^{nr} \ang{ \bra{2^{-m}}_{\natural} \bra{2^{-n}}_\natural S , \u}
 \\& = \lim_{n \to \infty} \sum_{m=n}^{+\infty}  2^{mr} \ang{ \bra{2^{-m}}_\natural S, \u} = 0.
\end{split}
\]
\item Moreover, if $\u$ is strong $h$-Dini for some $h>r$, then the $k$-gauge $\tilde{\u}_r$ defined as in~\eqref{eq:v-dini}, that is 
\[ \ang{S, \tilde{\u}_r} := \sum_{n=0}^{+\infty}2^{n r} \ang{ \bra{2^{-n}}_\natural S, \u} < \infty \quad \text{for every $S \in \simp^k(\O)$,}\]
is $h$-Dini as well, since it is smaller than the $k$-gauge $\tilde{\u}_h$ defined by
\[ \ang{S, \tilde{\u}_h} := \sum_{n=0}^{+\infty}2^{n h} \ang{ \bra{2^{-n}}_\natural S, \u} < \infty \quad \text{for every $S \in \simp^k(\O)$.}\]
\end{enumerate}
\end{remark}

\begin{example}\label{ex:gauge-diam} The $k$-gauge $\u := \diam^{\gamma_1} \vol_2^{\gamma_2}$ is strong $r$-Dini for every $r < \gamma_1+ 2 \gamma_2$ with
\begin{align*}
\ang{S, \tilde{\u}}& =  \sum_{n=0}^{+\infty} 2^{nr} \ang{ (2^{-n})_\natural S, \u} = \sum_{n=0}^{+\infty} 2^{nr} 2^{-n(\gamma_1 + 2 \gamma_2)} \ang{ S, \u}\\
& = (1-2^{r-(\gamma_1 + 2 \gamma_2)} )^{-1} \ang{S, \u},
\end{align*}
for every $S \in \simp^k(\O)$.
\end{example}

\section{Sewable germs}\label{sec:sewable}

In this section, we investigate $k$-germs that can be locally ``integrated'' along all affine $k$-planes, providing in particular a version of the so-called sewing lemma in case of $2$-germs ( Theorem~\ref{thm:sew-existence}). Its proof, together with that of the other stated theorems, is postponed  in Appendix~\ref{appendix-sewing}. %Since our main results hold for $k\in \cur{1,2}$, we restrict ourselves to these cases.

\subsection{Nonatomic and regular germs}

We further need the following notions.

\begin{definition}\label{def_reg1}
We call a $k$-germ $\omega \in \germ^{k}(\O)$
\begin{itemize}
\item[(i)] nonatomic, if $\omega_S= 0$  for every  $S \in \simp^k(\O)$ with $\vol_k{S}= 0$,
\item[(ii)] closed on $k$-planes, if for every $\varphi\colon I \subseteq \R^k \to \O$ affine, the germ $\varphi^\natural \omega \in \germ^k(I)$ is closed, i.e., $\delta (\varphi^\natural \omega) = 0$.
\item[(iii)] regular, if it is nonatomic and closed on $k$-planes.
\end{itemize}
\end{definition}

Clearly, each of the above conditions remain preserved by linear combinations of germs, and hence
define the linear subspaces of germs.
Notice also that the nonatomicity condition is void for $k=0$, while for $k=1$ it reduces just to $\omega_{pp} = 0$, for every $p \in \O$.  We also mention that by Remark~\ref{rem:poincare} one could replace closedness on $k$-planes with exactness on  convex subsets of $k$-planes.

\begin{remark}\label{prop_additive1}
Every regular $k$-germ $\omega \in \germ^{k}(\O)$ is \emph{additive} in the sense that
if $S\in \simp^k(\O)$ and $\{S_i\}\subset \simp^k(\O)$ is a finite family such that
$ S=\sum_i S_i + N + \partial T$, where $T$ is a $(k+1)$-chain with $\vol_{k+1}(T) = 0$, i.e., $\conv (T) \subseteq \pi$ for some  $k$-plane $\pi$, and $N = \sum_{j} a_j N_j \in \chain^k(\O)$ with $\vol_k (N_j)= 0$ for every $j$, then
then
\[
\ang{S, \omega}=\sum_i\ang{S_i, \omega},
\]
since $\ang{N, \omega} = 0$ and $\ang{\partial T, \omega}=\ang{T, \delta\omega} =0$.
\end{remark}

\begin{proposition}\label{prop_alternating1}
Every regular  $\omega \in \germ^{k}(\O)$ is alternating in the sense of Definition~\ref{def:alternating}. %,  $\sigma^\prime \omega = (-1)^\sigma \omega$, for every permutation $\sigma$ on $\cur{0,1,\ldots,k}$.
\end{proposition}

\begin{proof}
This follows from Lemma~\ref{lem:perm} and  Remark~\ref{rem:equivalence-closed} applied on the restriction of $\omega$ to a $k$-plane containing any given $S \in \simp^k(\O)$.
\end{proof}

\begin{proposition}\label{prop_reg1}
A $k$-germ $\omega \in \germ^{k}(\O)$ is regular, if and only if both $\omega$ and $\delta\omega$ are nonatomic.
\end{proposition}

\begin{proof}
If $\omega$ is regular, it suffices to show that $\delta\omega$ is nonatomic (because $\delta\omega$ is automatically closed). The latter is true because if
$T\in \simp^{k+1}(\O)$ with $\vol_{k+1} (T)=0$, then $T$ is supported over a $k$-plane, hence $\ang{T,\delta\omega}=0$, because
$\delta\omega$ is closed over $k$-planes.

On the contrary, if both $\omega$ and $\delta\omega$ are nonatomic, and $T\in \simp^{k+1}(\O)$ is supported over a $k$-plane,
then
$\ang{T,\delta\omega}=0$ because $\vol_{k+1}( T)=0$, so that $\omega$ is closed over $k$-planes and hence regular as claimed.
\end{proof}

Proposition~\ref{prop_reg1} easily implies the following corollary, since nonatomicity is stable by pointwise convergence.

\begin{corollary}\label{cor:pointwise-limit-regluar}
A limit of a pointwise converging sequence $(\omega^n)_{n \ge 1}$ of regular germs is regular.
\end{corollary}

\begin{remark}\label{rem_dom_reg1}
Note that $\delta\omega$ is nonatomic if and only if it is regular (since it is automatically closed). Thus
$\delta$ sends regular germs to regular ones, in particular, $\delta f$ for $f \in \germ^0(D)$ is always regular. Its restriction to regular germs will further be denoted by $\d$, namely,
for a regular $\omega\in \germ^k(\O)$ we write $\d\omega:=\delta\omega$.
\end{remark}

\begin{example}\label{ex:xidxj}
Let $x^j\colon D \subseteq \R^d \to \R$, $j=1,\ldots, d$, denote coordinate functions. Then the 
$1$-germ $\omega^{ij} = \frac 1 2 \bra{ x^i \cup (\delta x^j) + (\delta x^j ) \cup x^i}$ is regular. In fact it is nonatomic, and using~\eqref{eq:identity-leibniz} we get
\[  \delta \omega^{ij} = \frac{1}{2} \bra{ \delta x^i \cup \delta x^j - \delta x^j \cup \delta x^i},\]
which acts on $S=[pqr] \in \simp^2(D)$ as the (oriented) area of the projection of $S$ on the $(x^i, x^j)$-plane:
\begin{equation}\label{eq:signed-area} \ang{ [pqr],  \delta( \omega^{ij}) } = \frac 1 2 \det \bra{ \begin{array}{cc} q^i -p^i & r^i - q^i \\
q^j - p^j & r^j - q^j \end{array}}.\end{equation}
Hence if $\vol_2([pqr]) = 0$, i.e., the three vertices are collinear, then $\delta \omega^{ij}_S = 0$ or in other words  $\delta \omega^{ij}$ is also nonatomic. Therefore, the regularity of $\omega^{ij}$ follows from Proposition~\ref{prop_reg1}.
\end{example}

For regular germs $\omega \in \germ^k(\O)$, we use the integral notation, i.e.\ write
\[
\int_S \omega := \ang{S, \omega} \quad\mbox{for $S \in \simp^k(\O)$.}
\]
This is justified by the following example.

\begin{example}\label{ex:leb-germ} For any continuous $f$, $g \in \germ^0(D)$ with  $\sqa{ \delta g}_1 < \infty$, i.e., $g$ is Lipschitz continuous, consider the $1$-germ $f \d g$, defined by
\begin{equation}\label{eq:integral-leb} \ang{ [pq], f \d g } = \int_{[pq]} f \d  g  := \int_{0}^1 f(p+ t(q-p)) \frac{\d}{\d t} g(p+t(q-p)) \d t,\end{equation}
where the integral is intended in the Lebesgue sense (recall that $t\in [0,1] \mapsto g(p+ t(q-p))$ is Lipschitz continuous). This germ is clearly nonatomic and $\delta (f \d g)$ restricted to every line is zero, which is shown by a simple reparametrization and additivity of Lebesgue integral. Note that the germ $\omega^{ij}$ in the above Example~\ref{ex:xidxj} coincides with $x^i \d x^j$: in fact, by a direct computation of the integral in~\eqref{eq:integral-leb}
one gets
\begin{align*}
\ang{[pq],x^1\d x^j} & = \int_{0}^1 \left(p^i+ t\left(q^i-p^i\right)\right) \left(q^j-p^j\right) \d t \\
 & = p^i\left(q^j-p^j\right) +\frac 1 2  \left(q^i-p^i\right)\left(q^j-p^j\right)
 = \ang{[pq],\omega^{ij}}.
\end{align*}
\end{example}

\subsection{Sewing as regularization}

\begin{definition}[sewable germs]\label{def_sewable1}
 We say that $\omega \in \germ^{k}(\O)$ is $\v$-sewable for some  $k$-Dini gauge $\v  \in \germ^{k}(\O)$,
 if there exists
a regular germ $\sew \omega \in \germ^k(\O)$  such that
\[  \omega \approx_\v \sew \omega.\]
We will say that $\omega \in \germ^{k}(\O)$ is sewable, if it is $\v$-sewable
for some  $k$-Dini gauge $\v  \in \germ^{k}(\O)$. 
\end{definition}

When, $\v = \diam^\alpha$, we say $\alpha$-sewable instead of $\v$-sewable. If $\omega \in \germ^{k}(\O)$ is sewable, then
  $\sew\omega$ is alternating (by Proposition~\ref{prop_additive1}) and additive (in the sense of Remark~\ref{prop_alternating1}).

Sewable germs can be thought of as the linear subspace of germs ``infinitesimally'' near to regular germs. 
 Any regular $\omega$ is $\v$-sewable for any $\v$ with $\sew \omega := \omega$.

\begin{example}\label{ex:sew-leb}
For any $f$, $g \in \germ^0(D)$ with $f$ (uniformly) continuous and $\sqa{ \delta g}_1 < \infty$, the $1$-germ $f \cup \delta g$ is sewable with $\sew (f \cup \delta g) = f \d g$ introduced in Example~\ref{ex:leb-germ}. Indeed, letting $\eta_f$ stand for the modulus of continuity of $f$, we have
\[    \begin{split} \abs{ \int_{[pq]} f \d g - f_p \cup \delta g_{pq}   } & = \abs{ \int_{0}^1 \bra{f(p+ t(q-p)) -f(p)} \frac{\d}{\d t} g(p+t(q-p)) \d t} \\
& \le \eta_f(|q-p|) [\delta g]_{\diam} |q-p| =: \v([pq]),  \end{split}\]
and clearly $\v$ is a $1$-Dini gauge.
\end{example}

\begin{example}\label{ex_abs_dt1}
The germ $\omega\in \germ^1(\R)$ defined by $\omega_{st}:=|t-s|$ is not sewable. In fact, otherwise for
some regular $\eta\in \germ^1(\R)$ one would have with $s_n:= 2^{-n}$ the relationships
\begin{align*}
s_n & =\omega_{s_n 0} = \eta_{s_n 0}+o(s_n),\\
s_n & =\omega_{0 s_n} = \eta_{0 s_n}+o(s_n)=-\eta_{s_n 0}+o(s_n),
\end{align*}
so that $2s_n= o(s_n)$, a contradiction.
\end{example}

%\subsection{Uniqueness}

Note that the use of only $k$-Dini gauges in the Definition~\ref{def_sewable1} of sewable $k$-germs is not
casual. In fact, the germ $\omega\in \germ^1(\R)$ from the above Example~\ref{ex_abs_dt1} is $\v$-near both the regular $1$-germ $\eta\equiv 0$ and, say, the regular $1$-germ $\xi$ defined by $\xi_{st}:= t-s$
with $\v:=\diam^\alpha$ with every $\alpha\in [0,1)$.
On the contrary, with Definition~\ref{def_sewable1},
the sewing map $\omega \mapsto \sew \omega$ is well-defined, %on sewable germs,
i.e.,  $\sew \omega$ is uniquely determined, if it exists (i.e. when $\omega$ is sewable), as a corollary of the following result.

\begin{theorem}[sewing, uniqueness]\label{thm:sew-unique}
Let $k \in \cur{1,2}$. If $\omega \in \germ^k(\O)$ is regular and $\omega \approx_\v 0$ for some  $k$-Dini gauge $\v \in \germ^k(\O)$, then $\omega =0$. 
\end{theorem}

\begin{corollary}
Let $\omega \in \germ^k(\O)$ be both $\v$-sewable and $\tilde{\v}$-sewable, with associated regular $k$-germs $\sew \omega$ and $\widetilde{\sew} \omega$. Then $\sew \omega = \widetilde\sew \omega$.
\end{corollary}

\begin{proof}
By definition, $\sew \omega \approx_{\v} \omega \approx_{\tilde{\v}}\widetilde \sew \omega$. Then 
$\sew \omega \approx_{\w} \widetilde\sew  \omega$, i.e., 
\[\sew \omega - \widetilde \sew \omega \approx_\w 0 \quad\text{with $\w:= \max\cur{\v, \tilde{\v}}$},
\] and $\w$ is a $k$-Dini gauge in view of Remark~\ref{rem_Dini_gauge1}(iii). Since the difference of regular germs is regular, by Theorem~\ref{thm:sew-unique} we obtain the thesis.
\end{proof}

Another simple consequence of Theorem~\ref{thm:sew-unique} is the following result.

\begin{corollary}\label{coro:close-to-sewable} Let $\omega, \tilde{\omega} \in \germ^{k}(\O)$, where $\omega$ is $\v$-sewable and $\omega \approx_\v \tilde{\omega}$. Then  $\tilde{\omega}$ is $\v$-sewable with $\sew\tilde{\omega}= \sew \omega$.
\end{corollary}

\begin{proof}
One has $\tilde{\omega} \approx_\v \omega \approx_\v \sew \omega$.
%One has $[ \tilde{\omega} -\sew \omega ]_\v \le [\omega -\sew \omega ]_{\v} + [\omega - \tilde{\omega}]_{\v} <\infty$.
\end{proof}

As an example, if $\omega\in \germ^1(D)$ is sewable, also $\tilde{\omega}_{st}  := \omega_{st} + |t-s|^p$ with $p>1$ is sewable, and $\sew \tilde{\omega}  = \sew \omega$.

As an additional consequence, the sewing map is linear: if $\omega$, $\tilde{\omega}$ are sewable, then any linear combination $\lambda \omega + \tilde{\lambda} \tilde{\omega}$ is sewable as well, with
\[ \sew( \lambda \omega + \tilde{\lambda} \tilde{\omega}) = \lambda \sew\omega + \tilde{\lambda} \sew \tilde{\omega},\]
simply because $\lambda \omega + \tilde{\lambda} \tilde{\omega} \approx_{\w} \lambda \sew \omega + \tilde{\lambda} \sew \tilde{\omega}$, with $\w := \max\cur{\v, \tilde{\v}}$, if we assume that $\omega$ is $\v$-sewable and $\tilde{\omega}$ is $\tilde{\v}$-sewable.

\begin{corollary}[locality]\label{cor:locality}
Let $\tilde \O \subseteq \O$ and $\omega \in \germ^k(\O)$ be sewable. Then, its restriction $\tilde{\omega}$ to $\simp^k(\tilde{ \O})$ is  sewable and $\sew \tilde{\omega}$ is the restriction of $\sew \omega$ to $\simp^k(\tilde{\O})$.
\end{corollary}

\begin{proof}
The restriction of $\sew \omega$ to $\simp^k(\tilde{\O})$ is regular, and $\tilde{ \omega} \approx_{\tilde{ \v}} \sew \omega$ with $\tilde{\v}$ being the restriction in $\tilde{\O}$ of a $k$-Dini gauge $\v \in \germ^k(\O)$ such that $\omega$ is $\v$-sewable.
\end{proof}

A further consequence is the following result. Denote by $\norm{\cdot}_0$ the usual supremum norm of a function.

\begin{corollary}\label{coro:close-to-sewable-2} Let $f \in \germ^0(\O)$, $\omega \in \germ^{k}(\O)$ being  $\v$-sewable, and $\norm{f}_0 <\infty$. Then $f\cwedge \omega$ is $\v$-sewable, if and only if  $f \cwedge \sew \omega$ is $\v$-sewable, and in such a case one has
\[ \sew( f\cwedge \omega) = \sew \bra{ f \cwedge \sew \omega}.\]
%and $f \in with $\omega$ $\v$-sewable and $\omega \approx_\v \tilde{\omega}$. Then  $\tilde{\omega} \approx_\v \sew \omega$.
\end{corollary}

\begin{proof}
One has $[f\cwedge \omega - f \cwedge \sew \omega]_\v \le \norm{f}_0[\omega - \sew \omega]_\v<\infty$, hence  $f \cwedge \omega \approx_\v f \cwedge \sew \omega$.
\end{proof}

At this level of generality, we may define the  ``integral'' of  $\omega \in \germ^{k}(\tilde{\O})$ along parametrizations
$\varphi: I \to \O$ by pull-back, i.e.\ as $\sew \varphi^\natural \omega$ (if it is sewable).  
As in the smooth setting, we have invariance with respect to the choice of $\varphi$, provided that the ``reparametrization'' is sufficiently nice.

\begin{proposition}\label{prop:indep-par}
Let  $\omega \in \germ^k(\O)$, $\varphi: I\subseteq \R^{n} \to \O$, $\psi: J \subseteq \R^{m}\to \O$, $\varrho: I \to J$ with $\varphi = \psi \circ \varrho$. Let $\psi^\natural \omega$ be $\w$-sewable, and $\varrho^\natural \w \le \v$, for some  $k$-Dini gauge $\v \in \germ^k(I)$.
Then, $\varphi^\natural \omega$ is $\v$-sewable, if and only if $\varrho^\natural\sew\bra{\psi ^\prime \omega}$ is $\v$-sewable, and in such a case one has
\[\sew \bra{ \varrho^\natural \sew \bra{ \psi^\natural \omega}} = \sew \bra{ \varphi^\natural \omega }.\]

\end{proposition}

\begin{proof}
We have
\[ \begin{split} [ \varrho^\natural \sew( \psi ^\prime \omega) - \varphi^\natural \omega ]_{\v} &  = [ \varrho^\natural\bra{  \sew \psi^\natural \omega - \psi^\natural \omega} ]_{\v} \quad \text{ since $\varphi^\natural = \varrho^\natural \psi^\natural$,}\\
& \le  [\varrho^\natural\bra{ \sew \psi^\natural \omega - \psi^\natural \omega} ]_{\varrho^\natural \w}  \quad \text{since $\varrho^\natural \w \le \v$,} \\
 & \le  [ \sew \psi^\natural \omega - \psi^\natural \omega ]_{\w} \quad \text{by Remark~\ref{rem:moduli-pullback},}\\
 & < \infty.\end{split}\]
Therefore, $\varrho^\natural \sew ( \psi ^\prime \omega) \approx_\v \varphi^\natural \omega$ and the thesis follows.
\end{proof}

 One could ask if for a sewable $(k-1)$-germ $\eta$ one has that  $\delta \eta$ is sewable with $\sew \delta \eta = \delta \sew \eta$, the problem being that one can easily prove $\delta \eta \approx_{\v} \delta \sew \eta$, but only for a $(k-1)$-Dini gauge $\v \in \germ^{k}(\O)$. However, the answer is positive if $\delta \eta$ is sewable. 
 
\begin{theorem}[Stokes-Cartan]\label{thm:stokes}
Let $k\in \cur{1,2}$, $\eta \in \germ^{k-1}(\O)$ be sewable with $\delta \eta$ sewable. Then $\delta \sew \eta = \sew \delta \eta$. In particular, if $\eta$ is regular (so that $\d\eta:=\delta\eta$ is regular by Proposition~\ref{prop_reg1} and Remark~\ref{rem_dom_reg1}), then
\[ \int_{\partial S} \eta = \int_S \,\d\eta \quad \text{for every $S \in \simp^k(\O)$.}\]
\end{theorem}

Sufficient conditions ensuring the existence of $\sew \omega$ are provided by the following result. 

\begin{theorem}[sewing, existence]\label{thm:sew-existence}
Let $k \in \cur{1,2}$. For any continuous strong $k$-Dini gauge $\u \in \germ^{k+1}(\O)$ such that  the germ $\tilde{\u}$ in~\eqref{eq:v-dini} is continuous, there exists a continuous $k$-Dini gauge $\v \in \germ^{k}(\O)$ such that the following holds: if $\omega \in \germ^k(\O)$ is continuous, nonatomic and $ \delta \omega \approx_\u 0$, then $\omega$ is $\v$-sewable, with $\sew \omega$ continuous and
\begin{equation}\label{eq:sew-thesis} [\omega - \sew \omega  ]_{\v}\le [\delta \omega]_\u.\end{equation}
%Moreover, if $\omega$ is continuous, then $\sew \omega$ is continuous.
\end{theorem}

\begin{remark}
If $k=1$, then the condition $\delta\omega \approx_\u 0$ already implies that $\omega$ is nonatomic, because $|\omega_{pp}| = |(\delta\omega)_{ppp}| \le [\delta \omega]_\u \u([ppp])= 0$, for $p \in D$.
\end{remark}

\begin{remark}\label{rem:continuous-sewing-lemma}
It is worth noting that this is the only theorem that requires continuity of the gauges $\u$ and $\tilde{\u}$.
An inspection of its proof  shows that the continuity conditions of $\u$, $\tilde{\u}$ and $\omega$ can be replaced by continuity of their restrictions to every $k$-plane; in this case however also $\v$ and $\sew \omega$ are only guaranteed to be continuous on $k$-planes. Similarly, the condition $\delta \omega \approx_\u 0$ may be verified only over each $k$-plane, as long as it is satisfied uniformly over $k$-planes, then $[\delta \omega]_\u$ in~\eqref{eq:sew-thesis} has to be replaced with   its least upper bound over all $k$-planes.
\end{remark}

\begin{remark}\label{rem:explicit}
The proof of Theorem~\ref{thm:sew-existence} and Remark~\ref{rem_Dini_gauge1} (v) give that, if $\u$ is strong $r$-Dini (with $r \ge k$), then $\v$ is in fact $r$-Dini. For $k=1$, letting $\u := \diam^{\gamma}\in \germ^{2}(\O)$ (with $\gamma >1$) the proof of Theorem~\ref{thm:sew-existence}  gives $\v \le \c(\gamma)\diam^{\gamma}$. For $k=2$, letting $\u := \diam^{\gamma_1} \vol_2^{\gamma_2} \in \germ^{k+1}(\O)$ (with $\gamma_1 + 2 \gamma_2 >k$) the proof gives $\v \le \c(\gamma_1, \gamma_2)\diam^{\gamma_1} \vol_2^{\gamma_2}$.
\end{remark}

We end this section with the following result on continuity of the sewing map.

\begin{theorem}[sewing, continuity]\label{thm:sew-continuity}
Let $k\in\cur{1,2}$, $\v \in \germ^k(D)$ be $k$-Dini,  $(\omega^j)_{j \ge 1} \subseteq \germ^k(\O)$ be sewable with %nonatomic with  $ \delta \omega^n \approx_\u 0$,
\[ \omega^j \to \omega \in \germ^k(\O) \quad \text{pointwise} \quad \text{and} \quad \sup_{n \ge 1} [\omega^j - \sew \omega^j]_\v < \infty.\]
Then  $\omega$ is $\v$-sewable and $\sew \omega^j \to \sew \omega$ pointwise.
\end{theorem}

 In particular, in view of Theorem~\ref{thm:sew-existence}, the following result holds:

 \begin{corollary}\label{cor:sew-stability}
Let $k\in\cur{1,2}$, $\u \in \germ^{k+1}(D)$ be $k$-Dini, continuous with $\tilde{\u}$ in~\eqref{eq:v-dini} continuous. Let $(\omega^j)_{j \ge 1} \subseteq \germ^k(\O)$ be continuous, nonatomic with % with the above result holds  if $\u \in \germ^{k+1}(D)$ is strong $k$-Dini and $(\omega^j)_{j \ge 1}$ are nonatomic with
\[  \omega^j \to \omega \in \germ^k(\O) \quad \text{pointwise} \quad \text{and} \quad \sup_{j \ge 1} [\delta \omega]_{\u} < \infty.\]
Then $(\omega^j)_{j \ge 1}$ and $\omega$ are sewable and $\sew \omega^j \to \sew
 \omega$ pointwise.
\end{corollary}
%\begin{proof}
%\[
%
%Letting $\c := \sup_{n \ge 1} [\delta \omega^n]_\u$, then every $n \ge 1$ and $S \in \germ^k(D)$,
%\[ \ang{S, \omega^n-  \sew \omega^n } \le [\delta \omega^n]_\u\v(S) \le \c \v(S).\]
%\end{proof}

\section{Young and Z\"ust integrals}\label{sec:young-zust}

In this section, we show how both Young~\cite{young_inequality_1936} and  Z\"ust~\cite{zust_integration_2011} integrals (in the case of $2$-germs) can be recovered in the framework introduced above. For simplicity of exposition, we restrict ourselves to Dini gauges of the form $\diam^\gamma$ in case of $1$-germs or $\diam^{\gamma_1}\vol^{\gamma_2}_2$ in case of $2$-germs.

Let us introduce the following notation (extending that of Example~\ref{ex:sew-leb}): for a $0$-germ $f$ and a $(k-1)$-germ $\eta$, if $f \cwedge \delta \eta$ is sewable, we write  $f \d \eta := \sew (f \cwedge \delta \eta)$. %, we have $f\d g = f \d^{\mathcal{L}} g$.

\subsection{Young integral}

% $\eta = g$ is a function and we obtain Young integrals. % It is well-known~\cite{} that the Young integral can be obtained as an application of Lemma~\ref{lem:sew}, for $d=1$, to $1$-germs of the form $f \cwedge \delta g$. In this section, we regularizable $1$-germs of such form. %More precisely, for $f$, $g \in \germ^0(\O)$, if $f \cwedge \delta g \in \germ^1(\O)$ is regularizable, we write $f \d g:= \reg (f \delta g)$. Theorem~\ref{thm:sew-unique} and Theorem~\ref{thm:sew-existence} give respectively the following results.

%
%\begin{theorem}\label{lem:uniqueness-young}
%Let  $f, g \in \germ^0(\O)$ with $f \cwedge \delta g$ $\v$-regularizable, and $\omega \in \germ^1(\O)$ such that
%%\[ \sum_{i=1}^n [\delta f_i \cwedge \delta g_i]_\gamma< \infty\]% + [\delta f_i' \cwedge \delta g_i']_\gamma< \infty\]
%%and
% $\omega \approx_\v f \cwedge \delta g$. Then $\omega$ is $\v$-regularizable and $\reg \omega = f \d g$.
%\end{theorem}
%
%\begin{remark}\label{rem:iteration-young}
%Given $f^1, f^2, g \in \germ^0(\O)$ such that both $f^1 \cwedge \delta g$ and $(f^1 f^2)\cwedge \delta g$ are $\v$-regularizable, if $[f^2]_0<\infty$ then $f^2 \cwedge (f^1 \d g)$ is $\v$-regularizable and $\reg(f^2 \cwedge (f^1 \d g)) =(f^1 f^2) \d g$:
%\[ [f^2 \cwedge (f^1 \d g) - (f^1 f^2) \d g]_\v \le  [f^2 \cwedge \bra{ f^1 \cwedge  \delta g- f^1 \d g}]_\v  + [ (f^1f^2)\cwedge \delta g  - (f^1 f^2)\d g]_\v  < \infty.\]
%\end{remark}

As already noticed, every $1$-germ $\delta f$ is regular. Hence,  Theorem~\ref{thm:sew-existence} together with Remark~\ref{rem:explicit} yields the following result. %, which is an extension (only because the domain $\O$ can be of~\cite{gubinelli, friz-hairer} in our language.

\begin{theorem}[Young integral]\label{thm:young}
Let  $\gamma >1$, $f$, $g \in \germ^0(\O)$ be continuous with $\sqa{\delta f \cwedge \delta g}_{\gamma} < \infty$. %Then, be continuous such that \in C^{\alpha}(\O)$, $g \in C^{\beta}(\O)$.
Then,  $f \cwedge \delta g$ is $\gamma$-sewable with
\begin{equation}\label{eq:bound-young-regularized-germ} \sqa{f\cwedge \delta g-  f \d g}_{\gamma} \le \c \sqa{\delta f \cwedge \delta g}_{\gamma},\end{equation}% \quad \text{and} \quad \sqa{\delta( f \d g)}_{(2-\gamma) \oplus 2 (\gamma-1)} \le \c[\delta f \cwedge \delta g]_{\gamma},\end{equation}
where $\c = \c(\gamma)$.
%there exists a unique continuous, closed and alternating $1$-germ $f \d g$ such that
%\begin{equation}\label{eq:young-integral}  \abs{ \int_{[s_0 s_1]} f\d g - \ang{[s_0 s_1],  f \cwedge \delta g} } \le  \c \sqa{\delta f\cwedge \delta g}_{\gamma} \abs{s_1 - s_0}^{\gamma}, \, \text{for every $[s_0 s_1] \in \simp^1(\O)$,}\end{equation}
%where $\c = \c(\gamma)$. In particular, this holds if $f \in C^{\alpha}(\O)$, $g \in C^\beta(\O)$, with $\alpha+\beta = \gamma>1$.% Moreover, if $\omega$ is a continuous closed $1$-germ such that
%\[  \abs{ \ang{[s_0 s_1], \omega - f \cwedge \delta g}  } \le o \bra{ \abs{s_0-s_1}}, \quad \text{for every $[s_0 s_1] \in \simp^1(\O)$}\]
%then $\omega = f \d g$.
\end{theorem}

In particular, the result applies if $\sqa{\delta f}_\alpha + \sqa{\delta g}_\beta <\infty$ and $\gamma:=\alpha+\beta>1$. The triangle inequality in~\eqref{eq:bound-young-regularized-germ}  gives the bound, for any $\beta\le \gamma$,
\[ \sqa{f \d g}_{\beta} \le \norm{f}_0 \sqa{\delta g}_\beta + \c \, \diam(\O)^{\gamma-\beta} \sqa{\delta f \cwedge \delta g}_{\gamma}.\]

\begin{remark}[continuity of Young integral]\label{rem:continuity-young}
By Corollary~\ref{cor:sew-stability}, we have that, if  $\gamma >1$, $(f^j)_{j \ge 1}$, $(g^j)_{j \ge 1} \subseteq  \germ^0(\O)$ are continuous with $\sup_{j \ge 1} \sqa{\delta f^j \cwedge \delta g^j}_{\gamma} < \infty$ and $f^j \to f \in \germ^0$, $g^j \to g \in \germ^0$ pointwise, then $f \cup \delta g$ is $\gamma$-sewable and $ f^j \d g^j \to f \d g$ pointwise. Since $f \d g$, when $\sqa{\delta g}_1<\infty$, is given by classical Lebesgue integrals (Example~\ref{ex:leb-germ}), this would allow to transfer by approximation many results from classical differential forms to germs of Young type, e.g., the chain rule Proposition~\ref{prop:chain-rule}. However, we prefer to give alternative proofs allowing for self-contained exposition of the theory.%This in particular shows that %Then, becontinuous such that \in C^{\alpha}(\O)$, $g \in C^{\beta}(\O)$.
%  $f \cwedge \delta g$ is $\gamma$-sewable with
%\[ f^j \d g^j \to f \d g\]
%\begin{equation}\label{eq:bound-young-regularized-germ} \sqa{f\cwedge \delta g-  f \d g}_{\gamma} \le \c \sqa{\delta f \cwedge \delta g}_{\gamma},\end{equation}% \quad \text{and} \quad \sqa{\delta( f \d g)}_{(2-\gamma) \oplus 2 (\gamma-1)} \le \c[\delta f \cwedge \delta g]_{\gamma},\end{equation}
%where $\c = \c(\gamma)$.
\end{remark}

 \begin{remark}[locality]\label{cor:loc-young}
Under the assumptions of Theorem~\ref{thm:young}, if $g$ is constant in $\O$, by uniqueness of $\sew$ we obtain $f \d g = 0$. By Corollary~\ref{cor:locality}, we deduce that $f \d g$ is local, i.e., if $g$ is constant on $\conv\bra{ [p_0 p_1]} \subseteq \O$, then $\int_{[p_0 p_1]} f \d g  =0$.
\end{remark}

\begin{corollary}\label{coro:young-iterated}
Let $\gamma>1$,  $f$, $g$, $h \in \germ^0(\O)$ be continuous with
\[ \norm{f}_0 + \norm{h}_0 <\infty \quad \text{and} \quad \sqa{\delta f \cwedge\delta g}_{\gamma}  + \sqa{\delta h \cwedge \delta g}_{\gamma} <\infty.\]
 Then $f\cwedge (h \d g)$,  $h \cwedge (f\d g)$ and $(hf) \cwedge \delta g$ are $\gamma$-sewable with
 \[ (fh) \d g = \sew( (fh) \cup \delta g) = \sew\bra{f \cwedge (h \d g )} = \sew \bra{h \cwedge (f \d g)}.\]
\end{corollary}

\begin{proof}
One has
\[ \sqa{\delta \bra{ (hf) \cwedge \delta g}}_{\gamma} \le  \norm{f}_0 \sqa{\delta h \cwedge \delta g}_\gamma + \norm{h}_0 \sqa{\delta f \cwedge \delta g}_{\gamma} <\infty\]
hence $(hf)\cwedge \delta g$ is $\gamma$-sewable and the thesis follows from Corollary~\ref{coro:close-to-sewable-2}.
\end{proof}

%As applications of Proposition~\ref{prop:pull-back-reg} and Proposition~\ref{prop:chain-rule-gen}, we have the following results.

We also have the following chain rule (which could be also proved by approximation with smooth functions, using Remark~\ref{rem:continuity-young}).

\begin{proposition}[chain rule]\label{prop:chain-rule}
Let $\alpha$, $\beta$, $\sigma >0$ with $\min\cur{\alpha, \sigma \beta} + \beta>1$ and let $D\subseteq \R^d$ be bounded. If $f \in \germ^0(D)$, $g = (g^i)_{i=1}^\ell \in \germ^0(D; \R^\ell)$, $\Psi\colon \R^\ell \to \R$ are such that   $\sqa{ \delta f}_\alpha <\infty$, $\sqa{ \delta g^i}_{\beta}< \infty$,  and $\Psi$ is $C^{1, \sigma}$, i.e., the partial derivatives $\partial_i \Psi$ satisfy $\sqa{\delta (\partial_i \Psi)}_\sigma <\infty$ for every $i = 1, \ldots, \ell$, then
\begin{equation}
\label{eq:chain-rule-thesis}
 f \d (\Psi \circ g) = (f \nabla \Psi \circ g) \cdot \d g := \sum_{i=1}^\ell (f \partial_i \Psi\circ g )\d g^i. \end{equation}
\end{proposition}

\begin{proof}
Taylor formula for $\Psi$ reads in our language as
\[ \delta \Psi \approx_{1+\sigma} \sum_{i=1}^\ell \partial_i \Psi \cup \delta x^i,\]
where $(x^i)_{i=1}^\ell$ denote the coordinate functions on $\R^\ell$. By composition with $g$, using the assumption $\sqa{ \delta g^i}_{\beta}< \infty$ for $i  = 1, \ldots, \ell$, we deduce that
\[ \delta (\Psi \circ g)  \approx_{ \beta(1+\sigma)}  \sum_{i=1}^\ell \partial_i \Psi \circ g \cup \delta g^i.\]
Since $\delta(\Psi \circ g) = \d (\Psi \circ g)$ and $\sew \bra{ \partial_i \Psi \circ g \cup \delta g^i}= \partial_i \Psi \circ g\,  \d g^i$  by Theorem~\ref{thm:young}, we conclude that
\[ \d (\Psi \circ g) =  \sum_{i=1}^\ell \partial_i \Psi \circ g \d g^i.\]
Applying $f \cup $ to both sides  and using Corollary~\ref{coro:young-iterated} we obtain the thesis. %(we can argue on a bounded $\tilde{\O} \subseteq D$, hence assume that $f$ is bounded) we have
%\[ f \cup \delta (\Psi \circ g)  \approx_{\beta(1+\sigma)} \sum_{i=1}^\ell  f \cup  \partial_i \Psi \circ g \cup \delta g^i.\]
%By Theorem~\ref{thm:young} applied to $f \cup \delta (\Psi \circ g)$ and each $f \cup  \partial_i \Psi \circ g \cup \delta g^i$ we deduce that
%\[\begin{split} f \d (\Psi \circ g) & \approx_{\alpha+ \beta} f \cup \delta (\Psi \circ g) \approx_{\beta(1+\sigma)} \sum_{i=1}^\ell  f \cup  \partial_i \Psi \circ g \cup \delta g^i \\
%&  \approx_{\min\cur{\alpha, \sigma \beta} + \beta} \sum_{i=1}^\ell  (f  \partial_i \Psi \circ g)  \d g^i,
%\end{split} \]
%hence equality holds between the first and last terms (they are regular germs).
%By the vectorial version of Proposition~\ref{prop:pull-back-young}  with $g$ instead of $\varphi$ and $f_0$
\end{proof}

As usual, the chain rule implies Leibniz rule, for which we give a self-contained proof in a slightly more general assumption on $\sqa{\delta g^1 \cup \delta g^2}_\gamma$ instead of $\sqa{ \delta g^1}_{\beta} + \sqa{\delta g^2}_\beta$ as it would follow from Proposition~\ref{prop:chain-rule} (with $\beta = \gamma/2$).

\begin{corollary}[Leibniz or integration by parts rule]\label{cor:leibniz}
Let $g^1$, $g^2 \in \germ^0(\O)$ be continuous with  $\sqa{\delta g^1 \cwedge \delta g^2}_\gamma<\infty$ for some $\gamma>1$. Then $ \delta (g^1 g^2) =g^1 \d g^2 + g^2 \d g^1$, i.e.,
\[ \int_{[pq]} g^1 \d g^2 = \delta (g^1 g^2)_{pq}  - \int_{[pq]} g^1 \d g^2 \quad \text{for $[pq] \in \simp^1(D)$.}\]

\end{corollary}

In particular, if $g\in \germ^0(\O)$ is such that $\sqa{\delta g}_{\alpha} <\infty$ for  $\alpha>1/2$, then $\delta g^2 = 2 g \d g$.

\begin{proof}
We use the identity %valid for coordinate maps $(x^1,x^2)$ on $\R^2$,
%\[ \delta (x^1x^2)  = x^1 \cwedge  \delta x^2 + x^2 \cwedge \delta x^1 +(\delta x^1) (\delta x^2).\]
%By composition with $ g =(g^1, g^2)$, we have
\[ \delta (g^1 g^2) = g^1 \cwedge \delta g^2 + g^2 \cwedge \delta g^1 + (\delta g^1) (\delta g^2).\]
For $[pq] \in \simp^1(\O)$, we have
\[ \ang{[pq], (\delta g^1) (\delta g^2) } =  (\delta g^1)_{pq} (\delta g^2)_{pq} = -\ang{ [pqp],  \delta g^1  \cup  \delta g^2}\]
hence, by the assumption $\sqa{ \delta g^1 \cup \delta g^2}_{\gamma} < \infty$,
\[ \delta (g^1 g^2) \approx_{\gamma} g^1 \cwedge \delta g^2 + g^2 \cwedge \delta g^1.\]
We conclude then as in Proposition~\ref{prop:chain-rule}, applying $\sew$ to both sides of the above relationship.
\end{proof}

\subsection{Z\"ust integral}

To study the case of of integration of $2$-germs of the type $f \d g^1 \wedge \d  g^2$, for $f, g^1, g^2 \in \germ^0(\O)$, as introduced in  R.\ Z\"ust~\cite{zust_integration_2011}, the starting point is to notice that the naive germs $f \cup \delta g^1 \cup \delta g^2$ %such germs %\[ f \cwedge  \delta g^1 \cwedge \delta g^2,\]
in general fail to be nonatomic, already when $D \subseteq \R$. Our  construction (as well as Z\"ust's) is then based on the identity $\delta g^1 \cwedge \delta g^2 =  \delta \bra{ g^1 \cwedge \delta g^2 }$,
which suggests to apply  Theorem~\ref{thm:sew-existence} to the $2$-germ  $f \cwedge \delta \eta$, where $\eta =  g^1 \d g^2 $, i.e.,
\[ \eta_{[pq]} = \int_{[pq]}g ^1 \d g^2,\]
the integral being in the sense of Young as  in the previous section. More explicitly,
\[ \begin{split} \ang{[p_0 p_1 p_2], \delta (g^1 \d g^2)}&  = \ang{\partial [p_0 p_1 p_2], g^1 \d g^2} \\
& = \int_{ [p_1 p_2]} g^1 \d g^2 -  \int_{[p_0p_2]} g^1 \d g^2 + \int_{[p_0p_1]} g^1 \d g^2,\end{split}\]
and
\[ \ang{ [p_0 p_1 p_2], f \cup \delta (g^1 \d g^2)} = f_{p_0} \ang{[p_0 p_1 p_2], \delta g^1 \d g^2}.\]
 % Therefore, if $g^1 \cwedge \delta g^2$ is regularizable, and $f \cwedge \delta (g^1 \d g^2)$ is regularizable, we write $f \d g^1 \wedge \d g^2 := \reg (f \delta (g^1 \d g^2))$.

 \begin{example}[classical forms]\label{ex:classical-zust}
 When $g^1 =x^i$, $g^2 = x^j$ coordinate functions, in the notation of Example~\ref{ex:xidxj}  we have that $\delta (x^i \d x^j)$ equals $\delta \omega^{ij}$ in~\eqref{eq:signed-area}, hence one can argue as in Example~\ref{ex:sew-leb} and show that the $2$-germ $f \cup \delta (x^i \d x^j)$ is sewable if $f$ is (uniformly) continuous, with
 \[ \ang{S, \sew( f \cup \delta (x^i \d x^j)} =\int_{S} f \d x^i \wedge \d x^j,\]
 the integral being in the classical sense of differential forms. Indeed, letting $\eta_f$ stand for the modulus of continuity of $f$, we have
\[ \begin{split}    \abs{ \int_{[pqr]} f \d x^i \wedge \d x^j - f_p \cup \delta (x^i \d x^j)_{pqr}   }  & = \abs{ \int_{[pqr]} (f - f_p) \d x^i \wedge \d x^j } \\
& \le \eta_f(\diam[pqr])  \vol_2([pqr]) =: \v([pq])\end{split}\]
which defines a $2$-Dini gauge.
 \end{example}

To argue with more general functions than coordinates, we rely on the following bound for $\delta (g^1 \d g^2)$.

\begin{proposition}\label{prop:bound-delta-young}
Let  $\beta >1$, $g^1, g^2 \in \germ^0(\O)$ be continuous with $\sqa{\delta g^1 \cwedge \delta g^2}_\beta< \infty$. %Then, be continuous such that \in C^{\alpha}(\O)$, $g \in C^{\beta}(\O)$.
For  $\c = \c(\beta)$, one has then
%\[ [\delta (f \d g)]_{\diam^{2-\beta} \vol_2^{\beta-1}}  \le
\begin{equation}
\label{eq:bound-germ-pre-zust} \sqa{ \delta (g^1 \d g^2)}_\beta \le  \c \sqa{\delta g^1 \cwedge \delta g^2}_{\beta}.\end{equation}
\end{proposition}
%which is a useful bound in order to apply again Theorem~\ref{thm:sew-existence}.

\begin{proof}
%Being $\delta( f \d g)$ regular, in view of Lemma~\ref{lem:bound-improvement}, it is sufficient to argue that $\sqa{\delta (f \d g) }_\beta \le \c\sqa{\delta f \cwedge \delta g}_{\beta}$ for some $\c = \c(\beta)$. % the claimed bound following then from Theorem~\ref{regular-improve-bound}. To this aim, for any
For any $S\in \simp^2(\O)$, to prove that,  for some constant $\c =\c(\beta)$, one has
\[\abs{ \ang{ S, \delta( g^1 \d g^2) } }   \le \c \sqa{ \delta g^1 \cup \delta g^2}_\beta \diam^\beta(S),\]
we add and subtract $\delta( g^1 \cup \delta g^2)= \delta g^1 \cup \delta g^2$, hence
\[\begin{split} \abs{ \ang{  S, \delta( g^1 \d g^2)} } & \le  \abs{ \ang{ S, \delta( g^1 \d g^2 - g^1 \cup \delta g^2)}} + \abs{ \ang{ S, \delta g^1 \cup \delta g^2 } } \\
& \le \abs{ \ang{ \partial S,  g^1 \d g^2 - g^1 \cup \delta g^2}} +\sqa{ \delta g^1 \cup \delta g^2}_\beta \diam^\beta(S).\end{split} \]
%We bound from above the second term in the sum as
%\[   \abs{ \ang{  S, \delta( f \cup \delta g) } } \le \sqa{ \delta f \cup \delta g}_\beta \diam^\beta(S).\]
For the first term, since $\diam(S^i) \le \diam(S)$ for each of the three $1$-simplices in $\partial S = \sum_{i=0}^2 (-1)^i S^i$, by~\eqref{eq:bound-young-regularized-germ} we have
\[\begin{split}   \abs{ \ang{ \partial S, g^1 \d g^2 - g^1 \cup \delta g^2}} &  \le  \sum_{i=0}^2  \abs{ \ang{ S^i, g^1 \d g^2 - g^1 \cup \delta g^2}} \\
& \le \sum_{i=0}^2 \c  \sqa{ \delta g^1\cup \delta g^2}_\beta\diam^{\beta}(S^i)\\
& \le 3 \c \sqa{ \delta  g^1 \cup \delta g^2}_\beta\diam^{\beta}(S),\end{split}\]
hence the thesis.
\end{proof}

\begin{remark}[bound improvement]
By Lemma~\ref{lem:bound-improvement}, since $\delta( g^1 \d g^2)$ is regular, we can improve~\eqref{eq:bound-germ-pre-zust} to the inequality
\begin{equation}\label{eq:bound-improvemente-dg-dg} [ \delta (g^1 \d g^2) ]_{\diam^{2- \beta} \vol_2^{\beta-1}} \le \c(\beta) \sqa{ \delta g^1 \cup \delta g^2}_{\beta}.\end{equation}
\end{remark}

\begin{theorem}[Z\"ust integral]\label{thm:zust}
Let $\alpha >0$, $\beta >1$ with $\gamma:= \alpha+\beta>2$, let $f$, $g^1$, $g^2 \in \germ^0(D)$ be continuous with $\sqa{\delta f}_\alpha+\sqa{\delta g^1 \cup \delta g^2}_{\beta} < \infty$. Then, $f \cup \delta (g^1 \d g^2)$ is $\gamma$-sewable with
\begin{equation}\label{eq:bound-zust-regularized-germ} \sqa{f\cwedge \delta (g^1 \d g^2) - f \d( g^1 \d g^2) }_{\gamma} \le \c \sqa{\delta f}_{\alpha} \sqa{\delta g^1 \cup \delta g^2}_{\beta},
 \end{equation}
where $\c = \c(\alpha, \beta)$.
\end{theorem}

\begin{remark}[alternating property]\label{rem:alternating-zust}
The assumptions of Theorem~\ref{thm:zust} are symmetric in $g^1$, $g^2$, and by Corollary~\ref{cor:leibniz} we have $g^1 \d g^2 =  \delta (g^1 g^2) - g^2 \d g^1$, hence $\delta( g^1 \d g^2) = - \delta(g^2 \d g^1)$. By uniqueness Theorem~\ref{thm:sew-unique} we obtain that
\[ f \d (g^1 \d g^2) = - f \d (g^2 \d g^1),\]
which justifies the introduction of the usual wedge notation % alternating property holds:
\[ f \d g^1 \wedge \d g^2 :=  f \d ( g^1 \d g^2).\]
\end{remark}

 \begin{remark}[locality]\label{cor:loc-zust}
Under the assumptions of Theorem~\ref{thm:zust}, if $g^1$ or $g^2$ is constant in $\O$, by uniqueness of $\sew$ we obtain $f \d g^1 \wedge \d g^2 = 0$. By Corollary~\ref{cor:locality}, we deduce that $f \d g^1 \wedge \d g^2$ is local, i.e., if $g^1$ or $g^2$ is constant on the geometric simplex $\conv\bra{ S} \subseteq \O$, then $\int_{S} f \d g^1 \wedge \d g ^2  =0$.
\end{remark}

\begin{remark}[continuity of Z\"ust integral]\label{rem:continuity-zust}
As for Young type germs, by Corollary~\ref{cor:sew-stability}, we have that, if  $\alpha>0$, $\beta>1$ with $\alpha+\beta>2$, and $(f^j)_{j \ge 1}$, $(g^j)_{j \ge 1}$, $(\tilde g^j)_{j \ge 1}\subseteq  \germ^0(\O)$ are continuous with $\sup_{j \ge 1} \sqa{\delta f^j}_{\alpha} + \sqa{ \delta g^j \cwedge \delta \tilde g^j}_{\beta} < \infty$ and $f^j \to f \in \germ^0$, $g^j \to g \in \germ^0$, $\tilde{ g}^j \to \tilde{g} \in \germ^0(\O)$ pointwise, then $f^j \d g^j \wedge \d \tilde{g}^j \to f \d g \wedge \d \tilde g$. Again, this would allow us to obtain calculus results by approximation from classical differential forms.
\end{remark}

\begin{remark}[bound improvement for integrals]
From~\eqref{eq:bound-improvemente-dg-dg} we have that
\[ [ \delta f \cup  \delta (g^1 \d g^2) ]_{\diam^{\alpha+2 - \beta} \vol_2^{\beta-1}} \le \sqa{ \delta f}_\alpha [\delta (g^1 \d g^2) ]_{\diam^{2 - \beta} \vol_2^{\beta-1}} < \infty. \]%\le \c(\beta) \sqa{ \delta f}_{\alpha}\sqa{ \delta g^1 \cup \delta g^2}_\beta.\]
By Theorem~\ref{thm:sew-existence} applied with the strong $2$-Dini gauge $\diam^{\alpha+2 -\beta} \vol_2^{\beta-2}$ (under the assumption $\alpha+ \beta>2$) and Remark~\ref{rem:explicit}, we improve inequality~\eqref{eq:bound-zust-regularized-germ} to
\[\sqa{f\cwedge \delta (g^1 \d g^2) - f \d g^1 \wedge \d g^2}_{\diam^{\alpha+2 - \beta} \vol_2^{\beta-1}} \le \c \sqa{\delta f}_{\alpha} \sqa{\delta g^1 \cup \delta g^2}_{\beta},\]
from which, by the triangle inequality and~\eqref{eq:bound-improvemente-dg-dg} again,  we deduce that
\begin{equation}\label{eq:bound-size-zust} [f \d g^1 \wedge \d g^2]_{\diam^{2 - \beta} \vol_2^{\beta-1}} \le \c\bra{ \norm{f}_0 + \diam(D)^{\alpha} \sqa{ \delta f}_\alpha} \sqa{ \delta g^1 \cup \delta g^2}_\beta.\end{equation}
% than $\diam^{\gamma}$.
This is analogous to~\cite[Corollary 3.4]{zust_integration_2011} (for $k=2$) with triangles instead of rectangles.
\end{remark}

The following result is analogous to Corollary~\ref{coro:young-iterated}.

\begin{corollary}\label{cor:zust-iterated}
Let $\alpha >0$, $\beta >1$ with $\gamma:= \alpha+\beta>2$, let $f$, $h$, $g^1$, $g^2 \in \germ^0(D)$ be continuous with
\[ \norm{f}_0 + \norm{h}_0 < \infty \quad \text{and} \quad  \sqa{\delta f}_\alpha+\sqa{\delta h}_\alpha+\sqa{\delta g^1 \cup \delta g^2}_{\beta} < \infty.\] Then $f \cup(h \d g^1 \wedge \d g^2)$, $h \cup (f \d g^1 \wedge \d g^2)$ and $(fh) \cup ( \d g^1 \wedge \d g^2)$ are $\gamma$-sewable with
\[ \begin{split} (fh) \d g ^1\wedge \d g^2 &=  \sew ( (fh) \cup ( \d g^1 \wedge \d g^2)) \\
& = \sew( f \cup(h \d g^1 \wedge \d g^2)) = \sew (h \cup (f \d g^1 \wedge \d g^2)).\end{split}\]
\end{corollary}

The chain rule also holds: in order to prove it we first establish the following Leibniz-type expansion.

\begin{lemma}\label{lem:leibniz-zust}
Let $\alpha_1$, $\alpha_2$, $\beta>0$ with $\gamma:= \alpha_1+\alpha_2+\beta>2$. Let $f^1$, $f^2$, $g \in \germ^0 (\O)$ with $\sqa{\delta f^1}_{\alpha_1} + \sqa{\delta f^2}_{\alpha_2} + \sqa{\delta g}_\beta< \infty$. Then
\[ \delta (f^ 1f^2 \d g) \approx_\gamma f^1 \cup \delta (f^2 \d g) + f^2 \cup \delta (f^1 \d g).\]
\end{lemma}

\begin{proof}
Write $\omega:= \delta (f^1f^2 \d g) - f^1 \cup \delta (f^2 \d g)- f^2 \cup \delta (f^1 \d g) \in \germ^1(\O)$.  For any given $S = [p_0 p_1 p_2]$, we prove that, for some $\c = \c(\alpha_1, \alpha_2, \beta)$,
\begin{equation}\label{eq:leibniz-zust} \abs{ \ang{S, \omega}} \le \c \sqa{\delta f^1}_{\alpha_1}\sqa{\delta f^2}_{\alpha_2} \sqa{\delta g}_{\beta}  \diam^{\gamma}(S).  \end{equation}
Without loss of generality, we  assume that $\O = S$ (since the constant $\c$ does not depend on $\O$ either). By linearity of Young integral, we have the identity
\[ (f^1f^2 \d g) - f^1_{p_0}  (f^2 \d g) + f^2_{p_0} (f^1 \d g) = (f^1 - f^1_{p_0})(f^2-f^2_{p_0}) \d g -  f^1_{p_0} f^2_{p_0}  \delta  g,\]
hence (with $p_0$  fixed)
\[ \begin{split} \delta (f^1f^2 \d g) - f^1_{p_0} \delta (f^2 \d g) + f^2_{p_0}\delta (f^1 \d g) & = \delta \bra{ (f^1 - f^1_{p_0})(f^2-f^2_{p_0}) \d g }-  f^1_{p_0} f^2_{p_0} \delta^2 g\\
& = \delta \bra{ (f^1 - f^1_{p_0})(f^2-f^2_{p_0}) \d g},\end{split}\]
and evaluating at $S = [p_0 p_1 p_2]$ we obtain the identity
\[ \ang{S,\omega} = \ang{S, \delta \bra{ (f^1 - f^1_{p_0})(f^2-f^2_{p_0}) \d g}}.\]
The thesis follows by Proposition \ref{prop:bound-delta-young} applied with $\tilde \gamma := \min\cur{\alpha_1, \alpha_2} + \beta$ instead of $\gamma$, because
\begin{equation}\label{eq:leibniz-zust-2} \abs{  \ang{S, \delta \bra{ (f^1 - f^1_{p_0})(f^2-f^2_{p_0}) \d g}} } \le \c \sqa{ \delta\bra{ (f^1 - f^1_{p_0})(f^2-f^2_{p_0}) }\cup \delta g}_{\tilde \gamma} \diam(S)^{\tilde \gamma}\end{equation}
and
\[ \begin{split} \| \delta\bra{ (f^1 - f^1_{p_0})(f^2-f^2_{p_0})} & \cup \delta g \|_{\tilde \gamma}  \le  \sqa{ \delta f^1}_{\min\cur{\alpha_1, \alpha_2}} \norm{f^2-f^2_{p_0} }_0 \sqa{\delta g}_{\beta}\\
&\qquad \qquad + \norm{f^1-f^1_{p_0} }_0\sqa{ \delta f^2}_{\min\cur{\alpha_1, \alpha_2}}  \sqa{\delta g}_{\beta}\\
&  \le 2 \sqa{\delta f^1}_{\alpha_1} \sqa{\delta f^2}_{\alpha_2} \sqa{ \delta g}_{\beta} \diam(S)^{\alpha_1 + \alpha_2 - \min \cur{\alpha_1, \alpha_2}},  \\
& =  2 \sqa{\delta f^1}_{\alpha_1} \sqa{\delta f^2}_{\alpha_2} \sqa{ \delta g}_{\beta} \diam(S)^{\gamma - \tilde{ \gamma}}\end{split}\]
hence~\eqref{eq:leibniz-zust} follows from~\eqref{eq:leibniz-zust-2}.
\end{proof}

\begin{proposition}[chain rule]\label{prop:chain-rule-zust}
Let $\alpha>0$, $\beta>0$, $\sigma \in (0,1]$ with $\min\cur{\alpha, \sigma \beta} + 2\beta>2$. If $f \in \germ^0(D)$, $g = (g^i)_{i=1}^\ell \in \germ^0(D; \R^\ell)$, $\Psi = (\Psi^1, \Psi^2) \colon \R^\ell \to \R^2$ are such that $\sqa{\delta f}_\alpha <\infty$, $\sqa{ \delta g^i}_{\beta}< \infty$,  and $\Psi$ is $C^{1, \sigma}$, i.e., the partial derivatives $\partial_i \Psi^m$ satisfy $[\delta (\partial_i \Psi^m)]_\sigma <\infty$ for every $i \in \cur{1, \ldots, \ell}$, $m \in \cur{1,2}$, then
\[ f \d (\Psi^1 \circ g) \wedge \d (\Psi^2 \circ g) = \sum_{\substack{i,j=1, \ldots,\ell \\ i<j}} ( f \, J^{ij}_\Psi \circ g ) \d g^i \wedge \d g ^j, \quad \text{where} \]
\[ \quad J^{ij}_\Psi = \det\bra{ \begin{array}{cc}
\partial_i \Psi^1  &  \partial_j \Psi^1  \\ \partial_i \Psi^2 & \partial_j \Psi^2
\end{array}}.\]
\end{proposition}

\begin{proof}
Letting $\gamma:=\beta(\sigma +2)$, we prove that
\begin{equation}\label{eq:thesis-chain-rule-zust} \delta \bra{ (\Psi^1 \circ g)  \d (\Psi^2 \circ g) } \approx_{\gamma}  \sum_{\substack{i,j=1, \ldots,\ell \\ i<j}} ( f \, J^{ij}_\Psi \circ g ) \cup \delta ( g^i \d g ^j)\end{equation}
holds, hence equality must hold if we compose with $\sew$ both sides (the case with $f$ follows from Corollary~\ref{cor:zust-iterated}).

To prove~\eqref{eq:thesis-chain-rule-zust} notice first that  Proposition~\ref{prop:chain-rule} gives the identity
\begin{equation}\label{eq:chain-rule-first-identity-zust}  \delta (\Psi^2 \circ g) = \sum_{i=1}^\ell  ( \partial_i \Psi^2 \circ g) \d g^i\end{equation}
as well as (with $\Psi^1\circ g$ instead of $f$)%. By Corollary~\ref{coro:young-iterated} we obtain that
\[ (\Psi^1\circ g) \d (\Psi^2 \circ g)  = \sum_{i=1}^\ell  (\Psi^1 \circ g \, \partial_i \Psi^2 \circ g) \d g^i.\]

Therefore,
\begin{equation}\label{eq:chain-rule-zust-proof-2} \begin{split} \delta ( \Psi^1 \circ g \d \Psi^2 \circ g) & =  \sum_{i=1}^\ell\delta\bra{ ( \Psi^1 \circ g \partial_i \Psi^2 \circ g) \d g^i} \\
& \approx_{\gamma} \sum_{i=1}^\ell \Psi^1 \circ g  \cup \delta ( \partial_i \Psi^2 \circ g \d g^i) +  \partial_i \Psi^2 \circ g   \cup \delta (\Psi^1 \circ g  \d g^i) \\
& \quad \quad \text{by Lemma~\ref{lem:leibniz-zust},}\\
& \approx_{\gamma}    \sum_{i=1}^\ell \partial_i \Psi^2 \circ g   \cup \delta ( (\Psi^1 \circ g)  \d g^i),  \end{split}\end{equation}
the sum of the first  terms vanishing since, by~\eqref{eq:chain-rule-first-identity-zust},
\[ \sum_{i=1}^\ell   \Psi^1 \circ g  \cup \delta ( (\partial_i \Psi^2 \circ g) \d g^i)  = \Psi^1 \circ g  \cup \delta ( \sum_{i=1}^\ell \partial_i \Psi^2 \circ g \d g^i)  =  \Psi^1 \circ g  \cup \delta^2  (\Psi^2 \circ g )= 0.\]
Integrating by parts, for every $i=1, \ldots, \ell$,
\[ \delta (\Psi^1 \circ g  \d g^i) = - \delta^2 ( g^i \Psi^1 \circ g ) - \delta( g^i \d \Psi^1 \circ g) =  - \delta( g^i \d \Psi^1 \circ g).\]
Using again  Lemma~\ref{lem:leibniz-zust}, in a similar way as in~\eqref{eq:chain-rule-zust-proof-2}, with $\Psi^2$ instead of $\Psi^1$ and $g^i$ instead of $\Psi^1 \circ g$, %$\d \Psi^1 \circ g = \sum_{j= 1}^\ell \partial_j \Psi^1 \circ g \d g^j$ instead of
we obtain
\[ \delta( g^i \d \Psi^1 \circ g) \approx_\gamma \sum_{j=1}^\ell \partial_j \Psi^1 \circ g \cup \delta (g^i \d g^j).\]
Using this relationship in each summand in the last line of~\eqref{eq:chain-rule-zust-proof-2}, we have %  ubstituing Collecting all the terms, we have proved that
\[ \delta ( \Psi^1 \circ g \d \Psi^2 \circ g) \approx_\gamma -\sum_{i,j=1}^\ell \bra{\partial_i \Psi^2 \circ g \, \partial_j \Psi^1 \circ g} \cup \delta (g^i \d g^j),\]
which is equivalent to~\eqref{eq:thesis-chain-rule-zust} by Remark~\ref{rem:alternating-zust}.
\end{proof}

\begin{remark}[classical chain rule]\label{rem:classical-chain-rule}
We notice that same arguments may be used to prove the classical chain rule for $2$-germs of the type $f \d x^i \wedge \d x^j$ and maps $\Psi$ that are differentiable with (uniformly) continuous derivatives, introducing the $2$-Dini gauge $S \mapsto \eta(\diam(S)) \vol_2(S)$,  $\eta$ standing for the modulus of continuity of $\nabla \Psi$.
\end{remark}

\begin{remark}
The requirement $\Psi \in C^{1, \sigma}$ enters  in the proof only to write~\eqref{eq:chain-rule-first-identity-zust} (and the analogous identity for $\Psi^1$). Therefore, substituting $\Psi \circ g$ with $h = (h^1, h^2) \in \germ^0(\O; \R^2)$ such that, for $m \in \cur{1,2}$,
\begin{equation}\label{eq:g-diff} \delta h^m \approx_{\beta(1+\sigma)} \sum_{i=1}^\ell (\partial_{g^i} h^m) \cup \delta g^i\end{equation}
where $\partial_{g^i} h^m \in \germ^0(\O; \R^2)$ are such that $\sqa{ \delta ( \partial_{g^i} h^m) }_{\sigma \beta} < \infty$, we have that
\[ f \d h^1 \wedge \d h^2 = \sum_{\substack{i,j=1, \ldots,\ell \\ i<j}}  f \, (J_gh)^{ij} \d g^i \wedge \d g ^j, \quad \text{where} \]
\[ \quad (J_g h)^{ij} = \det\bra{ \begin{array}{cc}
\partial_{g^i} h^1  &  \partial_{g^j} h^1  \\ \partial_{g^i} h^2 & \partial_{g^j} h^2
\end{array}}.\]
The class of functions $h$ satisfying~\eqref{eq:g-diff} generalize that of {controlled paths} in~\cite{gubinelli_controlling_2004} (corresponding to the case $\O \subseteq \R$) and it is tempting to call them $g$\emph{-differentiable}, a concept that will be further investigated in~\cite{frobenius-zust}. Note that in particular when $\ell = 1$ (i.e., $g^1:=g$) and $h^1$, $h^2$ are both $g$-differentiable, then $f \d h^1 \wedge \d h^2 = 0$.
\end{remark}

%\end{document}

\subsection{Pull-back}

 In this section, we  study how germs of Young and Z\"ust type introduced above behave via the pull-back of a map $\varphi: I \subseteq \R^{k} \to D$. More precisely, we compare the ``algebraic'' and ``differential'' pull-backs, in Young's case (i.e., $f \d g$) respectively  given by
\begin{equation}\label{eq:pull-back-young} \varphi^\natural( f \d g ) \quad \text{and} \quad \varphi^*(f \d g ) := (f \circ \varphi) \d (g \circ \varphi),\end{equation}
and similarly in  Z\"ust's case (i.e., $f \d g ^1 \wedge \d g^2$) by
\begin{equation}\label{eq:pull-back-zust} \varphi^\natural( f \d g^1 \wedge \d g^2) \quad \text{and}  \quad \varphi^*(f \d g^1 \wedge \d g^2 ) := (f \circ \varphi) \d (g^1 \circ \varphi) \wedge \d (g^2 \circ \varphi),\end{equation}
provided these are defined, e.g., if the assumptions of Theorem~\ref{thm:young} (respectively  Theorem~\ref{thm:zust}) are met.

In the case of $1$-germs of Young type, the two pull-backs in~\eqref{eq:pull-back-young} are related to each other by a sewing procedure, as the following result shows.

\begin{proposition}[integration over curves via pull-back]\label{prop:pull-back-young}
Let  $\gamma>1$, $\sigma>1/\gamma$, $f, g \in \germ^0(\O)$ be continuous with $\sqa{\delta f \cwedge \delta g}_\gamma < \infty$ and let $\varphi \colon I \subseteq \R \to \O$ with $\sqa{\delta \varphi}_\sigma < \infty$. Then $\varphi^\natural(f \d g)$ is $\sigma \gamma$-sewable with%$\varphi^\natural( f \cup \delta g)$ is $\sigma\gamma$-sewable with
\[ \sew \bra{ \varphi^\natural( f \d g) } = \sew( f \circ \varphi \cup \delta g\circ \varphi) = \varphi^*( f \d g)\] %(f\circ \varphi) \d (g\circ \varphi).\]
\end{proposition}

%We may introduce the notation $\int_{ \varphi( [st]) } f \d g = \int_{[st]} (f \circ \varphi) \d ( g \circ \varphi)$ (note that we use the curve $\varphi([st])$ and not simply $\varphi_{\natural}[st] = [\varphi_s \varphi_t]$).

 The above proposition shows that when the latter regular germ is defined (using the conditions on $\gamma$ and $\sigma$) this is a limit both of Riemann sums of increments $f_{p^i} \delta g_{p^ip^{i+1}}$ computed along the curve $\varphi(I)$ and of ``integrated'' increments $\int_{[p^i p^{i+1}]} f \d g$, even in cases where $\varphi$ is not a rectifiable curve. In particular, this shows that the curvilinear integral $\int_{\varphi(I)} f \d g := \int_I \varphi^* (f \d g)$ remains invariant by reparametrization of $\varphi$.

\begin{proof}
The inequality
\[  \sqa{\delta (f\circ \varphi) \cwedge  \delta (g\circ \varphi)}_{\sigma\gamma } \le \sqa{\delta \varphi}_\sigma \sqa{\delta f \cwedge \delta g}_\gamma < \infty\]
implies that $ (f\circ \varphi) \cwedge \delta (g\circ \varphi)$ is $\sigma\gamma$-sewable by Theorem~\ref{thm:young}. Since $\varphi^\natural\diam ^\gamma \le \sqa{\delta \varphi}_{\sigma}^\gamma \diam^{\sigma \gamma}$, we have
\[ \sqa{ \varphi^\natural (f \d g) - \varphi^\natural (f \cwedge \delta g ) } _{\sigma \gamma} \le \c [ \varphi^\natural(  f \d g - f \cwedge \delta g) ] _{\varphi^\natural\diam^\gamma} \le \c \sqa{ f \d g - f \cwedge \delta g } _{\gamma} < \infty.\]
by Remark~\ref{rem:moduli-pullback}. Therefore,
\[  \varphi^\natural (f \d g) \approx_{\sigma \gamma} \varphi^\natural (f \cwedge \delta g ) = f\circ \varphi \cwedge \delta (g\circ \varphi) \approx_{\sigma \gamma} \sew ( (f \circ \varphi) \cup \delta (g \circ \varphi)) =  (f \circ \varphi) \d (g \circ \varphi)\]
and the thesis follows.
\end{proof}

In case of $2$-germs $f \d g^1 \wedge \d g^2$, the situation is not so simple, for additional boundary terms appear already in the classical (differentiable) case.

\begin{example}\label{ex:pull-back-zust-smooth}
Let $I \subseteq \R^2$, $\O \subseteq \R^d$ be bounded, $g^1 = x^i$, $g^2 = x^j$ be coordinate functions and let $f \in \germ^0(\O)$  and $\varphi: I \subseteq \R^2 \to \O$ be continuously differentiable. Then, by Example~\ref{ex:classical-zust} and Remark~\ref{rem:classical-chain-rule},
\[ (f \circ \varphi) \d \varphi^i \wedge \d \varphi^j =   (f \circ \varphi) J_\varphi \d s^1 \wedge \d s^2 \quad
\text{with $J_\varphi = \det \bra{
 \begin{array}{cc}
\partial_{s^1} \varphi^i  & \partial_{s^2} \varphi^i  \\
 \partial_{s^1} \varphi^j & \partial_{s^2} \varphi^j
\end{array}
}$}.\]

\begin{figure}[ht]
 \begin{minipage}{0.48\textwidth}
\begin{center}
     \begin{tikzpicture}[scale=1.9]
\filldraw[bottom color=lightgray, top color=white] (0,0) to [bend left=20]  (1,2) to [bend left =30] (2,1) to [bend right=30] (0,0) -- cycle;
	\draw   (0,0) node[below left]{$\varphi(p_0)$} (1,2) node[above]{$\varphi(p_1)$} (2,1) node[right]{$\varphi(p_2)$} (1,1.3) node {$\varphi(S)$}(1,0.7) node {$\varphi_\natural S$};
	\draw[dashed] (0,0) -- (1,2) -- (2,1);
		\draw (2,1) -- (0,0) ;

    \end{tikzpicture}
    \end{center}\subcaption{$\varphi_\natural S$ and $\varphi(S)$}\label{fig:pull-back-classical_1}
\end{minipage}
 \begin{minipage}{0.48\textwidth}
\begin{center}
     \begin{tikzpicture}[scale=1.9]
\filldraw[bottom color=lightgray, top color=white] (0,0) to [bend left=20]  (1,2) to [bend left =30] (2,1) to [bend right=30] (0,0) -- cycle;
	\draw   (0,0) node[below left]{$\varphi(p_0)$} (1,2) node[above]{$\varphi(p_1)$} (2,1) node[right]{$\varphi(p_2)$} (1,0.7) node {$S_1$} (1.65,1.5) node {$S_0$} (0.4,1.05) node {$S_2$};
		\draw[dashed] (0,0) -- (1,2) -- (2,1);
		\draw (2,1) -- (0,0) ;
	\draw[dashed, draw=none, fill = gray, fill opacity=0.1,]	 (0,0) to [bend left=20]  (1,2) --(0,0) -- cycle;
	\draw[dashed, draw=none, fill = gray, fill opacity=0.1] (1,2) to [bend left=30]  (2,1) --(1,2) -- cycle;
		\draw[draw=none, fill = gray, fill opacity=0.2] (2,1) to [bend right=30]  (0,0) --(2,1) -- cycle;

	    \end{tikzpicture}
    \end{center}\subcaption{The ``side faces'' $S_0$, $S_1$, $S_2$}\label{fig:pull-back-classical_2}
\end{minipage}
\caption{}\label{fig:pull-back-classical}
\end{figure}

Given a simplex $S = [p_0 p_1 p_2]$, see Figure~\ref{fig:pull-back-classical}, the quantity
\[ \ang{S, \varphi^\natural (f \d x^i \wedge \d x^j)} = \int_{[\varphi(p_0)\varphi(p_1) \varphi(p_1)]} f \d x^i \wedge \d x^j\]
is computed on the ``basis'' $\varphi_\natural S = [\varphi(p_0)\varphi(p_1) \varphi(p_1)]$, while
\[ \ang{S, \varphi^*( f \d x^i \wedge \d x^j )} = \int_{S}(f \circ \varphi) J_\varphi \d s^1 \wedge \d s^2 \]
is computed over the surface $\varphi(S)$ (Figure~\ref{fig:pull-back-classical_1}), hence their difference can be computed using Stokes-Cartan theorem on a (piecewise $C^1$) domain $U$ whose boundary consists of $\varphi_\natural S$, $\varphi(S)$ and the three suitably oriented ``side faces'' $(S_\ell)_{\ell=0}^2$  (Figure~\ref{fig:pull-back-classical_2}). We have the identity
\begin{equation}\label{eq:identity-zust-pull-back-classical} \begin{split} \ang{S, \varphi^\natural (  f \d x^i \wedge \d x^j )  - \varphi^*(f \d x ^i \wedge \d x^j)} & =  \int_{ \varphi_\natural S} f \d x^i\wedge \d x^j - \int_{\varphi(S)}  f \d x^i \wedge \d x^j  \\
& = \sum_{\ell=0}^2\int_{S_\ell} f \d x^i \wedge \d x^j  + \int_{\partial U} f \d x^i\wedge \d x^j \\
& = \sum_{\ell=0}^2\int_{S_\ell} f \d x^i \wedge \d x^j  + \int_{U} \d f \wedge \d x^i \wedge \d x^j.\end{split}\end{equation}
Since the ``orders'' of the two terms in the last line are respectively $\diam^2(S)$ and $\diam^3(S)$, this shows that we cannot in general expect a result analogous to Proposition~\ref{prop:pull-back-young}, i.e., that $\sew( \varphi^\natural (  f \d x^i \wedge \d x^j ) ) = \varphi^*(f \d x ^i \wedge \d x^j)$, unless some cancellations happen in the ``side faces''.
\end{example}

To generalize the above example  to less regular germs and maps we suitably define the integrals along the ``side faces'' $S_\ell$, $\ell \in \cur{0,1,2}$. Ultimately, their appearance is due to the fact that the $1$-germ $\omega:=\varphi^\natural( f \d g^1 \wedge \d g^2)$ may fail to be nonatomic, hence we first study a more general construction of a ``side compensator'' for $2$-germs.

\begin{definition}[side compensator]
Given $\omega \in \germ^2(\O)$, we say that $\eta \in \germ^1(\O)$ is a \emph{side compensator} of $\omega$  if  $\eta \approx_{\u} 0$ for some $1$-Dini gauge $\u \in \germ^1(\O)$ and
\begin{equation}\label{eq:boundary-comp} \ang{ [prq], \omega - \delta \eta} = 0 \quad \text{for every $[prq] \in \simp^1(\O)$ such that $r = \frac{p+q}{2}$.}\end{equation}
\end{definition}

\begin{example}\label{ex:area-corrector}
If $\omega$ is nonatomic, then $\eta = 0$ is a side compensator for $\omega$. If $\omega = \delta \eta$, with $\eta \approx_\u 0$ for some $1$-Dini gauge $\u \in \germ^1(\O)$ ($\eta$ possibly not regular)  $\eta$ itself is a side compensator for $\omega$.
\end{example}

Condition~\eqref{eq:boundary-comp} is equivalent to
\begin{equation}\label{eq:area-corrector}  \eta_{pq} =  \eta_{pr}+\eta_{rq}-\omega_{prq}  \quad \text{for every $[prq] \in \simp^1(\O)$ such that $r = \frac{p+q}{2}$.}\end{equation}
Uniqueness of $\eta$ is a consequence of the following result (applied to the difference of two side compensators of a given $\omega$). Its proof is given in Appendix~\ref{app:side-compensator}.

\begin{theorem}[side compensator, uniqueness]\label{thm:comp-uni}
Let $\eta \in \germ^1(\O)$,  $\eta \approx_\u 0$ for some $1$-Dini gauge $\u \in \germ^1(\O)$, and $\ang{ [prq],  \delta \eta} = 0$ for every $[prq] \in \simp^1(\O)$ such that $r = \frac{p+q}{2}$. Then $\eta = 0$. %or equivalently if
\end{theorem}

In what follows, we write $L(\omega):=\eta$ whenever the side compensator $\eta$ of $\omega$ exists. With this notation, Example~\ref{ex:area-corrector} gives
\begin{eqnarray} L(\omega) &=& 0 \quad \text{if $\omega$ is nonatomic, and} \label{eq:ex-sude-corr-1}\\
 L(\delta \eta) &=& \eta \quad \quad \text{if $\eta \approx_\u 0 $  with $\u$  $1$-Dini.} \label{eq:ex-sude-corr-2}
 \end{eqnarray}
 The map $\omega \mapsto L(\omega)$ is linear on $2$-germs for which it is defined and hence, by~\eqref{eq:ex-sude-corr-2},
 \begin{equation}\label{eq:ex-side-corr-3} L(\omega - \delta L(\omega)) = L(\omega) - L(\delta L(\omega)) = L(\omega) - L(\omega) = 0.\end{equation}

 \begin{remark}[locality]\label{rem:locality-side}
Arguing as in Corollary~\ref{cor:locality}, we have that, if $\omega \in \germ^2(\O)$ is such that $L(\omega)$ exists, then, for every $\tilde{\O} \subseteq \O$, its restriction $\tilde{\omega}$ to $\simp^k(\tilde{ \O})$ is such that $L(\tilde\omega)$ exists, being simply the restriction of $L(\omega)$ to $\simp^k(\tilde{\O})$.\end{remark}

\begin{remark}\label{rem:side-alterating} If $\omega$ is alternating, then $L(\omega)$ is also alternating, since if $\sigma$ is the transposition $\sigma[pq]= [qp]$, then letting $\tilde\sigma [prq] := [qrp]$ one has $\tilde \sigma^\prime \omega = -\omega$ and since $r = \frac{p+q}{2} = \frac{q+p}{2}$, we have
\[\begin{split}
 0 & =  \ang{ [qrp],  \delta  \eta - \omega}  = \ang{ \tilde \sigma [qrp],  \delta  \eta - \omega} = \ang{ [prq], \tilde \sigma^\prime( \delta  \eta - \omega)} \\
 &  =\ang{ [prq], \delta (\sigma^\prime  \eta)  - \tilde \sigma^\prime \omega}
\end{split}
\]
yielding $\sigma^\prime \eta = L( \tilde \sigma^\prime \omega)  = L(-\omega) = - L(\omega)$ (we also use the fact that $\sigma^\prime \eta \approx_\u 0$ follows from $\eta \approx_\u 0$).
\end{remark}

Sufficient conditions for existence of the side compensator are given in the following result, whose proof is postponed in Appendix~\ref{app:side-compensator}.

\begin{theorem}[side compensator, existence]\label{thm:comp:ex}
Let $\u \in \germ^2(\O)$ be strong $1$-Dini. Then, there exists a $1$-Dini gauge $\v \in \germ^1(\O)$ such that, for every $\omega \in \germ^2(\O)$ with $\omega \approx_\u 0$, the side compensator $L(\omega)$ exists with $[L(\omega)]_{\v} \le [\omega]_\u$. Moreover, if $\u = \diam^{\alpha}$ then $\v = \c(\alpha) \diam^{\alpha}$.
\end{theorem}

As in the case of Theorem~\ref{thm:sew-existence}, the proof of this result and Remark~\ref{rem_Dini_gauge1} (v) give that if $\u$ is strong $r$-Dini (with $r \ge 1$) it follows that $\v$ is $r$-Dini as well. Another useful result is the following one (again, the proof is postponed in Appendix~\ref{app:side-compensator}).

\begin{theorem}\label{thm:cancellation}
Let $\omega \in \germ^2(\O)$ be closed on $2$-planes, alternating and such that $\omega \approx_\u 0$ for some strong $2$-Dini gauge $\u \in \germ^2(\O)$. Then $\omega = \delta L(\omega)$.
\end{theorem}

Using the notion of side compensator, we formulate the analogue of Proposition~\ref{prop:pull-back-young} for $2$-germs.

\begin{proposition}\label{prop:pull-back-zust}
Let $\sigma \in (0,1]$, $\alpha>0$, $\beta>1$ with $\gamma := \alpha +\beta >2/\sigma$. Let $f, g^1, g^2 \in \germ^0(\O)$ be continuous with $\sqa{\delta f}_{\alpha} + \sqa{ \delta g^1 \cwedge \delta g^2}_\beta < \infty$ and let $\varphi \colon I \subseteq \R^2 \to \O$ with $\sqa{\delta \varphi}_\sigma < \infty$. Then  %$\varphi^\natural( f \cup \delta g)$ is $\sigma\gamma$-sewable with
%\[\sew( \varphi^ \prime ( f \cup \delta g) )  = \sew \bra{ \varphi^\natural( f \d g) } =  (f\circ \varphi) \d (g\circ \varphi).\]
\begin{equation}\label{eq:pull-back-zust-proof} \varphi^\natural ( f \d g^1 \wedge \d g^2) - \varphi^* ( f \d g^1 \wedge \d g^2) \approx_{\gamma \sigma} \delta \eta,\end{equation}
with $\eta =  L(\varphi^\natural ( f \d g^1 \wedge \d g^2))$.
\end{proposition}

\begin{remark}
By Theorem~\ref{thm:zust} and Example~\ref{ex:classical-zust}, letting $x^0, x^1, x^2$ coordinate functions in $D \subseteq \R^3$ bounded,
\[ x^0 \d x^1 \wedge \d x^2 \approx_3 x^0 \cup \frac 12 \bra{ \delta x^1\cup \delta x^2 - \delta x^2 \cup \delta x^1}.\]
Therefore, choosing $\varphi: I^2 \to \O \subseteq \R^3$, $\varphi:= (f, g^1, g^2)$ with $\sqa{\delta \varphi}_\alpha$ and $\alpha> 2/3$, we have
\[ \varphi^\natural( x^0 \d x^1 \wedge \d x^2) \approx_{3\alpha}f \cup   \frac 1 2 \bra{ \delta g^1 \cup \delta g^2 - \delta g^2 \cup \delta g^1}\]
and $\varphi^*(x^0 \d x^1 \wedge \d x^2) = f \d g^1 \wedge \d g^2$, hence we obtain from~\eqref{eq:pull-back-zust-proof} that
\[   f \cup \frac 1 2 \bra{ \delta g^1 \cup \delta g^2 - \delta g^2 \cup \delta g^1} - f \cup \delta(g^1 \d g^2) \approx_{3 \alpha} \delta \eta,\]
with $\eta:= L( \varphi^\natural( x^0 \d x^1 \wedge \d x^2))$. In particular,
\[ \sew\bra{   f \cup\frac 12  \bra{ \delta g^1 \cup \delta g^2 - \delta g^2 \cup \delta g^1} - \delta \eta} = f \d g ^1 \wedge \d g^2. \]%\delta g^2 \cup \delta g^1} - f \cup \delta(g^1 \d g^2)  \approx_{3 \alpha} \delta \eta \]
Note also that, from Theorem~\ref{thm:comp:ex}, one has $\eta \approx_{3\alpha} L(  f \cup\frac 12  \bra{ \delta g^1 \cup \delta g^2 - \delta g^2 \cup \delta g^1})$.
\end{remark}

\begin{proof}
Consider first the case $f = 1$. In this case, we claim that
\begin{equation}\label{side-corrector-proof-step-1-zust}  L(  \varphi^\natural(\d g^1 \wedge\d g^2) ) = \varphi^\natural(g^1 \d g^2) - (g^1 \circ \varphi) \d (g^2 \circ \varphi) =: \eta,\end{equation}
so that~\eqref{eq:pull-back-zust-proof} is an identity, i.e.,
\begin{equation}\label{eq:identity-proof-pull-back-zust} \varphi^\natural ( \d g^1 \wedge \d g^2) - \varphi^* ( \d g^1 \wedge \d g^2) = \delta \eta.\end{equation}
Minding Example~\ref{ex:pull-back-zust-smooth}, this follows because the volume integral in the right hand side in the last line of~\eqref{eq:identity-zust-pull-back-classical} is identically zero.
%We have
%\[\varphi^\natural (f \d g^1 \wedge \d g^2) = \varphi^\natural (\delta (g^1 \d g^2)) = \delta \varphi^\natural(g^1 \d g^2), \quad \varphi^*(f \d g^1 \wedge \d g^2) = \delta( (g^1 \circ \varphi) \d (g^2 \circ \varphi) ).\]
To prove~\eqref{side-corrector-proof-step-1-zust}, we notice that $\eta$ in a  $1$-germ with $\eta \approx_{ \beta \sigma} 0 $, because of Proposition~\ref{prop:pull-back-young} (with $\beta$ instead of $\gamma$) and if $[prq] \in \simp^2(\O)$ is such that $r = \frac{p+q}{2}$, then
\[\ang{ [prq],  \delta \varphi^\natural(g^1 \d g^2) - \delta  \eta} = \ang{ [pqr], \delta (g^1 \circ \varphi) \d (g^2 \circ \varphi)} = 0\]
since $(g^1 \circ \varphi) \d (g^2 \circ \varphi)$ is regular, hence $\delta (g^1 \circ \varphi) \d (g^2 \circ \varphi)$ is nonatomic.

Next, to obtain the thesis for a general $f$, we argue that
\[\begin{split}
\varphi^\natural ( f \d g^1 \wedge \d g^2) & - \varphi^* ( f \d g^1 \wedge \d g^2) \\
& \approx_{\gamma \sigma} \varphi^\natural \bra{ f \cup \delta (g^1 \d g^2) } - f\circ \varphi  \cup \delta \varphi^* (g^1 \d g^2) \quad \text{by Theorem~\ref{thm:zust}}\\
& = f\circ \varphi  \cup   \varphi^\natural \delta (g^1 \d g^2) -  f\circ \varphi  \cup \delta \varphi^* (g^1 \d g^2) \\
& = f\circ \varphi  \cup  \bra{  \varphi^\natural ( \d g^1 \wedge \d g^2) -  \varphi^* (\d g^1 \wedge \d g^2) }\\
& = f\circ \varphi  \cup \delta L( \varphi^\natural( \d g^1 \wedge \d g^2) ) \quad \text{by~\eqref{eq:identity-proof-pull-back-zust}}\\
& \approx_{\gamma \sigma} \delta L\bra{ \varphi^\natural( \d g^1 \wedge \d g^2) } \quad \text{(to be justified),}
\end{split}
\]
which concludes the proof up to justifying the last passage. The latter justification follows from
\begin{eqnarray}
f \circ \varphi \cup \delta L\bra{ \varphi^\natural( \d g^1 \wedge \d g^2) } & \approx_{\gamma\sigma} &  \delta L\bra{ f \circ \varphi \cup \varphi^\natural( \d g^1 \wedge \d g^2) }, \label{eq:aux-1-zust-pull-back}\\
L\bra{ f \circ \varphi \cup \varphi^\natural( \d g^1 \wedge \d g^2) }  &\approx_{\gamma \sigma}&  L\bra{ \varphi^\natural (f \d g^1 \wedge \d g^2) } . \label{eq:aux-2-zust-pull-back}
\end{eqnarray}
To prove~\eqref{eq:aux-2-zust-pull-back}, we notice that, as a consequence of Theorem~\ref{thm:comp:ex}, if $\omega_1 \approx_{\gamma\sigma} \omega_2$ and $L(\omega_1)$ and $L(\omega_2)$ exist, then $L(\omega_1) \approx_{\gamma \sigma} L(\omega_2)$ by linearity of $L$, and hence it is enough to apply this principle to $\omega_1 :=  f\circ \varphi \cup \delta \varphi^\natural ( g^1  \d g^2)$ and $\varphi^\natural( f \d g^1 \wedge \d g^2)$ observing that $L(\omega_1)$ and $L(\omega_2)$ exist by Theorem~\ref{thm:comp:ex} applied with $\u = \diam^{\beta \sigma}$, and by Theorem \ref{thm:zust}
\[ f \d g^1 \wedge \d g^2 \approx_{\gamma} f\cup \delta ( g^1  \d g^2) \quad \text{hence} \quad \varphi^\natural( f \d g^1 \wedge \d g^2) \approx_{\gamma \sigma} f\circ \varphi \cup \varphi^\natural (\delta g^1  \d g^2) = \omega_2. \]

To prove~\eqref{eq:aux-1-zust-pull-back}, minding that it becomes an identity if $f \circ \varphi$ is constant, we fix $S = [p_0 p_1 p_2]$ and without loss of generality assume that $D=S$ and $f_{\varphi_{p_0}} = 0$. In this case the left hand side of~\eqref{eq:aux-1-zust-pull-back} evaluated at $S$ is identically zero, and
\[ \begin{split}  \abs{\ang{S, \delta L\bra{ f \circ \varphi \cup \varphi^\natural( \d g^1 \wedge \d g^2)  }}} & = \abs{\ang{\partial S,  L\bra{ f \circ \varphi \cup \varphi^\natural( \d g^1 \wedge \d g^2)  }}} \\
& \le 3 \c(\beta \sigma) \sqa{ f \circ \varphi \cup \varphi^\natural( \d g^1 \wedge \d g^2) }_{\beta \sigma} \diam(S)^{\beta \sigma} \\
& \quad \text{by Theorem~\ref{thm:comp:ex} with $\u := \diam^{\beta \sigma}$} \\
&  \le \c \norm{ f \circ \varphi}_0 \sqa{\varphi^\natural( \d g^1 \wedge \d g^2) }_{\beta \sigma} \diam(S)^{\beta \sigma} \\
&   \le  \c \sqa{ \delta f \circ \varphi}_{\alpha \sigma} \sqa{\varphi^\natural( \d g^1 \wedge \d g^2) }_{\beta \sigma} \diam(S)^{ \gamma \sigma}
\end{split}
\]
using the assumption $f_{\varphi_{p_0}}= 0$ in the latter inequality.
\end{proof}

Combining this result with Theorem~\ref{thm:cancellation}, we deduce the following change of variables formula in top dimension (i.e., $d=2$), which  improves~\cite[Proposition 3.3 (6)]{zust_integration_2011}, since $\varphi$ is not required to fix the boundary of $S$.

\begin{corollary}
Under the assumptions of the above proposition, assume furthermore that $D \subseteq \R^2$. Then, for every $S= [p_0 p_1 p_2] \in \simp^2(I)$,
\begin{equation}\label{eq:formula-pull-back-final0} \int _{\varphi_\natural S} f\d g^1 \wedge \d g^2 = \int_{S} (f\circ \varphi) \d (g^1 \circ \varphi) \wedge \d ( g^2 \circ \varphi) + \ang{\partial S, \eta},\end{equation}
with $\eta =  L(\varphi^\natural ( f \d g^1 \wedge \d g^2))$. In particular, if $\varphi( \conv( [p_i p_j] ) \subseteq \conv( [ \varphi_{p_i} \varphi_{p_j}])$ for every $\cur{i,j} \subseteq \cur{0,1,2}$, then $\ang{[p_i p_j], \eta} = 0$, and hence
\begin{equation}\label{eq:formula-pull-back-final} \int _{\varphi_\natural S} f\d g^1 \wedge \d g^2 = \int_{S} (f\circ \varphi) \d (g^1 \circ \varphi) \wedge \d ( g^2 \circ \varphi).\end{equation}
\end{corollary}

\begin{proof}
We apply Theorem~\ref{thm:cancellation} with $\omega:= \varphi^\natural ( f \d g^1 \wedge \d g^2) - \varphi^* ( f \d g^1 \wedge \d g^2) - \delta \eta$, $\eta := L\bra{ \varphi^\natural ( f \d g^1 \wedge \d g^2)}$ and the strong $2$-Dini gauge $\u := \diam^{\gamma \sigma}$. The condition $\omega\approx_\u  0$ is~\eqref{eq:pull-back-zust}, $\omega$ is alternating because so are $f \d g^1 \wedge \d g^2$, $\varphi^* (f \d g^1 \wedge \d g^2)$ and $\varphi^\natural$, $L$ and $\delta$ preserve the alternating property (Remark~\ref{rem:side-alterating}). We have
\[ \delta \omega = \varphi^\natural\bra{\delta ( f \d g^1 \wedge \d g^2)} - \delta \varphi^*(f\d g^1 \wedge \d g^2) = 0,\]
the first term being zero because $\delta ( f \d g^1 \wedge \d g^2) = 0$ by the top-dimensional assumption $D \subseteq \R^2$, the second term because it is a regular form. Finally,
\[\begin{split} L(\omega) & = L ( \varphi^\natural ( f \d g^1 \wedge \d g^2) - \varphi^* ( f \d g^1 \wedge \d g^2) - \delta \eta) \\
& =   - L(\varphi^* ( f \d g^1 \wedge \d g^2))  \quad \text{by~\eqref{eq:ex-side-corr-3} with  $\varphi^\natural ( f \d g^1 \wedge \d g^2)$ instead of $\omega$}\\
& =  0 \quad \quad \text{by~\eqref{eq:ex-sude-corr-1}, by nonatomicity of $\varphi^* ( f \d g^1 \wedge \d g^2)$,}\end{split}\]
from which we deduce from Theorem~\ref{thm:cancellation} that $\omega = 0$, i.e.,~\eqref{eq:formula-pull-back-final0} holds. To prove~\eqref{eq:formula-pull-back-final}, it is sufficient to notice that  the restriction of the $2$-germ $\varphi^\natural ( f \d g^1 \wedge \d g^2)$ to $\conv( [ {p_i} {p_j}])$ is zero, because $f \d g^1 \wedge \d g^2$ is nonatomic and $\vol_2( \varphi_\natural S) = 0$ for every simplex $S$ in $\conv( [ {p_i} {p_j}])$, hence by Remark~\ref{rem:locality-side} we deduce that $\ang{ [p_i p_j], \eta} = 0$.
\end{proof}

Other sufficient conditions implying that~\eqref{eq:formula-pull-back-final} holds, i.e., $\ang{\partial S,  \eta} = 0$, are for example $f\circ \varphi =0$ or $g^\ell \circ \varphi $ constant (for some $\ell=1,2$) on $\conv( [ {p_i} {p_j}])$, for every $\cur{i,j} \subseteq \cur{0,1,2}$, i.e. on the topological border of $S$.

\section{More irregular germs and correctors}\label{sec:rough}

Theorem~\ref{thm:sew-existence} and its Corollary~\ref{cor:sew-stability} provide a linear continuous ``regularizing'' map $\omega \mapsto \sew \omega$ for a large class of $k$-germs, as discussed in Section~\ref{sec:young-zust}. In this section we argue that the same result can be useful also in settings where it is not directly applicable, extending the scope of the sewing lemma~\cite{feyel_curvilinear_2006, gubinelli_controlling_2004, friz_course_2014} in rough paths theory to higher dimensions. To fix the ideas,  consider the following example.

\begin{example}\label{ex:pure-area}
Let $D \subseteq \R^d$ be bounded. For fixed $\xi \in \R^d$, we consider the sequence of $0$-germs
\[  f^n_p := \frac{1}{\sqrt{n}} \cos(n \xi \cdot p) \quad g^n_p := \frac{1}{\sqrt{n}} \sin(n\xi \cdot p) \quad \text{for $p \in D$.}\]

%z^n(p):= n^{-1/2} \exp\bra{i n \xi \cdot p} =
Although each $f^n$, $g^n$ is smooth, pointwise converging to $0$ as $j \to +\infty$, %the ``naive'' limit, even for a smooth $\varphi:\R \to\R$,
%\[ \int_{[pq]} \varphi(f^n) \d g^n  \to \int_{[pq]} \varphi(0) \d 0 = 0,\]
%does not hold in general, as it
it is easy to calculate manually that
\begin{equation}\label{eq:weak-convergence}  \lim_{n \to \infty} \int_{[pq]} f^n \d g^n  \to \frac 1 2  \xi\cdot(q-p).\end{equation}
In our terms, this means that although the germs $f^n \cup \delta g^n \to 0$ pointwise, their regularizations $f^n \d g^n = \sew( f^n \delta g^n) \not \to 0$, showing that we are not in the assumptions of Corollary~\ref{cor:sew-stability}. In particular, this shows that for any strong $r$-Dini gauge $\u \in \germ^2(D)$ (satisfying the assumptions of Corollary~\ref{cor:sew-stability}) with $r \ge 1$,
\begin{equation}\label{eq:finite-gamma-pure-area} \sup_{n \ge 1} \sqa{ \delta f^n  \cup \delta g^n}_{\u} = \infty.\end{equation}
\end{example}

Nevertheless, Example \ref{ex:pure-area} may fit into our framework as follows. Given $f \in \germ^0(D)$, $\eta \in \germ^{k-1}(D)$ regular such that $\delta (f \cup \delta \eta) = (\delta f) \cup (\delta \eta) \not \approx _\u 0$ for every strong $k$-Dini gauge $\u$ (with $\u$ and $\tilde{\u}$ in~\eqref{eq:v-dini} continuous), it may happen that we can find a nonatomic $\omega \in \germ^k(D)$, called further \emph{corrector}, such that
\begin{equation}\label{eq:corrector} \delta (f \cup \delta \eta) \approx_{\u} \delta \omega,\end{equation}
hence Theorem~\ref{thm:sew-existence} (under mild continuity assumptions) applies to $f \cup \delta \eta - \omega$ providing the regular $k$-germ
\[ \sew( f \cup \delta \eta - \omega).\]
%\[ f \cup \delta \eta - \omega \approx_\v \tilde{\omega}.\]
This scheme is advantageous if we are able to construct such a single corrector suitable for a reasonable class of $f$ and $\eta$, since it allows for a ``good'' calculus in terms $f$, $\eta$ and $\omega$. Indeed, once a corrector $\omega$ is chosen for a single pair $( f,  \eta)$, one can generate correctors for pairs that are ``approximately'' close to $( f,  \eta)$. For example, if we consider the $k$-germ $\varphi(f) \cup \delta \eta$, where $\varphi:\R \to \R$ is smooth (with bounded derivative $\varphi'$), then we have%in the sense that
\begin{equation}\label{eq:chain-rule-rough} \begin{split} \delta \bra{ \varphi(f) \cup \delta \eta} & = \delta \varphi(f) \cup \delta \eta \\
& = \varphi'(f) \cup  \delta f \cup \delta \eta + (\delta\varphi (f) -  \varphi'(f) \cup \delta f) \cup \delta \eta\\
& \approx_\u  \varphi'(f) \cup \delta \omega +  (\delta\varphi (f) -  \varphi'(f)  \cup  \delta f) \cup \delta \eta \\
& \qquad \text{by boundedness of $\varphi'$ and~\eqref{eq:corrector},}\\
& =  \delta \bra{ \varphi'(f)\cup \omega} - \delta \varphi'(f) \cup \omega +  (\delta\varphi (f) -  \varphi'(f) \cup  \delta f) \cup \delta \eta \\
& \approx_\u  \delta \bra{ \varphi'(f)\cup \omega},
\end{split}\end{equation}
provided that
\begin{equation}\label{eq:remainder-small-abstract-corrector}  - \delta \varphi'(f) \cup \omega+ (\delta\varphi (f) -  \varphi'(f) \cup  \delta f) \cup \delta \eta  \approx_\u 0 .\end{equation}
If the latter condition holds, then Theorem~\ref{thm:sew-existence} provides existence of the regular $k$-germ
\[ \sew ( \varphi(f) \cup \delta \eta - \varphi'(f) \cup \omega).\]
Therefore, if we interpret $\sew( f \cup \delta \eta - \omega) + \omega$ as a ``corrected'' version of $f\d \eta$ (not necessarily regular, because so maybe $\omega$), then the argument above shows how the ``corrected'' version of $ \varphi(f) \d \eta$ can be consistently defined, i.e., $ \sew ( \varphi(f) \cup \delta \eta - \varphi'(f) \cup \omega) +  \varphi'(f) \cup \omega$.

A second advantage of this scheme is a natural stability provided by Theorem~\ref{thm:sew-continuity} if a sequence $(f^j, \eta^j, \omega^j)_{j \ge 1}$ converges pointwise to $(f, \eta, \omega)$ satisfying~\eqref{eq:corrector} uniformly in $j$ (i.e., $\sup_j [ \delta f^j \cup \delta \eta^j - \omega^j]_\u < \infty$) then by Corollary \ref{cor:sew-stability} the corrected versions of $f^j \d \eta^j$ converge pointwise to that of $f \d \eta$.
%\[ f\d \eta $f \cup \delta fulfilling the hypothesis of

Of course, a trivial choice for~\eqref{eq:corrector} to hold is simply $\omega := f \cup \delta \eta$ (note that it is nonatomic because $\eta$ is assumed regular). On one side, this provides a trivial ``correction'' of $f \d g$, setting it equal to $f \cup \delta \eta$. On the other side, one faces several problems to extend the calculus, since to ensure that~\eqref{eq:remainder-small-abstract-corrector} hold, the first term therein is equal to  \[ \delta \varphi'(f) \cup \omega = (\delta\varphi'(f))\cup\bra{ f \cup \delta \eta} \approx_\u \delta\varphi'(f)\cup \delta \eta, \quad \text{(assuming that $f$ is bounded)}\]
which is comparable (in terms of gauges) to $(\delta f) \cup (\delta \eta)$, because so is $\delta \varphi'(f)$ with $\delta f$, since $\varphi'$ a Lipschitz function. Therefore, we cannot expect $\delta \varphi'(f)\cup\bra{ f \cup \delta \eta} \approx_\u 0$, while the second term in~\eqref{eq:remainder-small-abstract-corrector} is naturally expect to be $\approx_\u 0$ because of smoothness of $\varphi$ (it contains the remainder of the first order Taylor expansion of $\varphi$).

To avoid such trivialities, in rough paths theory (dealing with the case $k=d=1$) one imposes that $\omega$ must be \emph{more regular} than what $f \cup \delta \eta$ is expected to be, in the sense of H\"older (or $p$-variation) gauges, and the existence of $\omega$ in applications is usually shown by means of stochastic calculus.

\begin{example}[Brownian motion as a rough path]
In our terminology, if $D =[0,T] \subseteq \R$ is an interval and $f_t = \eta_t = B_t$ is a trajectory of Brownian motion, we may define a corrector $\omega$ (using Stratonovich integration) as a trajectory of
\[ \omega_{st} = B_s \delta B_{st} -  \int_{s}^t B_r \circ \d B_r = - \frac{1}{2} ( \delta B_{st})^2  \quad \text{for $s$, $t \in [0,T]$.}\]
Kolmogorov criterion gives that $\sqa{ \delta B}_{1/2 - \varepsilon} < \infty$ and $\sqa{ \omega}_{1-2\varepsilon} < \infty$ for every $\varepsilon>0$ (with full probability). As a consequence, one can prove that~\eqref{eq:remainder-small-abstract-corrector} holds with $\u = \diam^{\gamma}$ with $\gamma = (2+\alpha)(1/2 - \varepsilon)$, hence if $\varphi: \R \to \R$ is differentiable with $[\delta \varphi']_\alpha<\infty$, (with $\alpha>0$), one has that $\sew(\varphi(B) \cup \delta B - \varphi'(B) \cup \omega)_{st} = \int_s^t \varphi(B) \circ \d B$, obtaining the rough path (Stratonovich) integral.
\end{example}

 Below we provide two deterministic examples of this scheme, in the cases $k\in \cur{1,2}$ (the case $k=d =1$ being that of rough paths -- of regularity $\alpha >1/3$), extending the well-known ``pure-area'' rough path construction~\cite[Exercise 2.17]{friz_course_2014}. We point out that both examples could be also considered by elementary means or within the theory of Young measures, hence they should be just considered as proof-of-principles.

First, we show that we may use an appropriate corrector also in Example~\ref{ex:pure-area}.

\begin{example}\label{ex:pure-area-1}
In order to apply Corollary~\ref{cor:sew-stability} in Example~\ref{ex:pure-area}, for any $n \ge 1$, we choose the corrector
\begin{equation}\label{eq:corrector-pure-area-1} \omega^n_{pq} := f^n_p \delta g^n_{pq} - \int_{[pq]} f^n \d g^n.\end{equation}
A straightforward computation gives
\[ \int_{[pq]} f^n \d g^n=   \delta I^n_{pq}, \quad \text{where}  \quad I^n_p=   \frac{1}{2} \xi \cdot p +\frac{1}{2n} \sin(2 n \xi \cdot p), \]
hence~\eqref{eq:corrector} is satisfied for every $\u$ (with $f^n$ instead of $f$ and $g^n$ instead of $\eta$) since
\[ \delta \omega^n = \delta f^n \cup \delta g^n - \delta( \delta I^n) = \delta f^n \cup \delta g^n.\]

In order to argue as in~\eqref{eq:chain-rule-rough} we notice first that
\[ \sup_{n \ge 1} \sqa{ \delta f^n}_{1/2}+ \sqa{ \delta g^n}_{1/2}<\infty \quad \text{and also} \quad  \sup_{n \ge 1} \sqa{ \omega^n}_{1} < \infty,\]
using the inequalities
\[ \abs{ \cos(n \xi \cdot q) - \cos(n \xi \cdot p)  } \le \min \cur{ 2,   n |\xi \cdot q- \xi\cdot p|}  \le \sqrt{2 n |\xi \cdot (q-p)|} ,\]
and similarly for the sine terms. If $\varphi: \R \to \R$ is differentiable with $\sqa{\delta \varphi'}_\alpha<\infty$, with $\alpha \in (0, 1]$, then
\[ \sqa{\delta \varphi'(f^n) }_{\alpha/2} \le \sqa{\delta \varphi'}_\alpha \sqa{\delta f^n}_{1/2}\]
and the remainder of the Taylor expansion (with $\operatorname{Id}(x) = x$) gives
\[ \sqa{\delta \varphi(f^n) - \varphi'(f^n) \cup \delta f^n}_{(1+\alpha)/2} \le \sqa{\delta \varphi - \varphi' \cup \delta \operatorname{Id} }_{1+\alpha}  \sqa{\delta f^n}_{1/2},\]
hence the expression in~\eqref{eq:remainder-small-abstract-corrector} reads
\[ - \delta \varphi'(f^n) \cup \omega^n+ (\delta\varphi (f^n) -  \varphi'(f^n) \cup  \delta f^n) \cup \delta g^n \approx_{1+\frac{\alpha}{2}} 0,\]
uniformly as $n \to +\infty$. As a consequence, since $[ \omega^n]_{\diam^2} < \infty$ for every $n \ge 1$ (though not uniformly in $n \ge 1$), one has
\[ \sew( \varphi(f^n) \cup \delta g^n - \varphi'(f^n) \cup \omega^n) =  \sew( \varphi(f^n) \cup \delta g^n) = \varphi(f^n) \d g^n \]
and we deduce convergence of the corrected versions of $\varphi(f^n) \d g^n$, in particular yielding the limit
\[ \lim_{n \to +\infty} \int_{[pq]} \varphi(f^n) \d g ^n = \frac{\varphi'(0)}{2} \xi \cdot (q-p).\]  given by~\eqref{eq:weak-convergence}.
\end{example}

\begin{remark}
We notice that another reasonable choice for a corrector  in Example~\ref{ex:pure-area-1}  would be $\tilde \omega^n := f^n \cup \delta g^n - \frac 1 2 \xi \cdot \delta \operatorname{Id}$, neglecting  the Lipschitz continuous summand in $I^n$ above. With this choice, the only difference would be in the definition of corrected version of the integral of $\varphi(f^n) \d g^n$, where an additional integral would appear (still, negligible in the limit as $n \to +\infty$). However, when $\varphi = \Id$ this choice would give $\sew( f^n \cup \delta g^n - \tilde{ \omega}^n) = \frac 1 2 \xi \cdot \delta \operatorname{Id}$, this time independent of $n$.
\end{remark}

 To conclude, we give a similar example for $2$-germs.

\begin{example}\label{ex:pure-area-2}
Let $D \subseteq \R^2$ be bounded. Consider the sequence of $0$-germs, for $p = (p_1, p_2) \in D$,
\[  f^n_p := \frac{1}{n^{1-\varepsilon}} \cos(n p_1 ) \cos(n p_2), \quad g^n_p := \frac{1}{n^{\varepsilon+ 1/2}} \sin(n p_1), \quad h^n_p := \frac{1}{\sqrt{n}} \sin(n p_2), \quad \]
for some $\varepsilon \in (0,1)$, so that in particular
\[ \sup_{n \ge 1} \sqa{\delta f^n}_{1-\varepsilon}+ \sqa{\delta g^n}_{\varepsilon+ 1/2} + \sqa{\delta h^n}_{1/2}  < \infty.\]
We let $\eta^n := g^n \d h^n$, so that $\ang{[pq], \eta^n} =\int_{[pq]} f^n \d g^n$, and consider the (smooth) $2$-germ $f^n \cup \delta \eta^n$, for which
\begin{equation}\label{eq:pure-area-2d} \begin{split} \ang{ [pqr], \sew( f^n \cup \delta \eta^n )} & = \int_{[pqr]} f^n \d g^n \wedge \d h^n  \\
& = \int_{[pqr]} f^n \det( \nabla g^n, \nabla h^n) \d x^1 \wedge \d x^2 \quad \text{ by Proposition~\ref{prop:chain-rule-zust}}\\
%& \int_{[pqr]} f^n \det( \nabla g^n, \nabla h^n) \d x^1 \wedge \d x^2 \quad \text{by Example~\ref{ex:classical-zust}}\\
& = \int_{ [pqr]} \cos^2(n x) \cos^2(ny) \d x^1 \wedge \d x^2  \\
& = \frac{1}{4}\int_{[pqr]} \bra{1 - \cos(2nx^1)}\bra{1- \cos(2nx^2)} \d x^1 \wedge \d x^2\\
& = \frac{1}{4}\int_{[pqr]} \d x^1 \wedge \d x^2 + o(1) \quad \text{as $n \to \infty$.}\end{split}\end{equation}
Since $g^n \cup \delta h^n \to 0$ pointwise and $\sup_{n \ge 1} \| \delta g^n \cup \delta h^n \|_{1+\varepsilon}  < \infty$, by Corollary~\ref{cor:sew-stability} applied to $\omega^n := g^n \cup \delta h^n$, we obtain that $\eta^n = g^n \d h^n \to 0$ pointwise, hence also $f^n \cup \delta  \eta^n \to 0$ pointwise.  This fact, in combination with~\eqref{eq:pure-area-2d}, shows that we are not in a position to apply Corollary~\ref{cor:sew-stability} to the germs $f^n \cup \delta \eta^n$, in particular it must be $\sup_{n \ge 1} \sqa{ \delta f^n \cup \delta \eta^n}_\u = \infty$ for every strong $r$-Dini gauge $\u\in \germ^3(D)$ satisfying the assumptions in Corollary~\ref{cor:sew-stability}, with $r \ge 2$.

In this case we choose as a corrector,
\[  \omega^n_{pqr} := f^n_p \delta \eta^n_{pqr} - \int_{[pqr]} f^n \d g^n \wedge \d h^n \quad \text{for $[pqr]\in \simp^2(D)$,} \]
 so that~\eqref{eq:corrector} becomes an identity $\delta \omega^n = \delta f^n \cup \delta \eta^n$, since the integral term is a regular top-dimensional form, hence closed. Moreover, by~\eqref{eq:pure-area-2d}, we have $\omega^n \to \frac 1 4 \d x \wedge \d y$ pointwise, where  $\ang{[pqr], \d x \wedge \d y} = \det( q-p, r-p)$  is the ``oriented area'' $2$-germ, and
  \[  \sup_{n \ge 1} [ \omega^n]_{\diam^2}  < \infty,\]% \int_{[pqr]} \d x \wedge \d y$.
using also the inequalities (by Proposition~\ref{prop:bound-delta-young} with $\gamma =2$, $g^n$ instead $f$ and $h^n$ instead of $g$)
\[ [ f^n \cup \delta \eta^n]_{\diam^2} \le \c \norm{f^n}_0 \sqa{ \delta g^n}_1 \sqa{ \delta h^n}_1 \le \c  n^{\varepsilon-1} n^{1/2-\varepsilon} n^{1/2} \le \c.\]

 Finally, if $\varphi: \R \to \R$ is as in in Example~\ref{ex:pure-area-1}, we obtain that the expression in~\eqref{eq:remainder-small-abstract-corrector} reads
\[ - \delta \varphi'(f^n) \cup \omega^n+ (\delta\varphi (f^n) -  \varphi'(f^n) \cup  \delta f^n) \cup \delta g^n \approx_{2 + \alpha(1-\varepsilon)} 0,\]
uniformly as $n \to +\infty$. As a consequence, we have the convergence of the corrected versions of $\varphi(f^n) \d g^n$ given by~\eqref{eq:weak-convergence}, which in particular gives the limit
\[ \lim_{n \to +\infty} \int_{[pqr]} \varphi(f^n) \d g^n \wedge \d h^n = \frac1 4  \varphi'(0) \det( q-p, r-p).\]
\end{example}
%
% also yields
%\begin{equation}\label{eq:delta-eta-pure-area-2} \sup_{n \ge 1} \sqa{ \delta \eta^n}_{1+\varepsilon}< \infty,\end{equation}

\begin{remark}
We notice that in Example~\ref{ex:pure-area-1} (respectively, in Example~\ref{ex:pure-area-2}) we are exactly in the limit case of Young (respectively, Z\"ust) non-integrability, that is, the sums of H\"older exponents of the integrands equal $k$, the dimension of the respective germs. One may also make similar examples with regularity below this threshold, for instance in Example~\ref{ex:pure-area-1} one could have, for $\varepsilon \in [0,1/2)$,
\[ g^n_p := \frac{1}{n^{1/2-\varepsilon}} \sin(n \xi \cdot p) \mathds{1}_{A_n}( \xi \cdot p),\]
where $A_n \subseteq \R$ is a finite union of intervals with endpoints in  $\cur{x \colon \sin( n x) = 0 }$ and, for every $s \le t \in \R$, $\lim_{n \to +\infty} n^{\epsilon} \int_{A_n \cap [s,t]} \d x = (t-s)$. With this choice, $\sqa{g^n}_{1/2-\varepsilon}$ is uniformly bounded but the resulting corrector defined as in~\eqref{eq:corrector-pure-area-1} satisfies only $\sup_{n \ge 1} \sqa{ \omega^n}_{1-\varepsilon} < \infty$. Estimating the expression~\eqref{eq:remainder-small-abstract-corrector} gives, for a $\varphi \in C^{1,\alpha}$,
\[ - \delta \varphi'(f^n) \cup \omega^n+ (\delta\varphi (f^n) -  \varphi'(f^n) \cup  \delta f^n) \cup \delta g^n \approx_{1+\frac{\alpha}{2}-\varepsilon} 0,\]
so that convergence of the corrected versions of $\varphi(f^n) \d g^n$ is guaranteed if $1 + \frac{\alpha}{2} - \varepsilon >1$, that is $\alpha > 2 \varepsilon$.
\end{remark}
\appendix

\section{Decomposition of simplices}\label{app:decomposition-simplex}

\subsubsection*{Geometric maps}
We say that a  map $\tau: \simp^{k}(\O)  \to \simp^{h}(\O)$, $S \mapsto \tau S$, is \emph{geometric} if it is continuous and commutes with affine transformations, i.e., for every $\varphi: \O \to \O$ affine, one has $\tau \varphi_\natural  = \varphi_\natural  \tau$. We extend this definition to the linear span of geometric maps, i.e., we say that $\tau: \simp^{k}(\O) \to \chain^{h}(\O)$ is geometric if it is a finite linear combination  of  geometric maps $\tau^i: \simp^k(\O) \to \simp^{h}(\O)$, $\tau = \sum_{i=1}^\ell \lambda_i \tau_i$, with $\lambda_i \in \R$. We then define $\abs{\tau}$ as the map $\abs{\tau} := \sum_{i=1}^\ell |\lambda_i| \tau_i$. We notice the following fact.

\begin{lemma}[geometric maps and Dini gauges]\label{lem:geo-dini}
Let $\tau: \simp^{k}(\O)  \to \chain^{h}(\O)$ be geometric and $\u \in \germ^h(\O)$ be a gauge. Then the following assertions are true.
\begin{enumerate}[(i)]
\item If $\u$ is $r$-Dini, then $|\tau|' \u$ is $r$-Dini.
\item If $\u$ is strong $r$-Dini then $|\tau|' \u$ is strong $r$-Dini with the germ as in~\eqref{eq:v-dini} given by $|\tau|' \tilde{\u}$.
\item  If $\u$ is continuous, then $|\tau|' \u$ is continuous.
\item For every $\omega \in \germ^h(\O)$,  $[ |\tau|' \omega]_{|\tau|' \u} \le [\omega]_\u$.
\end{enumerate}
\end{lemma}

\begin{proof}
It is sufficient to argue in the case $\tau: \simp^k(\O) \to \simp^h(\O)$, the general case following since the properties are stable with respect to linear combinations. The only non-trivial property is then (i). This follows from the fact that if $S$ is isometric to $S' \in \simp^{k}(\O)$ then $\tau S$ is isometric to $\tau S'$, since geometric maps commute with isometries. Then, for every $n \ge 0$, one has
\[ 2^{nr} \ang{(2^{-n})_\natural S, \tau' \u } =  2^{nr}  \ang{(2^{-n})_\natural \tau S,   \u } \to 0.\]
as $n \to +\infty$. %Moreover, if $\u$ is strong $h$-Dini (with associated $\v$) arguing similarly one obtains
%\[ \sum_{n=0}^{+\infty}  2^{nh} \ang{(2^{-n})_\natural S, \abs{\varrho^\natural} \u } =  \sum_{i=1}^r \abs{\lambda_i} 2\sum_{n=0}^{+\infty}  2^{nh} \ang{(2^{-n})_\natural\varrho_i S,  \u }  = \ang{S, \abs{\varrho^\natural} \v}.\]
%which is finite and continuous. The last statement can be obtained as in Remark~\ref{rem:moduli-pullback}.
\end{proof}

\subsubsection*{Decomposable maps}
We say that $\tau: \simp^{k}(\O)  \to \chain^{h}(\O)$ is \emph{decomposable} if there exist geometric maps $\nu: \simp^{k}(\O)  \to \chain^{h}(\O)$, $\varrho:  \simp^{k}(\O)  \to \chain^{h+1}(\O)$ such that
\begin{equation}\label{eq:zero-volume}\nu  = \sum_{i=1}^\ell \lambda_i \nu_i  \quad \text{with $\vol_h \circ \nu_i = 0$ for $i=1, \ldots, \ell$}\end{equation}
 and
\[ \tau = \nu + \partial \varrho.\]
Given $\tau_1: \simp^{k}(\O)  \to \chain^{h}(\O)$, $\tau_2:\simp^{k}(\O)  \to \chain^{h}(\O)$, we write $\tau_1 \equiv \tau_2$ if the difference $\tau_1 - \tau_2$ is decomposable.

\begin{remark}\label{rem:equivalence-closed}
If $\tau_1 \equiv \tau_2$, then for every $\omega \in \germ^h(\O)$ is closed and nonatomic one has $\tau_1' \omega= \tau_2'\omega$, since for $S \in \simp^k(\O)$,
$\ang{ S, \tau_1' \omega -\tau_2'\omega}= \ang{ \nu S, \omega}  + \ang{ \varrho S, \delta \omega} = 0$.
\end{remark}

We notice that, if $\tau_1: \simp^{k_1}(\O) \to \chain^{k_2}(\O)$, $\tau_2: \simp^{k_2}(\O)  \to \chain^{k_3}(\O)$ are both geometric, then the composition $\tau_2 \circ \tau_1$ is geometric, and if moreover $\tau_2$ is decomposable, then $\tau_2 \circ \tau_1$ is decomposable, for $\tau_2 \circ \tau_1 = \nu \circ \tau_1 + \partial (\varrho \circ \tau_1)$.

\subsubsection*{Permutations} A first example of decomposable maps is provided by permutations.

\begin{lemma}\label{lem:perm}
Let $\sigma$ be a permutation on $\cur{0,1, \ldots, k}$. Then the induced map  $\sigma: \simp^k(\O) \to \simp^k(\O)$ is geometric and $\sigma \equiv (-1)^\sigma \Id$.
\end{lemma}

\begin{proof}
As already noticed in Section~\ref{sec:pol-chains}, $\sigma$ commutes with push-fowards, hence in particular with affine transformations. To prove $\sigma \equiv (-1)^\sigma\Id$ by composition, it is sufficient to assume that $\sigma [p_0 \ldots p_i p_{i+1} \ldots p_k] = [p_0 \ldots  p_{i+1} p_i \ldots p_k]$ is a transposition
with $i\in \cur{0,1, \ldots, k-1}$. % the general case following from the fact that these elements generate the group $\cS^k$. For $C = [p_0 p_1\ldots p_k] \in \simp^k(\O)$, define
For any $S \in \simp^k(\O)$, define the geometric maps
\[ \varrho S := (-1)^i [p_0 p_1 \ldots p_i p_{i+1} p_i \ldots p_k] \in \simp^{k+1}(\O)\]
and
\[ \begin{split} \nu S  := &-\sum_{j = 0}^{i-1} (-1)^{i+j} [p_0 p_1 \ldots \hat{p}_j \ldots p_i p_{i+1} p_i \ldots p_k] + [p_0 p_1 \ldots p_{i}  p_i \ldots p_k]\\
& - \sum_{j = i+2}^{k} (-1)^{i+j} [p_0 p_1  \ldots p_i p_{i+1} p_i \ldots \hat{p}_j \ldots p_k] \in \chain^k(\O),\end{split}\]
which satisfies~\eqref{eq:zero-volume} $p_i$ appears at least twice in every simplex. Then,
\[ \begin{split} \partial \varrho S & = \sum_{j = 0}^{i-1} (-1)^{i+j} [p_0 p_1 \ldots \hat{p}_j \ldots p_i p_{i+1} p_i \ldots p_k] \\
& \quad + [p_0  \ldots p_{i-1} p_{i+1} p_i \ldots p_k]  - [p_0  \ldots p_{i}  p_i \ldots p_k] +  [p_0 \ldots p_{i} p_{i+1} p_{i+2} \ldots p_k] \\
& \quad + \sum_{j = i+2}^{k} (-1)^{i+j} [p_0 p_1  \ldots p_i p_{i+1} p_i \ldots \hat{p}_j \ldots p_k] \\
& = -\nu S + S + \sigma S. \end{split}\]
showing the equivalence $\sigma \equiv (-1)^\sigma\Id$.
\end{proof}

\subsubsection*{Edge-flipping} A second example of decomposable maps is \emph{edge-flipping} on $2$-chains. In general, given $(p_i)_{i=0}^3$, a general edge-flipping consists of replacing $[p_0 p_1 p_2]+[p_3 p_2 p_1]$ with $[p_2 p_0 p_3] + [p_1 p_3 p_0]$ (Figure~\ref{fig:flip-1}).

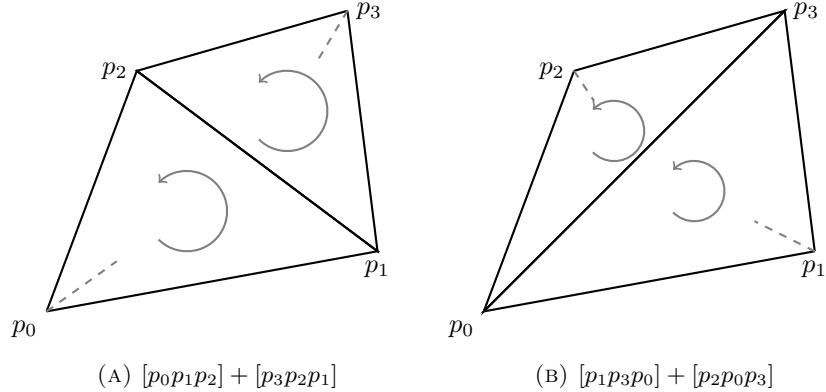
\begin{figure}[h]
\begin{minipage}{.45\textwidth}
\begin{tikzpicture}[scale=4]
 \drawchain{1}{1}{1.1}{.2}{0.3}{0.8}{0.4}{-135}{1}
        \drawchain{0}{0}{0.3}{0.8}{1.1}{0.2}{0.4}{-135}{1}
	\draw   (0,0) node[below left]{$p_0$} (1.1,0.2) node[below]{$p_1$} (0.3,0.8) node[left]{$p_2$}
	(1,1) node[right]{$p_{3}$}  ;
    \end{tikzpicture}\subcaption{$[p_0 p_1 p_2]+[p_3 p_2 p_1]$ }
\end{minipage}
\begin{minipage}{.45\textwidth}
\begin{tikzpicture}[scale=4]
      \drawchain{1.1}{0.2}{0}{0}{1}{1}{0.3}{-135}{1}
        \drawchain{0.3}{0.8}{1}{1}{0}{0}{0.3}{-135}{1}
	\draw   (0,0) node[below left]{$p_0$} (1.1,0.2) node[below]{$p_1$} (0.3,0.8) node[left]{$p_2$}
	(1,1) node[right]{$p_{3}$};
    \end{tikzpicture} \subcaption{$[p_1 p_3 p_0]+[p_2 p_0 p_3]$}
\end{minipage}
\caption{A general edge flipping operation.} \label{fig:flip-1}
\end{figure}

However, to define it as a decomposable map on $2$-simplices, we consider only the case when the points belong to a parallelogram, i.e.\
\begin{equation}\label{eq:parallelogram-point} p_0+ p_3 = p_1 +p_2. \end{equation}
For $S = [p_0 p_1 p_2]$, we define the maps $\pi$, $\tilde \pi:\simp^2(\O) \to \chain^2(\O)$,
\[ \pi S := [p_0 p_1 p_2]+[p_3 p_2 p_1], \quad  \tilde \pi S := [p_1 p_3 p_0]+[p_2 p_0 p_3]\]
with $p_3$ such that~\eqref{eq:parallelogram-point} holds,  and let $\flip S := \pi - \tilde \pi$.

\begin{remark} In fact, it may be that $[p_3 p_2 p_1]$, $[p_1 p_3 p_0]$ or $[p_2 p_0 p_3]$ do not belong to $\simp^2(\O)$, e.g.\ when $p_3 \not \in \O$. Hence, rigorously $\pi$, $\tilde{ \pi}$ and $\flip$ are well-defined on a smaller domain than $\simp^2(\O)$, precisely those $S \in \simp^2(\O)$ such that the three simplices above still belong to  $\simp^2(\O)$. In our applications, this will be always the case hence we avoid to add this specification.
\end{remark}

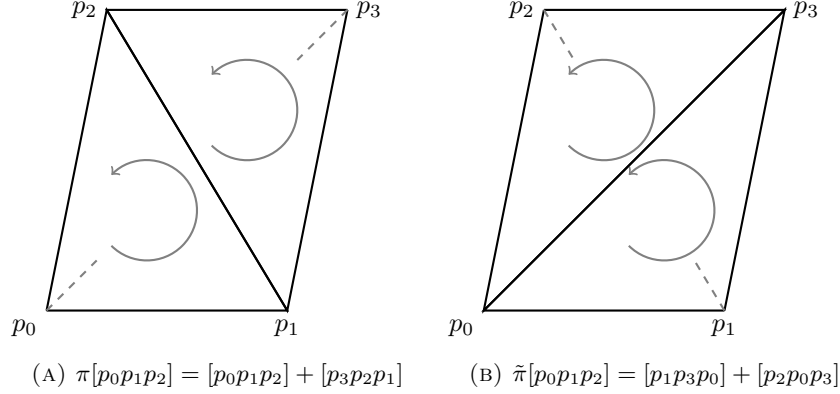
\begin{figure}[ht]
\begin{minipage}{.45\textwidth}
\begin{tikzpicture}[scale=4]
      \drawchain{0}{0}{0.8}{0}{0.2}{1}{0.5}{-135}{1}
        \drawchain{1}{1}{0.2}{1}{0.8}{0}{0.5}{-135}{1}
	\draw   (0,0) node[below left]{$p_0$} (0.8,0) node[below]{$p_1$} (0.2,1) node[left]{$p_2$}
	(1,1) node[right]{$p_{3}$}  ;
    \end{tikzpicture} \subcaption{$\pi [p_0 p_1 p_2] = [p_0 p_1 p_2]+[p_3 p_2 p_1]$ }
\end{minipage}
\begin{minipage}{.45\textwidth}
\begin{tikzpicture}[scale=4]
      \drawchain{0.8}{0}{0}{0}{1}{1}{0.5}{-135}{1}
        \drawchain{0.2}{1}{1}{1}{0}{0}{0.5}{-135}{1}
	\draw   (0,0) node[below left]{$p_0$} (0.8,0) node[below]{$p_1$} (0.2,1) node[left]{$p_2$}
	(1,1) node[right]{$p_{3}$}  ;
    \end{tikzpicture} \subcaption{$\tilde{\pi} [p_0 p_1 p_2] =[ p_1 p_3 p_0]+[p_2 p_0 p_3]$ }
\end{minipage}
\caption{Parallelograms associated to $[p_0 p_1 p_2]$.}\label{fig:flip-2}
\end{figure}

\begin{lemma}\label{lem:flip}
The maps $\pi$, $\tilde \pi$ are geometric and $\pi \equiv \tilde \pi$, i.e.\ $\flip$ is decomposable. %$T^\f: \chain^2(\O) \to \chain^3(\O)$, $R^\sigma: \chain^{2}(\O) \to \chain^{2}(\O)$ such that, for every $C \in \chain^2(\O)$, one has $\abs{R^\sigma_C} = 0$ and
 %\[ P_C = Q_C + R^\f_C+ \partial T^\f_C.\]
\end{lemma}

\begin{proof}
Given $S=[p_0 p_1 p_2]$, the maps $\tau_1 S := [p_1p_2 p_3]$, $\tau_2 S := [p_1 p_3 p_0]$ and $\tau_3S := [p_2 p_0 p_3]$ are clearly geometric entailing that $\pi = \Id + \tau_1$ and $\tilde \pi=\tau_2 + \tau_3$ are both geometric. To show that $\flip$ is decomposable, we notice that, introducing the geometric map $\varrho [p_0 p_1p_2] := [p_3p_0 p_1 p_2]$, one has
\begin{equation}\label{eq:flip-proof} \begin{split} \partial \varrho [p_0 p_1 p_2]  & = [p_0 p_1p_2] -[p_3 p_1 p_2]+ [p_3 p_0p_2] - [p_3 p_0p_1] \\
& = \pi[p_0 p_1 p_2] -\bra{[p_3 p_1 p_2] + [p_3p_2 p_1]} \\
& \quad - \tilde \pi[p_0 p_1 p_2] - \bra{[p_2 p_0 p_3] + [p_3 p_0p_2]} + \bra{[p_1 p_3 p_0] -[p_3 p_0p_1]}\end{split}
\end{equation}
an identity entailing that $\partial \varrho \equiv \pi - \tilde{\pi} = \flip$.
\end{proof}

\begin{remark}\label{rem:flip-proof-cancellation}
From identity\eqref{eq:flip-proof} it actually follows  that if $\omega \in \germ^2(\O)$ is closed and alternating (but not necessarily nonatomic), then $\ang{\flip [p_0p_1 p_2], \omega} =  0$.
\end{remark}

\begin{remark}[change of base points]\label{rem:change-flip}
For technical reasons, we may require to perform slightly different edge-flippings: given $[p_0 p_1 p_2] \in \simp^2(\O)$, letting $p_3$ as above, we  replace
\[ \pi^\dagger[p_0 p_1 p_2] := [p_2 p_0 p_1] + [p_1 p_3 p_2] \quad \text{with} \quad \tilde \pi^\dagger [p_0 p_1 p_2] := [p_3 p_0 p_1]+ [p_0p_3p_2],\]
(Figure~\ref{fig:flip-var}) letting $\flip^\dagger := \pi^\dagger - \tilde{ \pi}^\dagger$. However, it is immediate to notice that $\pi^\dagger = \sigma \pi$, with $\sigma[p_0p_1p_2] :=[p_2 p_0p_1]$ and $\tilde{\pi}^\dagger = \tilde{ \sigma} \tilde \pi^\dagger$ with $\tilde{\sigma}[p_0 p_1 p_2] := [p_1 p_2 p_0]$, hence $\flip^\dagger \equiv 0$.
\end{remark}

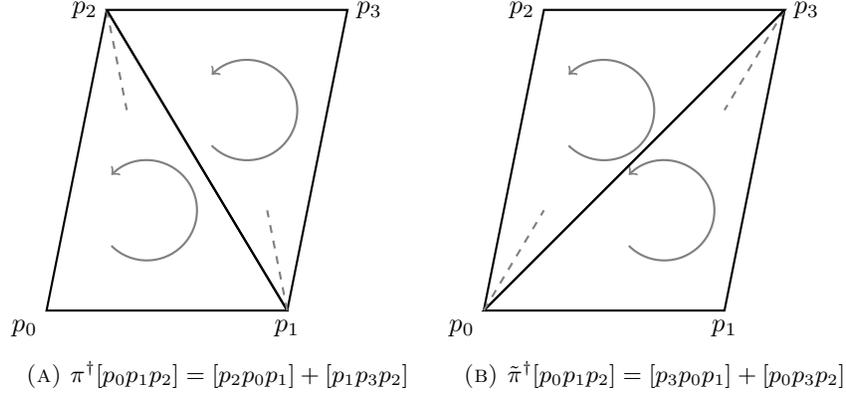
\begin{figure}[ht]
\begin{minipage}{.45\textwidth}
\begin{tikzpicture}[scale=4]
      \drawchain{0.2}{1}{0}{0}{0.8}{0}{0.5}{-135}{1}
        \drawchain{0.8}{0}{0.2}{1}{1}{1}{0.5}{-135}{1}
	\draw   (0,0) node[below left]{$p_0$} (0.8,0) node[below]{$p_1$} (0.2,1) node[left]{$p_2$}
	(1,1) node[right]{$p_{3}$}  ;
    \end{tikzpicture} \subcaption{$\pi^\dagger [p_0 p_1 p_2] = [p_2 p_0 p_1] + [p_1 p_3 p_2]$ }
\end{minipage}
\begin{minipage}{.45\textwidth}
\begin{tikzpicture}[scale=4]
      \drawchain{1}{1}{0.8}{0}{0}{0}{0.5}{-135}{1}
        \drawchain{0}{0}{0.2}{1}{1}{1}{0.5}{-135}{1}
	\draw   (0,0) node[below left]{$p_0$} (0.8,0) node[below]{$p_1$} (0.2,1) node[left]{$p_2$}
	(1,1) node[right]{$p_{3}$}  ;
    \end{tikzpicture} \subcaption{$\tilde \pi^\dagger [p_0 p_1 p_2] =[p_3 p_0 p_1]+ [p_0p_3p_2]$ }
\end{minipage}
\caption{The two parallelograms with changed base points.}\label{fig:flip-var}
\end{figure}

\subsubsection*{Edge-cutting} As a third example, consider the following maps. For fixed $t \in [0,1]$, define
\[ \cut_t : \simp^1(\O) \to \chain^1(\O), \quad \cut_t [p_0 p_1] := [p_0 p_t] + [p_t p_1]\]
where $p_t := (1-t) p_0 + t p_1$ (Figure~\ref{fig:cut-0}), and
\[ \cut_t: \simp^2(\O) \to \chain^2(\O), \quad  \cut_t [p p_0 p_1]:= [p_t p_1 p ] + [p_t p p_1 ],\]
where $p_t: = (1-t) p_0+ t p_1$ (Figure~\ref{fig:cut-1}).

\begin{figure}[b]
\begin{minipage}{.45\textwidth}
\begin{tikzpicture}[scale=4]
     % \draw[->, thick] (0.8, 0.8) -- (1.2,0.8);
      %\draw[->] (0,-0.1) -- (0,1) node[above] {$y$};
  %  \drawchain{0}{0}{1}{0}{0}{1}{0.5}{-135}{1}
%	\draw   (0,0) node[below left]{$p_0$} (1,0) node[below]{$p_1$} (0,1) node[left]{$p_2$};
%\begin{tikzpicture}[scale=4]
	\node (0) at (-0.5, 0) {$p_0$};
	\node (1) at (0.5, 0.5) {$p_1$};
	\draw[thick, ->]   (0) edge (1);
	\end{tikzpicture}  \subcaption{ $[p_0 p_1]$}
	\end{minipage}
	\begin{minipage}{.45\textwidth}

	\begin{tikzpicture}[scale=4]
%	\begin{tikzpicture}[scale=4]
	\node (0) at (-0.5, 0) {$p_0$};
	\node (1) at (0, 0.25)  {$p_{1/2}$};
		\node (2) at (0.5, 0.5) {$p_1$};
	\draw[thick, ->]   (0) edge (1);
	\draw[thick, ->]  (1) edge (2);
 %  \drawchain{0.5}{0.5}{0}{0}{0}{1}{0.4}{-135}{1}
 %  \drawchain{0.5}{0.5}{0}{0}{1}{0}{0.4}{-135}{1}
%	\draw   (0,0) node[below left]{$p_0$} (1,0) node[below]{$p_1$} (0,1) node[left]{$p_2$}
%	(0.5,0.5) node[above right]{$q$};% (0.5,0) node[below]{$q_2$} (0, 0.5) node[left]{$q_1$};
     \end{tikzpicture} \subcaption{$[p_0 p_{1/2}] + [p_{1/2} p_1]$}
\end{minipage}
 \caption{ $\cut_{1/2} [p_0 p_1]$.} \label{fig:cut-0}
\end{figure}

\begin{figure}[b]
\begin{minipage}{.45\textwidth}
\begin{tikzpicture}[scale=4]
     % \draw[->, thick] (0.8, 0.8) -- (1.2,0.8);
      %\draw[->] (0,-0.1) -- (0,1) node[above] {$y$};
    \drawchain{0}{0}{1}{0}{0}{1}{0.5}{-135}{1}
	\draw   (0,0) node[below left]{$p$} (1,0) node[below]{$p_0$} (0,1) node[left]{$p_1$};
	\end{tikzpicture}  \subcaption{ $[p p_0 p_1]$ \hfill}
	\end{minipage}
	\begin{minipage}{.45\textwidth}

	\begin{tikzpicture}[scale=4]
   \drawchain{0.5}{0.5}{0}{0}{0}{1}{0.4}{-135}{1}
   \drawchain{0.5}{0.5}{0}{0}{1}{0}{0.4}{-135}{1}
	\draw   (0,0) node[below left]{$p$} (1,0) node[below]{$p_0$} (0,1) node[left]{$p_1$}
	(0.5,0.5) node[above right]{$p_{1/2}$};% (0.5,0) node[below]{$q_2$} (0, 0.5) node[left]{$q_1$};
     \end{tikzpicture} \subcaption{$[p_{1/2} p_1 p ] + [p_{1/2} p p_0 ]$}
\end{minipage}
 \caption{ $\cut_{1/2} [p_0 p_1 p_2]$.} \label{fig:cut-1}
\end{figure}

\begin{lemma}\label{lem:cut}
The map $\cut_t$ is geometric and  $\cut_t \equiv \Id$.
%There are geometric maps $T^\c: \chain^2(\O) \to \chain^3(\O)$, $R^\c: \chain^{2}(\O) \to \chain^{2}(\O)$ such that, for every $C \in \chain^2(\O)$, one has $\abs{R^\c_C} = 0$ and
% \[ C = C^\c + R^\c_C + \partial T^\c_C.\]
\end{lemma}

\begin{proof}
Consider first the case of $1$-simplices.  Since $\tau_1[p_0 p_1]:= [p_0 p_t]$, $\tau_2[p_0 p_1]:= [p_t p_1]$ are both geometric, so it is their sum $\cut_t$. Moreover, letting $\varrho [p_0 p_1]:= [p_0 p_t p_1]$, one has $\cut_t - \Id = \partial \varrho$.

In case of $2$-simplices, the maps $\tau_1[p p_0 p_1] := [p_t pp_0]$ and $\tau_2[p p_0 p_1] := [p_t p_1 p]$ are geometric and, letting $\varrho [pp_0 p_1]:= [p_t pp_0 p_1]$,  we notice that
\begin{equation}\label{eq:cut-t-proof}\begin{split}
\partial \varrho[pp_0 p_1] & = [pp_0 p_1] - [p_t p_0 p_1] +[p_t p p_1] - [p_t p p_0] \\
 & = (\Id - \cut_t)[p p_0 p_1] - [p_t p_0 p_1] +\bra{ [p_t p p_1] + [p_t p_1 p]}\\
 & =  (\Id - \cut_t)[pp_0 p_1] - \nu[p p_0 p_1] + (\Id -(-1)^\sigma) [p_t p p_1] \\
 & \quad \text{with $\nu[pp_0 p_1]:= [p_t p_0 p_1]$ and $\sigma[pqr] := [p rq]$.}
  \end{split}\end{equation}
  Since $\vol_2(\nu[pp_0 p_1]) = 0$, it follows that  $\Id \equiv \cut_t$.
%  Since $[qp_0 p_2] + [qp_2 p_0] \equiv 0$ and $[q p_1 p_2]$ has
%are geometric hence $\cut_t$ is geometric. Moreover, choosing a suitable permutation of $\cur{0,1,2}$, one has
%\[  \tau_2 \equiv -\tau_2,   \quad  \text{for $\tau_1'[q_0 q_1 q_2] :=  [q_0q_2q_1]$,}\]
%and
%\[ \Id  -\cut_t\equiv \Id - \tau_1+ \tau_2' = \nu + \partial \varrho,\]
%where $\varrho [p_0 p_1 p_2] := -[q p_0 p_1 p_2]$ and $\nu [p_0 p_1 p_2]:= [q p_1 p_2]$ are such that  $\nu + \partial \varrho \equiv 0$.
%
%
%For $C = [p_0 p_1 p_2]$, let $q = (p_1+p_2)/2$ and $T_C := -[q_0 p_0 p_1 p_2] \in \simp^3(\R^d)$, so that
%\[ [p_0 p_1 p_2] =   [q p_1 p_2] - [q p_0 p_2] + [q p_0 p_1]+  \partial T.\]
%By Lemma~\ref{lem:permutations} with $\sigma = (1,2)$, we have
%\[ [q p_0 p_2]  = \sigma [q p_2 p_0] =  -[q p_2 p_0] + R^\sigma_{ [q p_2 p_0] } + \partial T^\sigma_{[q p_2 p_0] },\]
%hence
%\[ [p_0 p_1 p_2] =  [q p_2 p_0] + [q p_0 p_1] + R^\c_C + \partial T^\c_C,\]
%where
%\begin{equation}\label{eq:r-t-cut-half} R^\c := [q p_1 p_2]-R^\sigma_{ [q p_2 p_0]  }, \quad T^\c_C := T_C + T^\sigma_{ [q p_2 p_0]  }. \qedhere\end{equation}
\end{proof}

\begin{remark}\label{rem:cut-proof-cancellation}
The following observation will be useful in the proof of Theorem~\ref{thm:cancellation}. If $\omega \in \germ^2(\O)$ is closed (not necessarily nonatomic), alternating and \[ \ang{ [p_0 p_t p_1], \omega} = 0 \quad \text{with $p_t :=(1-t)p_1 +tp_2$,}\]
then by~\eqref{eq:cut-t-proof}
\[ \ang{[pp_0 p_1], \omega} =  \ang{ \cut_t [p p_0 p_1], \omega} .\]
\end{remark}

More generally, we need to cut one edge into $n$-equal parts. To this aim, we recursively define, for $n\ge 2$, the map (on $1$-simplices)
\[ \cut^n [p_0 p_1] := \cut^{n-1} [p_0 p_{1/n}] + [p_{1/n} p_1],\]
where $p_{1/n} := \bra{1-\frac 1 n } p_0 + \frac 1 n p_1$, and the map on $2$-simplices
\[ \cut^n [p p_0 p_1] := \cut^{n-1}[p p_0 p_{1/n}] + [p p_{1/n} p_1],\]
where  $p_{1/n} := \bra{1-\frac 1 n } p_0 + \frac 1 n p_1$ (Figures~\ref{fig:cut-n-0} and \ref{fig:cut-n-1}). Notice that $\cut^2$ is different from $\cut_{1/2}$ (but one can prove that $\cut^2 \equiv \cut_{1/2}$).

\begin{figure}[ht]
\begin{minipage}{.45\textwidth}
\begin{tikzpicture}[scale=4]
    \node (0) at (-0.5, 0) {$p_0$};
	\node (1) at (-0.166, 0.1666)  {$p_{1/3}$};
	\node (2) at (0.166, 0.33) {};
	\node (3) at (0.5, 0.5) {$p_1$};

	\draw[thick, ->]   (0) edge (1);
	\draw[thick, ->]  (1) edge (2);
	\draw[thick, ->]  (2) edge (3);	
	    \end{tikzpicture}\caption{ $\cut^3[p_0p_1]$.}\label{fig:cut-n-0}
    \end{minipage}
\begin{minipage}{.45\textwidth}
\begin{tikzpicture}[scale=4]
      \drawchain{0}{0}{1}{0}{0.8}{0.2}{0.2}{-135}{1}
       \drawchain{0}{0}{0.8}{0.2}{0.6}{0.4}{0.2}{-135}{1}
       \drawchain{0}{0}{0.6}{0.4}{0.4}{0.6}{0.2}{-135}{1}
       \drawchain{0}{0}{0.4}{0.6}{0.2}{0.8}{0.2}{-135}{1}
       \drawchain{0}{0}{0.2}{0.8}{0}{1}{0.2}{-135}{1}
	\draw   (0,0) node[below left]{$p$} (1,0) node[above right]{$p_0$} (0,1) node[above right]{$p_1$} (0.8, 0.2) node[right] {$p_{1/5}$};
	%(0.8,0.2) node[above right]{$p_{1/5}$}	(0.6,0.4) node[above right]{$p_{2/5}$}
	%(0.4,0.6) node[above right]{$p_{3/5}$}	(0.2,0.8) node[above right]{$p_{4/5}$};
    \end{tikzpicture}\caption{ $\cut^5[p p_0 p_1]$.}\label{fig:cut-n-1}
    \end{minipage}
\end{figure}

Arguing inductively, one has the following result.

\begin{lemma}\label{lem:cut-n}
For $n\ge1$, the map $\cut^n$ is geometric and  $\cut^n \equiv \Id$.
\end{lemma}

\subsubsection*{Dyadic decomposition}\label{sec:dyadic-decomposition}
As a fourth operation, obtained by composition of the others, we introduce the following decomposition of $S \in \simp^2(\O)$ into four parts, each isometric to $(2^{-1})_\natural S$, hence the called \emph{dyadic}. For $S =[p_0 p_1 p_2] \in \simp^2(\O)$, $i\in \cur{0,1, 2}$ let $q_i := \bra{p_j+ p_k}/{2}$, where $\cur{i,j,k} = \cur{0,1,2}$, and  (Figure~\ref{fig:dya-1})
\[ \dya^1 S := [q_0 q_1 q_2], \quad \dya^2 S := [q_1 q_0 p_2], \quad \dya^3 S:= [q_2 p_1 q_0], \quad \dya^4 S :=  [p_0 q_2 q_1],\]
and finally $\dya := \sum_{i=1}^4 \dya^i$. We introduce similarly a dyadic decomposition of $1$-simplices, simply letting $\dya^1[p_0 p_1] = [p_0 q]$, $\dya^2[p_0 p_1] = [qp_1]$, where $q = \bra{p_0+p_1}/2$ and $\dya = \dya^1 + \dya^2$, which actually coincides with $\cut^2$ or $\cut_{1/2}$.

\begin{remark}[variants]\label{rem:variants} Although the definition of $\dya$ on a  $2$-simplex appears quite natural, there are  variants obtained by changing base points or orientation of the $4$ simplices. A useful one is $\dya^\dagger := - \sigma \dya^1 + \dya ^2 + \dya^3 + \dya^4$, where $\sigma[pqr] := [rqp]$ is a transposition (Figure~\ref{fig:dya-alt}), because the identity $\partial \dya^\dagger = \dya \partial$ holds (Figure~\ref{fig:bd-dya-alt}):
\[ \begin{split} \dya \partial [p_0 p_1 p_2] & = \dya\bra{ [p_1 p_2] - [p_0 p_2] + [p_0 p_1]} \\
& = [p_1 q_0]+ [q_0 p_2] - [p_0 q_1] - [q_1 p_2] + [p_0q_2]+ [q_2 p_1]\\
& = - \bra{[q_1 q_0]- [q_2 q_0] + [q_2 q_1]}  + \bra{ [p_1 q_0] - [q_2q_0] + [q_2 p_1]} \\
 &\quad  + \bra{[q_0 p_2]  - [q_1 p_2] + [q_1 q_0] } + \bra{ [q_2 q_1] - [p_0 q_1] + [p_0q_2]  }\\
& =  \partial \bra{ -  [q_2  q_1 q_0]  + [q_2 p_1 q_0] + [q_1 q_0 p_2] + [ p_0 q_2 q_1]} = \partial \dya^\dagger [p_0 p_1 p_2].
\end{split}\]
\end{remark}

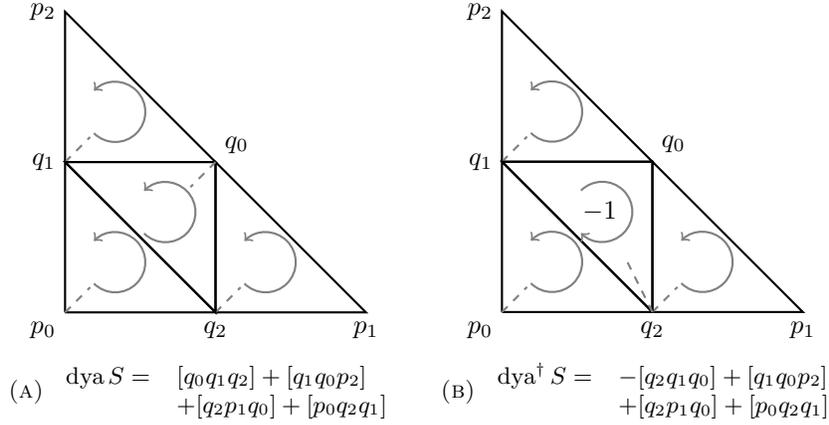
\begin{figure}[ht]
\begin{minipage}{0.45\textwidth}
\begin{center}
	\begin{tikzpicture}[scale=4]
  \drawchain{0}{0.5}{0.5}{0.5}{0}{1}{0.3}{-135}{1}
   \drawchain{0}{0}{0.5}{0}{0}{0.5}{0.3}{-135}{1}
     \drawchain{0.5}{0}{1}{0}{0.5}{0.5}{0.3}{-135}{1}
   \drawchain{0.5}{0.5}{0.5}{0}{0}{0.5}{0.3}{-135}{1}
	\draw   (0,0) node[below left]{$p_0$} (1,0) node[below]{$p_1$}
	(0,1) node[left]{$p_2$}
	(0.5,0.5) node[above right]{$q_0$} (0.5,0) node[below]{$q_2$}
	(0, 0.5) node[left]{$q_1$};
    \end{tikzpicture} \subcaption{$\begin{array}{ll}
        \dya S = & [q_0 q_1 q_2] + [q_1 q_0 p_2]
       \\ &  + [q_2 p_1 q_0]+ [p_0 q_2 q_1] \end{array}$}\label{fig:dya-1}
    \end{center}
\end{minipage}
\begin{minipage}{0.45\textwidth}
\begin{center}
	\begin{tikzpicture}[scale=4]
  \drawchain{0}{0.5}{0.5}{0.5}{0}{1}{0.3}{-135}{1}
   \drawchain{0}{0}{0.5}{0}{0}{0.5}{0.3}{-135}{1}
     \drawchain{0.5}{0}{1}{0}{0.5}{0.5}{0.3}{-135}{1}
   \drawchaininv{0.5}{0}{0.5}{0.5}{0}{0.5}{0.3}{-135}{-1}
	\draw   (0,0) node[below left]{$p_0$} (1,0) node[below]{$p_1$}
	(0,1) node[left]{$p_2$} (0.325, 0.335) node {$-1$}
	(0.5,0.5) node[above right]{$q_0$} (0.5,0) node[below]{$q_2$}
	(0, 0.5) node[left]{$q_1$};
    \end{tikzpicture}\subcaption{$\begin{array}{ll}
        \dya^\dagger S = & -[q_2 q_1 q_0] + [q_1 q_0 p_2]
       \\ &  + [q_2 p_1 q_0]+ [p_0 q_2 q_1] \end{array}$}\label{fig:dya-alt}
    \end{center}
\end{minipage}
    \caption{Dyadic decompositions of $[p_0p_1 p_2]\in \chain^2(\O)$.}
    \end{figure}

    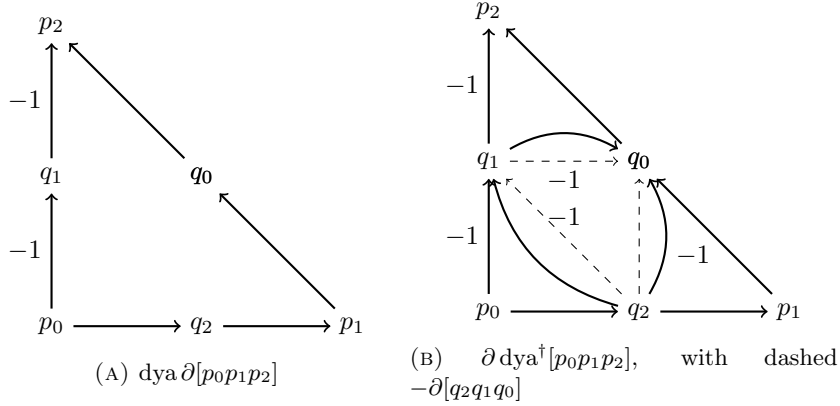
\begin{figure}[ht]
\begin{minipage}{0.45\textwidth}
\begin{center}
	\begin{tikzpicture}[scale=4]
 % \drawchain{0}{0.5}{0.5}{0.5}{0}{1}{0.3}{-135}{1}
  % \drawchain{0}{0}{0.5}{0}{0}{0.5}{0.3}{-135}{1}
  %   \drawchain{0.5}{0}{1}{0}{0.5}{0.5}{0.3}{-135}{1}
  % \drawchain{0.5}{0.5}{0.5}{0}{0}{0.5}{0.3}{-135}{1}
    \node (0) at (0, 0) {$p_0$};
	\node (1) at (1, 0)  {$p_1$};
	\node (2) at (0, 1) {$p_2$};
	\node (3) at (0.5, 0.5) {$q_0$};
		\node (4) at (0.5, 0.5) {$q_0$};
	\node (5) at (0, 0.5) {$q_1$};
	\node (6) at (0.5, 0) {$q_2$};

	\draw[thick, ->]   (0) edge (6);
	\draw[thick, ->]  (6) edge (1);
	\draw[thick, ->]  (1) edge (4);	
		\draw[thick, ->]  (4) edge (2);	
		\draw[thick, ->]  (0) edge node[left]{$-1$} (5);	
		\draw[thick, ->]  (5) edge node[left]{$-1$} (2);	
    \end{tikzpicture} \subcaption{$ \dya \partial [p_0 p_1 p_2]$}
    \end{center}
\end{minipage}
\begin{minipage}{0.45\textwidth}
\begin{center}
	\begin{tikzpicture}[scale=4]
   \node (0) at (0, 0) {$p_0$};
	\node (1) at (1, 0)  {$p_1$};
	\node (2) at (0, 1) {$p_2$};
	\node (3) at (0.5, 0.5) {$q_0$};
		\node (4) at (0.5, 0.5) {$q_0$};
	\node (5) at (0, 0.5) {$q_1$};
	\node (6) at (0.5, 0) {$q_2$};

	\draw[thick, ->]   (0) edge (6);
	\draw[thick, ->]  (6) edge (1);
	\draw[thick, ->]  (1) edge (4);	
		\draw[thick, ->]  (4) edge (2);	
		\draw[thick, ->]  (0) edge node[left]{$-1$} (5);	
		\draw[thick, ->]  (5) edge node[left]{$-1$} (2);
			\draw[thick, ->]   (6) edge[bend left] (5);
			\draw[thick, ->]   (6) edge[bend right] node[below right]{$-1$} (4);
			%\draw[thick, ->]   (6) edge node[right]{$-1$} (4);
			\draw[dashed, ->]   (6) edge  (4);
			\draw[dashed, ->]   (6) edge node[above]{$-1$}  (5);
			\draw[dashed, ->]   (5) edge node[below]{$-1$}  (4);
			\draw[thick, ->]   (5) edge[bend left]  (4);

    \end{tikzpicture}\subcaption{$
        \partial \dya^\dagger [p_0 p_1 p_2]$, with dashed $-\partial [q_2 q_1 q_0]$}
    \end{center}
\end{minipage}
    \caption{The identity $\dya \partial = \partial \dya^\dagger$.}\label{fig:bd-dya-alt}
    \end{figure}

\begin{lemma}\label{lem:dyadic-decomposition}
For $k\in \cur{1,2}$, $i \in \cur{1, \ldots, 2k}$ the maps $\dya^i$ are geometric and  $\dya \equiv \Id$.
%There are geometric maps $T^\dya: \chain^2(\O) \to \chain^3(\O)$, $R^\dya: \chain^{2}(\O) \to \chain^{2}(\O)$ such that, for every $C \in \chain^2(\O)$, one has $\abs{R^\c_C} = 0$ and
% \[ C = C^\dya + R^\dya_C + \partial T^\dya_C.\]
\end{lemma}

\begin{proof}
The case $k=1$ is obvious since $\dya = \cut_{1/2}$. When $k=2$, the maps $\dya^i$ are clearly geometric. By Lemma~\ref{lem:cut} with $t =1/2$, we have  $\Id \equiv \cut_{1/2}$ (Figure~\ref{fig:dya-step-1}) and, iterating, $\Id \equiv \cut_{1/2}\circ \cut_{1/2}$ (Figure~\ref{fig:dya-step-2}). Finally, since $\cut_{1/2}\circ \cut_{1/2}$ differs from $\dya$ by a $\flip$ (applied to $[pq_2q_1]$), using Lemma~\ref{lem:flip}, we conclude that $\cut_{1/2} \circ \cut_{1/2} \equiv \dya$, hence $\Id \equiv \dya$.\end{proof}

\begin{remark}\label{rem:dya-cancellations}
By Remarks~\ref{rem:flip-proof-cancellation} and ~\ref{rem:cut-proof-cancellation}, in view of the above proof, it follows that  if $\omega \in \germ^2(\O)$ is closed (not necessarily nonatomic), alternating and $\ang{ [prq], \omega} = 0$ whenever $[prq] \in \simp^2(\O)$ with $r = \frac{p+q}{2}$, then $\omega = \dya' \omega$, i.e.,
\[ \ang{S, \omega} =  \ang{ \dya S,  \omega} \quad \text{for every $S \in \simp^2(\O)$.}\]
\end{remark}

\begin{figure}[ht]
	\begin{minipage}{.45\textwidth}\begin{center}

	\begin{tikzpicture}[scale=4]
   \drawchain{0.5}{0.5}{0}{0}{0}{1}{0.4}{-135}{1}
   \drawchain{0.5}{0.5}{0}{0}{1}{0}{0.4}{-135}{1}
	\draw   (0,0) node[below left]{$p_0$} (1,0) node[below]{$p_1$} (0,1) node[left]{$p_2$}
	(0.5,0.5) node[above right]{$q$};% (0.5,0) node[below]{$q_2$} (0, 0.5) node[left]{$q_1$};
     \end{tikzpicture} \subcaption{$[q_0 p_2 p_0 ] + [q_0 p_0 p_1 ]$}\label{fig:dya-step-1}
     \end{center}
\end{minipage}
\begin{minipage}{0.45\textwidth}\begin{center}
	\begin{tikzpicture}[scale=4]
  \drawchain{0}{0.5}{0.5}{0.5}{0}{1}{0.3}{-135}{1}
   \drawchain{0}{0.5}{0.5}{0.5}{0}{0}{0.3}{-135}{1}
     \drawchain{0.5}{0}{0.5}{0.5}{1}{0}{0.3}{-135}{1}
   \drawchain{0.5}{0}{0.5}{0.5}{0}{0}{0.3}{-135}{1}
	\draw   (0,0) node[below left]{$p_0$} (1,0) node[below]{$p_1$}
	(0,1) node[left]{$p_2$}
	(0.5,0.5) node[above right]{$q_0$} (0.5,0) node[below]{$q_2$}
	(0, 0.5) node[left]{$q_1$};
    \end{tikzpicture}
 \subcaption{$[q_0 p_2 q_1]+[q_0 q_1 p_0]+ [q_0 p_0 q_2]+[q_0 q_2 p_1]$}\label{fig:dya-step-2}
  \end{center}   \end{minipage}
    \caption{$\cut_{1/2}[p_0 p_1 p_2]$ and $\cut_{1/2} \circ \cut_{1/2}[p_0 p_1 p_2]$.}
\end{figure}
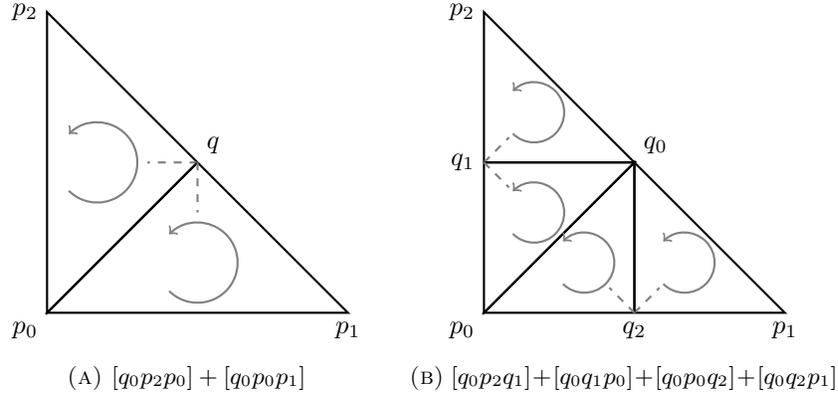

The dyadic decomposition is well-behaved with respect to the other maps.

\begin{lemma}\label{lem:dya-commutator}
The following properties hold.
\begin{enumerate}%[i)]
\item For every permutation $\sigma$ of $\cur{0,1,2}$, $\dya \sigma = \sigma \dya$.
\item There are geometric maps $(\tau^i)_{i=1}^4$, $\tau^i: \simp^2(\O) \to \simp^2(\O)$ such that $\tau^i S$ is isometric to $(2^{-1})_\natural S$ for every $S \in \simp^2(\O)$ and,  for $\tau = \sum_{i=1}^4$ one has
\[ \dya \pi  = \pi \tau, \quad \dya \tilde \pi = \tilde \pi \tau \quad \text{hence} \quad \dya \flip = \flip \tau.\]
\item For $k=1$, $n \ge 1$, $S \in \simp^1(\O)$,% there is $C \in \chain^2(\O)$ such that
\[ \dya \cut^n S = \cut^n \dya S.\]
\item For $k=2$, $n \ge 1$, $S \in \simp^2(\O)$, there exists $C \in  \chain^2(\O)$ such that
\[ \dya \cut^n S = \cut^n \dya S + \flip^\dagger C. \]
\end{enumerate}%sigma\in \cS^2$, one has $\bra{\sigma C }^\dya = \sigma C^\dya$.
\end{lemma}

\begin{proof}
To prove \emph{i)}, it is sufficient to consider transpositions. If $\sigma[p_0p_1 p_2] = [p_1 p_0 p_2]$, then (Figure~\ref{fig:dyadic-perm-1})
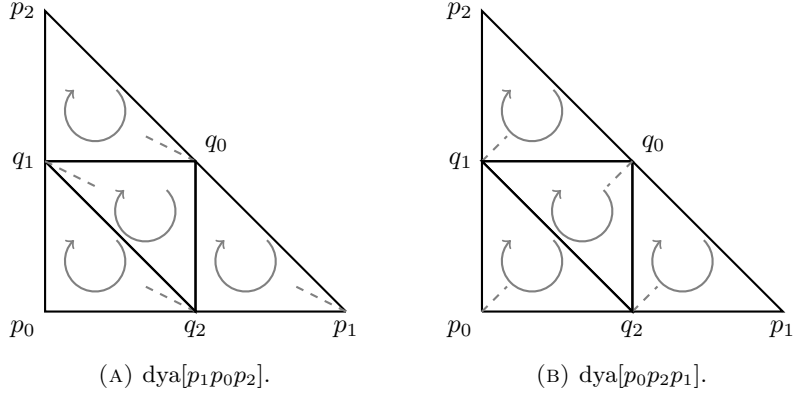
\begin{figure}
\begin{minipage}{0.45\textwidth}
\begin{center}
	\begin{tikzpicture}[scale=4]
  \drawchaininv{0.5}{0.5}{0}{0.5}{0}{1}{0.3}{-45}{-1}
   \drawchaininv{0.5}{0}{0}{0}{0}{0.5}{0.3}{-45}{-1}
     \drawchaininv{1}{0}{0.5}{0}{0.5}{0.5}{0.3}{-45}{-1}
   \drawchaininv{0}{0.5}{0.5}{0}{0.5}{0.5}{0.3}{-45}{-1}
	\draw   (0,0) node[below left]{$p_0$} (1,0) node[below]{$p_1$}
	(0,1) node[left]{$p_2$}
	(0.5,0.5) node[above right]{$q_0$} (0.5,0) node[below]{$q_2$}
	(0, 0.5) node[left]{$q_1$};
    \end{tikzpicture} \subcaption{ $ \dya [p_1 p_0 p_2]$.}\label{fig:dyadic-perm-1}%0 q_2 q_1]  + [q_2 p_1 q_0] + [q_1 q_0 p_2] + [q_0 q_1 q_2]$}
    \end{center}
\end{minipage}
\begin{minipage}{0.45\textwidth}
\begin{center}
	\begin{tikzpicture}[scale=4]
  \drawchaininv{0}{0.5}{0.5}{0.5}{0}{1}{0.3}{-45}{-1}
   \drawchaininv{0}{0}{0.5}{0}{0}{0.5}{0.3}{-45}{-1}
     \drawchaininv{0.5}{0}{1}{0}{0.5}{0.5}{0.3}{-45}{-1}
   \drawchaininv{0.5}{0.5}{0.5}{0}{0}{0.5}{0.3}{-45}{-1}
	\draw   (0,0) node[below left]{$p_0$} (1,0) node[below]{$p_1$}
	(0,1) node[left]{$p_2$}
	(0.5,0.5) node[above right]{$q_0$} (0.5,0) node[below]{$q_2$}
	(0, 0.5) node[left]{$q_1$};
    \end{tikzpicture} \subcaption{ $ \dya [p_0 p_2 p_1]$.}%0 q_2 q_1]  + [q_2 p_1 q_0] + [q_1 q_0 p_2] + [q_0 q_1 q_2]$}
    \label{fig:dyadic-perm-2}
    \end{center}
\end{minipage}
\caption{Dyadic decompositions of permutations of $[p_0 p_1 p_2]$.}
\end{figure}
\[ \begin{split} \dya \sigma[p_0 p_1 p_2]   & =   \dya [p_1 p_0 p_2]  = [p_{1} q_{2} q_{0}]  + [q_{2} p_{0} q_{1}] + [q_{0} q_{1} p_{2}] + [q_{1} q_{0} q_{2}]\\
& = \sigma \bra{[p_0 q_2 q_1]  + [q_2 p_1 q_0] + [q_1 q_0 p_2] + [q_0 q_1 q_2] } = \sigma \dya C .\end{split}\]
Similarly, if $\sigma[p_0p_1 p_2] = [p_2 p_1 p_0]$, then
\[ \begin{split} \dya \sigma[p_0 p_1 p_2]   & =  \dya [p_2 p_1 p_0]  = [p_{2} q_{0} q_{1}]  + [q_{1} p_{2} p_0] + [q_{1} q_{2} p_{0}] + [q_{2} q_{1} q_{0}]\\
& = \sigma\bra{[p_0 q_2 q_1]  + [q_2 p_1 q_0] + [q_1 q_0 p_2] + [q_0 q_1 q_2] } = \sigma \dya [p_0 p_1 p_2].\end{split}\]
Finally, if $\sigma[p_0p_1 p_2] = [p_0p_2 p_1]$, then (Figure~\ref{fig:dyadic-perm-2})
\[ \begin{split} \dya \sigma [p_0 p_1 p_2]    & =   \dya [p_0 p_2 p_1]  = [p_0 q_1 q_2] + [q_2 q_0 p_1] + [q_1 p_2 q_0] + [q_0 q_2 q_1]\\
& = \sigma\bra{[p_0 q_2 q_1]  + [q_2 p_1 q_0] + [q_1 q_0 p_2] + [q_0 q_1 q_2] } = \sigma [p_0 p_1 p_2] .\end{split} \]

To show \emph{ii)}, given $S =[p_0 p_1 p_2] \in \germ^2(\O)$, let $p_3 := p_1 + p_2 - p_0$ and $q = (p_1+p_2)/2$, $q_0=(p_2 +p_3)/2$, $q_1=(p_2+p_0)/2$, $q_2=(p_1+p_0)/2$, $q_3 = (p_1+ p_3)/2$. Then (Figure~\ref{fig:dyadic-para-p}),
\[ \begin{split} \dya \pi S & = \dya [p_0 p_1 p_2]+ \dya [p_3 p_2 p_1]\\% + E^\dya_{[p_0 p_1 p_2]} + E^\dya_{[p_3 p_2 p_1]} \\
					& = [p_0 q_2 q_1]  + [q_2 p_1 q] + [q_1 q p_2] + [q q_1 q_2]\\
					&  \qquad \qquad
					 + [p_3 q_0 q_3]  + [q_0 p_2 q] + [q_3 q p_1] + [q q_3 q_0] \\
					%& \quad +  E^\dya_{[p_0 p_1 p_2]} + E^\dya_{[p_3 p_2 p_1]}\\
					& = \bra{ [p_0 q_2 q_1]  +[q q_1 q_2] } + \bra{[q_2 p_1 q]+ [q_3 q p_1]}\\
					& \qquad \qquad
					 +\bra{[q_1 q p_2] + [q_0 p_2 q]} + \bra{ [q q_3 q_0] + [p_3 q_0 q_3]} \\
					%& \quad + E^\dya_{[p_0 p_1 p_2]} + E^\dya_{[p_3 p_2 p_1]}\\
					& =  \pi [p_0 q_2 q_1] +  \pi [q_2 p_1 q] +  \pi [q_1 q p_2] +\pi [q q_3 q_0]. %+ E^\dya_{[p_0 p_1 p_2]} + E^\dya_{[p_3 p_2 p_1]}.
					\end{split}\]
Similarly (see Figure~\ref{fig:dyadic-para-q}),
\[ \begin{split} \dya \tilde \pi S  & %= [p_1 p_3 p_0]+[p_2 p_0 p_3] \\
= \dya [p_1 p_3 p_0]+ \dya [p_2 p_0 p_3]\\% + E^\dya_{[p_1 p_3 p_0]} + E^\dya_{[p_2 p_0 p_3]} \\
					& = [p_1 q_3 q_2]  + [q_2  q p_0] + [q_3 p_3 q] + [q q_2 q_3] \\
					& \qquad \qquad
					 + [p_2 q_1 q_2]  + [q_1 p_0 q] + [q_0 q p_3] + [q q_0 q_1] \\
				%	& \quad + E^\dya_{[p_1 p_3 p_0]} + E^\dya_{[p_2 p_0 p_3]} \\
					& = \bra{  [q_2  q p_0] + [q_1 p_0 q]} + \bra{[p_1 q_3 q_2]  + [q q_2 q_3]}\\
					& \qquad \qquad+ \bra{[q q_0 q_1]+ [p_2 q_1 q_2]} + \bra{[q_3 p_3 q] +[q_0 q p_3] } \\
				%	& \quad + E^\dya_{[p_1 p_3 p_0]} + E^\dya_{[p_2 p_0 p_3]} \\
					& =  \tilde \pi [p_0 q_2 q_1] + \tilde \pi [q_2 p_1 q] +\tilde \pi [q_1 q p_2] + \tilde \pi [q q_3 q_0].% + E^\dya_{[p_0 p_1 p_2]} + E^\dya_{[p_3 p_2 p_1]}.
					\end{split}\]

\begin{figure}[ht]
\begin{minipage}{0.48\textwidth}
\begin{center}
	\begin{tikzpicture}[scale=4]
  \drawchain{0}{0.5}{0.5}{0.5}{0}{1}{0.3}{-135}{1}
   \drawchain{0}{0}{0.5}{0}{0}{0.5}{0.3}{-135}{1}
     \drawchain{0.5}{0}{1}{0}{0.5}{0.5}{0.3}{-135}{1}
   \drawchain{0.5}{0.5}{0.5}{0}{0}{0.5}{0.3}{-135}{1}
      \drawchain{0.5}{1}{0.5}{0.5}{0}{1}{0.3}{-135}{1}
   \drawchain{1}{0.5}{1}{0}{0.5}{0.5}{0.3}{-135}{1}
   \drawchain{1}{1}{1}{0.5}{0.5}{1}{0.3}{-135}{1}
     \drawchain{0.5}{0.5}{1}{0.5}{0.5}{1}{0.3}{-135}{1}
	\draw   (0,0) node[below left]{$p_0$} (1,0) node[below right]{$p_1$}
	(0,1) node[above left]{$p_2$}
	(0.55,0.4) node{$q$} (0.5,0) node[below]{$q_2$}
	(0, 0.5) node[left]{$q_1$} (1,1) node[above right]{$p_3$} (1,0.5) node[right]{$q_3$} (0.5,1) node[above]{$q_0$};
    \end{tikzpicture}\subcaption{ $ \dya \pi [p_0 p_1 p_2]$}\label{fig:dyadic-para-p}
    \end{center}
\end{minipage}
\begin{minipage}{0.48\textwidth}
\begin{center}
	\begin{tikzpicture}[scale=4]

   \drawchain{0.5}{0}{0.5}{0.5}{0}{0}{0.3}{-135}{1}
     \drawchain{0}{0.5}{0}{0}{0.5}{0.5}{0.3}{-135}{1}

      \drawchain{1}{0}{1}{0.5}{0.5}{0}{0.3}{-135}{1}
     \drawchain{0.5}{0.5}{0.5}{0}{1}{0.5}{0.3}{-135}{1}

       \drawchain{0.5}{0.5}{0.5}{1}{0}{0.5}{0.3}{-135}{1}
     \drawchain{0}{1}{0}{0.5}{0.5}{1}{0.3}{-135}{1}

         \drawchain{1}{0.5}{1}{1}{0.5}{0.5}{0.3}{-135}{1}
     \drawchain{0.5}{1}{0.5}{0.5}{1}{1}{0.3}{-135}{1}

	\draw   (0,0) node[below left]{$p_0$} (1,0) node[below right]{$p_1$}
	(0,1) node[above left]{$p_2$}
	(0.55,0.4) node{$q$} (0.5,0) node[below]{$q_2$}
	(0, 0.5) node[left]{$q_1$} (1,1) node[above right]{$p_3$} (1,0.5) node[right]{$q_3$} (0.5,1) node[above]{$q_0$};
    \end{tikzpicture}\subcaption{$\dya \pi^{\prime} [p_0 p_1 p_2]$}\label{fig:dyadic-para-q}
    \end{center}
\end{minipage}
\caption{Dyadic decompositions of $\pi [p_0 p_1 p_2]$ and $\tilde \pi [p_0p_1 p_2]$.}
\end{figure}
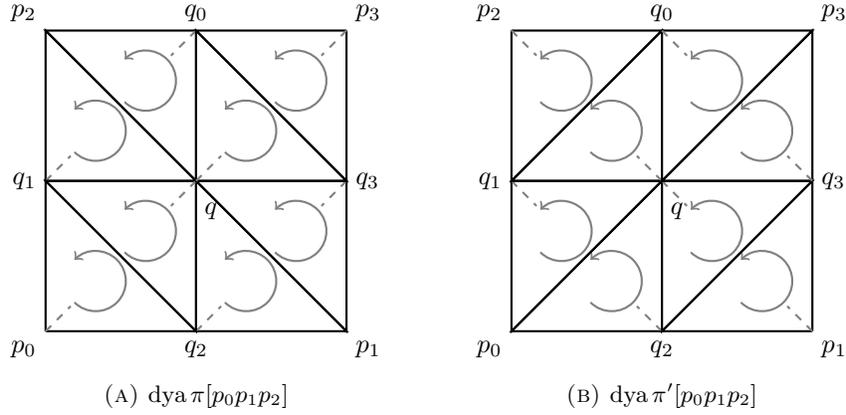

Hence, the thesis follows defining the geometric maps
\[ \tau^1S := [p_0 q_2 q_1], \quad \tau^2S := [q_2 p_1 q], \quad \tau^3S:= [q_1 q p_2], \quad \text{and} \quad \tau^4 S :=[q q_3 q_0].\]

The validity of \emph{iii)} is immediate by induction.

To show $\emph{iv)}$, we also argue by induction, the case $n=1$ being obvious and $n=2$ being proved as follows. For $S=[p_0 p_1 p_2]$, let $q_i = (p_{j} + p_{k})/2$, for $\cur{i,j, k} = \cur{0,1,2}$, and $r_i = (p_i +q_0)/2$ for $i \in \cur{0,1,2}$. Then, one has (Figure~\ref{fig:cutting-edge-2})
\[ \begin{split}  \dya \cut^2 S & = \dya [p_0p_1  q_0  ] + \dya [p_0 q_0 p_2 ] \\
& = \bra{[p_0 q_2 r_0] +[q_2 p_1 r_1 ]  +[r_1 r_0 q_2] +[r_0 r_1 q_0]} \\
& \quad +  \bra{ [p_0 r_0 q_1] +[q_1 r_2 p_2]   + [r_0 q_0 r_2] +[r_2 q_1 r_0]}\\
& = \bra{ [p_0 q_2 r_0] +[p_0 r_0 q_1] + [q_2 p_1 r_1 ] +[q_1 r_2 p_2] } \\
& \quad + \bra{ [q_0 r_0 q_2] + [q_2 r_1 q_0] } +  \bra{ [q_0 q_1 r_0] + [q_1 q_0 r_2]} \\
& \quad + \flip^\dagger ( [r_0 r_1 q_0] + [r_2 q_1 r_0] ) \\
& = \cut_{1/2} \dya S +   \flip^\dagger ( [r_0 r_1 q_0] + [r_2 q_1 r_0] )
\end{split}\]
hence the thesis follows letting $C :=[q_0 r_2 r_0] + [q_0 r_0 r_1]$.

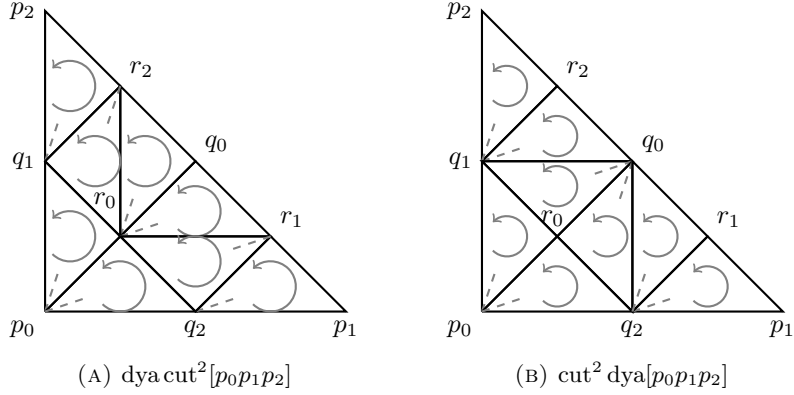
\begin{figure}
\begin{minipage}{.45\textwidth}\begin{center}
\begin{tikzpicture}[scale=4]
      \drawchain{0}{0}{0.25}{0.25}{0}{0.5}{0.25}{-135}{1}
   \drawchain{0.25}{0.25}{0.5}{0.5}{0.25}{0.75}{0.25}{-135}{1}
      \drawchain{0}{0.5}{0.25}{0.75}{0}{1}{0.25}{-135}{1}
      \drawchain{0.25}{0.75}{0}{0.5}{0.25}{0.25}{0.25}{-135}{1}

   \drawchain{0}{0}{0.25}{0.25}{0.5}{0}{0.25}{-135}{1}
      \drawchain{0.5}{0}{0.75}{0.25}{1}{0}{0.25}{-135}{1}
   \drawchain{0.25}{0.25}{0.5}{0.5}{0.75}{0.25}{0.25}{-135}{1}

   \drawchain{0.75}{0.25}{0.5}{0}{0.25}{0.25}{0.25}{-135}{1}

	\draw   (0,0) node[below left]{$p_0$} (1,0) node[below]{$p_1$} (0,1) node[left]{$p_2$}
	(0.5,0.5) node[above right]{$q_0$} (0.5,0) node[below]{$q_2$} (0, 0.5) node[left]{$q_1$}
	(0.25,0.75) node[above right]{$r_2$} (0.75, 0.25) node[above right]{$r_1$} (0.2, 0.37) node{$r_0$} ;
    \end{tikzpicture}\subcaption{ $ \dya \cut^2 [p_0 p_1 p_2]$}
\end{center}
\end{minipage}
\begin{minipage}{.45\textwidth}
\begin{center}
\begin{tikzpicture}[scale=4]
      \drawchain{0}{0}{0.25}{0.25}{0}{0.5}{0.2}{-135}{1}
   \drawchain{0}{0.5}{0.5}{0.5}{0.25}{0.75}{0.2}{-135}{1}
      \drawchain{0}{0.5}{0.25}{0.75}{0}{1}{0.2}{-135}{1}
      \drawchain{0.5}{0.5}{0.25}{0.25}{0}{0.5}{0.2}{-135}{1}

   \drawchain{0}{0}{0.25}{0.25}{0.5}{0}{0.2}{-135}{1}
      \drawchain{0.5}{0}{0.75}{0.25}{1}{0}{0.2}{-135}{1}
   \drawchain{0.5}{0.5}{0.25}{0.25}{0.5}{0}{0.2}{-135}{1}

   \drawchain{0.5}{0}{0.75}{0.25}{0.5}{0.5}{0.2}{-135}{1}

	\draw   (0,0) node[below left]{$p_0$} (1,0) node[below]{$p_1$} (0,1) node[left]{$p_2$}
	(0.5,0.5) node[above right]{$q_0$} (0.5,0) node[below]{$q_2$} (0, 0.5) node[left]{$q_1$}
	(0.25,0.75) node[above right]{$r_2$} (0.75, 0.25) node[above right]{$r_1$} (0.235, 0.31) node{$r_0$} ;
    \end{tikzpicture}\subcaption{ $\cut^2 \dya [p_0 p_1 p_2]$}
    \end{center}
\end{minipage}
\caption{Commutation between $\dya$ and $\cut^2$ using edge flipping.}\label{fig:cutting-edge-2}

\end{figure}

The inductive step then goes from $n-2$ to $n$. Given $S = [p p_0 p_1]$ (notice the slight change from notation of the previous case), define $q_0 = (p+p_0)/2$, $q_1 = (p+p_1)/2$, $p_t = (1-t) p_0 + t p_1$ and $q_t = (1-t) q_0 + t q_1$, for $t \in [0,1]$ (see Figures~\ref{fig:cut-n-2-pre} and~\ref{fig:cut-n-2}, with $n=5$).

\begin{figure}[ht]
\begin{minipage}{.45\textwidth}\begin{center}
\begin{tikzpicture}[scale=4]
      \drawchain{0}{0}{0.5}{0}{0.4}{0.1}{0.1}{-135}{1}
                     \drawchain{0.5}{0}{1}{0}{0.9}{0.1}{0.1}{-135}{1}
      \drawchain{0.4}{0.1}{0.9}{0.1}{0.8}{0.2}{0.1}{-135}{1}
      \drawchain{0.9}{0.1}{0.5}{0}{0.4}{0.1}{0.1}{-135}{1}

      \drawchain{0}{0}{0.4}{0.1}{0.3}{0.2}{0.1}{-135}{1}
       \drawchain{0}{0}{0.3}{0.2}{0.2}{0.3}{0.1}{-135}{1}
       \drawchain{0}{0}{0.2}{0.3}{0.1}{0.4}{0.1}{-135}{1}
       \drawchain{0}{0}{0.1}{0.4}{0}{0.5}{0.1}{-135}{1}
              \drawchain{0}{0.5}{0.1}{0.9}{0}{1}{0.1}{-135}{1}
       \drawchain{0.1}{0.4}{0.2}{0.8}{0.1}{0.9}{0.1}{-135}{1}
       \drawchain{0.1}{0.9}{0.1}{0.4}{0}{0.5}{0.1}{-135}{1}

            \drawchain{0.5}{0}{1}{0}{0.9}{0.1}{0.1}{-135}{1}
       \drawchain{0.7}{0.3}{0.4}{0.1}{0.3}{0.2}{0.1}{-135}{1}
              \drawchain{0.4}{0.1}{0.7}{0.3}{0.8}{0.2}{0.1}{-135}{1}
         \drawchain{0.3}{0.2}{0.7}{0.3}{0.6}{0.4}{0.1}{-135}{1}

               \drawchain{0.5}{0.5}{0.3}{0.2}{0.2}{0.3}{0.1}{-135}{1}
              \drawchain{0.3}{0.2}{0.6}{0.4}{0.5}{0.5}{0.1}{-135}{1}
         \drawchain{0.2}{0.3}{0.4}{0.6}{0.5}{0.5}{0.1}{-135}{1}

               \drawchain{0.3}{0.7}{0.1}{0.4}{0.2}{0.3}{0.1}{-135}{1}
              \drawchain{0.1}{0.4}{0.3}{0.7}{0.2}{0.8}{0.1}{-135}{1}
         \drawchain{0.2}{0.3}{0.3}{0.7}{0.4}{0.6}{0.1}{-135}{1}

%       \drawchain{0.4}{0.1}{0.8}{0.2}{0.7}{0.3}{0.1}{-135}{1}
%       \drawchain{0.4}{0.1}{0.7}{0.3}{0.6}{0.4}{0.1}{-135}{1}
%       \drawchain{0.4}{0.1}{0.6}{0.4}{0.5}{0.5}{0.1}{-135}{1}
%
%             \drawchain{0.1}{0.4}{0.5}{0.5}{0.4}{0.6}{0.1}{-135}{1}
%       \drawchain{0.1}{0.4}{0.4}{0.6}{0.3}{0.7}{0.1}{-135}{1}
%       \drawchain{0.1}{0.4}{0.3}{0.7}{0.2}{0.8}{0.1}{-135}{1}
%    %   \drawchain{0}{0.5}{0.2}{0.8}{0.1}{0.9}{0.1}{-135}{1}
%    %   \drawchain{0}{0.5}{0.1}{0.9}{0}{1}{0.1}{-135}{1}
%
%  %     \drawchain{0.5}{0.5}{0.5}{0}{0.4}{0.1}{0.1}{-135}{1}
%       \drawchain{0.5}{0.5}{0.4}{0.1}{0.3}{0.2}{0.1}{-135}{1}
%       \drawchain{0.5}{0.5}{0.3}{0.2}{0.2}{0.3}{0.1}{-135}{1}
%       \drawchain{0.5}{0.5}{0.2}{0.3}{0.1}{0.4}{0.1}{-135}{1}
%    %   \drawchain{0.5}{0.5}{0.1}{0.4}{0}{0.5}{0.1}{-135}{1}
%
	\draw   (0,0) node[below left]{$p$} (1,0) node[below]{$p_0$} (0,1) node[left]{$p_1$}
	%(0.9,0.1) node[above right]{$p_{1/10}$}	
	%(0.8,0.2) node[above right]{$p_{1/5}$}
	%(0.1,0.9) node[above right]{$p_{9/10}$}
	%	(0.2,0.8) node[above right]{$p_{4/5}$}
	%(0.4,0.1) node[above]{$q_{1/{5}}$}	(0.1,0.4) node[right]{$q_{4/{5}}$}
	(0.5,0) node[below]{$q_0$}	(0,0.5) node[left]{$q_1$};
    \end{tikzpicture}\subcaption{$\dya \cut^5 [p p_0 p_{1}]$}    \label{fig:cut-n-2-pre}
    \end{center}
\end{minipage}
    \begin{minipage}{.45\textwidth}
    \begin{center}
\begin{tikzpicture}[scale=4]
     \drawchain{0}{0}{0.5}{0}{0.4}{0.1}{0.1}{-135}{1}
      \drawchain{0}{0}{0.4}{0.1}{0.3}{0.2}{0.1}{-135}{1}
       \drawchain{0}{0}{0.3}{0.2}{0.2}{0.3}{0.1}{-135}{1}
       \drawchain{0}{0}{0.2}{0.3}{0.1}{0.4}{0.1}{-135}{1}
      \drawchain{0}{0}{0.1}{0.4}{0}{0.5}{0.1}{-135}{1}

       \drawchain{0.5}{0}{1}{0}{0.9}{0.1}{0.1}{-135}{1}
       \drawchain{0.5}{0}{0.9}{0.1}{0.8}{0.2}{0.1}{-135}{1}
       \drawchain{0.5}{0}{0.8}{0.2}{0.7}{0.3}{0.1}{-135}{1}
       \drawchain{0.5}{0}{0.7}{0.3}{0.6}{0.4}{0.1}{-135}{1}
       \drawchain{0.5}{0}{0.6}{0.4}{0.5}{0.5}{0.1}{-135}{1}

             \drawchain{0}{0.5}{0.5}{0.5}{0.4}{0.6}{0.1}{-135}{1}
       \drawchain{0}{0.5}{0.4}{0.6}{0.3}{0.7}{0.1}{-135}{1}
       \drawchain{0}{0.5}{0.3}{0.7}{0.2}{0.8}{0.1}{-135}{1}
       \drawchain{0}{0.5}{0.2}{0.8}{0.1}{0.9}{0.1}{-135}{1}
      \drawchain{0}{0.5}{0.1}{0.9}{0}{1}{0.1}{-135}{1}

       \drawchain{0.5}{0.5}{0.5}{0}{0.4}{0.1}{0.1}{-135}{1}
       \drawchain{0.5}{0.5}{0.4}{0.1}{0.3}{0.2}{0.1}{-135}{1}
       \drawchain{0.5}{0.5}{0.3}{0.2}{0.2}{0.3}{0.1}{-135}{1}
       \drawchain{0.5}{0.5}{0.2}{0.3}{0.1}{0.4}{0.1}{-135}{1}
       \drawchain{0.5}{0.5}{0.1}{0.4}{0}{0.5}{0.1}{-135}{1}

	\draw   (0,0) node[below left]{$p$} (1,0) node[below]{$p_0$} (0,1) node[left]{$p_1$}
%	(0.8,0.2) node[above right]{$p_{1/5}$}	(0.6,0.4) node[above right]{$p_{2/5}$}
%	(0.4,0.6) node[above right]{$p_{3/5}$}	(0.2,0.8) node[above right]{$p_{4/5}$}
	(0.5,0) node[below]{$q_0$}	(0,0.5) node[left]{$q_1$};
    \end{tikzpicture}\subcaption{ $\cut^5 \dya [p p_0 p_1]$} \label{fig:cut-n-2}
    \end{center}

    \end{minipage}

\end{figure}

Since  \[ \cut^n [p p_0 p_1] = [p p_0 p_{\frac 1 n}] +  \cut^{n-2} [p p_{\frac 1n} p_{\frac{n-1}{n}}] + [p p_{\frac {n-1}{n}} p_1],\]
we have (Figures~\ref{fig:cut-n-3} and~\ref{fig:cut-n-4})
\[ \dya \cut^n [p p_0 p_1] = \dya \bra{ [p p_0 p_{\frac 1 n}]  + [p p_{\frac {n-1}{n}} p_1]} + \cut^{n-2} \dya  [p p_{\frac 1n} p_{\frac{n-1}{n}}] + \flip^\dagger C^{n-2},\]
where $C^{n-2} \in \chain^{2}(\O)$.

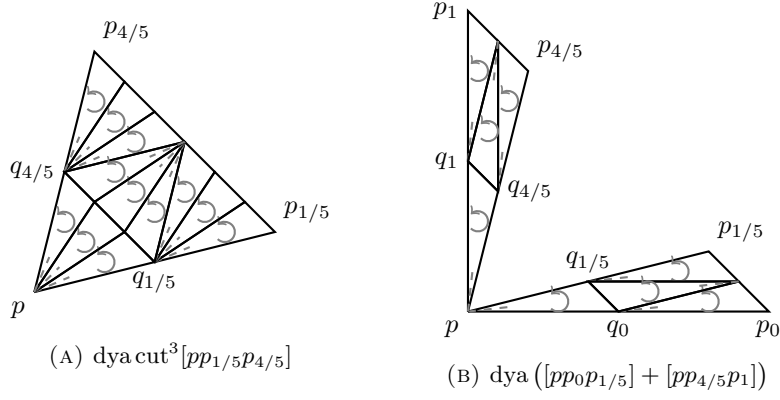
\begin{figure}[ht]
    \begin{minipage}{.45\textwidth}
    \begin{center}
\begin{tikzpicture}[scale=4]
     % \drawchain{0}{0}{0.5}{0}{0.4}{0.1}{0.1}{-135}{1}
      \drawchain{0}{0}{0.4}{0.1}{0.3}{0.2}{0.1}{-135}{1}
       \drawchain{0}{0}{0.3}{0.2}{0.2}{0.3}{0.1}{-135}{1}
       \drawchain{0}{0}{0.2}{0.3}{0.1}{0.4}{0.1}{-135}{1}
     %  \drawchain{0}{0}{0.1}{0.4}{0}{0.5}{0.1}{-135}{1}

    %        \drawchain{0.5}{0}{1}{0}{0.9}{0.1}{0.1}{-135}{1}
    %   \drawchain{0.5}{0}{0.9}{0.1}{0.8}{0.2}{0.1}{-135}{1}
       \drawchain{0.4}{0.1}{0.8}{0.2}{0.7}{0.3}{0.1}{-135}{1}
       \drawchain{0.4}{0.1}{0.7}{0.3}{0.6}{0.4}{0.1}{-135}{1}
       \drawchain{0.4}{0.1}{0.6}{0.4}{0.5}{0.5}{0.1}{-135}{1}

             \drawchain{0.1}{0.4}{0.5}{0.5}{0.4}{0.6}{0.1}{-135}{1}
       \drawchain{0.1}{0.4}{0.4}{0.6}{0.3}{0.7}{0.1}{-135}{1}
       \drawchain{0.1}{0.4}{0.3}{0.7}{0.2}{0.8}{0.1}{-135}{1}
    %   \drawchain{0}{0.5}{0.2}{0.8}{0.1}{0.9}{0.1}{-135}{1}
    %   \drawchain{0}{0.5}{0.1}{0.9}{0}{1}{0.1}{-135}{1}

  %     \drawchain{0.5}{0.5}{0.5}{0}{0.4}{0.1}{0.1}{-135}{1}
       \drawchain{0.5}{0.5}{0.4}{0.1}{0.3}{0.2}{0.1}{-135}{1}
       \drawchain{0.5}{0.5}{0.3}{0.2}{0.2}{0.3}{0.1}{-135}{1}
       \drawchain{0.5}{0.5}{0.2}{0.3}{0.1}{0.4}{0.1}{-135}{1}
    %   \drawchain{0.5}{0.5}{0.1}{0.4}{0}{0.5}{0.1}{-135}{1}

	\draw   (0,0) node[below left]{$p$} %(1,0) node[below]{$q_0^{10}$} (0,1) node[left]{$q_{10}^{10}$}
	(0.8,0.2) node[above right]{$p_{1/5}$}	%(0.6,0.4) %node[above right]{$p_{2/5}$}
%	(0.4,0.6) node[above right]{$p_{3/5}$}
	(0.2,0.8) node[above right]{$p_{4/5}$}
	(0.4,0.1) node[below]{$q_{1/5}$}	(0.1,0.4) node[left]{$q_{4/5}$};
    \end{tikzpicture}\subcaption{ $\dya \cut^3 [p p_{1/5} p_{4/5}]$} \label{fig:cut-n-3}
    \end{center}
\end{minipage}
 \begin{minipage}{.45\textwidth}\begin{center}
\begin{tikzpicture}[scale=4]
      \drawchain{0}{0}{0.5}{0}{0.4}{0.1}{0.1}{-135}{1}
                     \drawchain{0.5}{0}{1}{0}{0.9}{0.1}{0.1}{-135}{1}
      \drawchain{0.4}{0.1}{0.9}{0.1}{0.8}{0.2}{0.1}{-135}{1}
      \drawchain{0.9}{0.1}{0.5}{0}{0.4}{0.1}{0.1}{-135}{1}

   %   \drawchain{0}{0}{0.4}{0.1}{0.3}{0.2}{0.1}{-135}{1}
   %    \drawchain{0}{0}{0.3}{0.2}{0.2}{0.3}{0.1}{-135}{1}
   %    \drawchain{0}{0}{0.2}{0.3}{0.1}{0.4}{0.1}{-135}{1}
       \drawchain{0}{0}{0.1}{0.4}{0}{0.5}{0.1}{-135}{1}
              \drawchain{0}{0.5}{0.1}{0.9}{0}{1}{0.1}{-135}{1}
       \drawchain{0.1}{0.4}{0.2}{0.8}{0.1}{0.9}{0.1}{-135}{1}
       \drawchain{0.1}{0.9}{0.1}{0.4}{0}{0.5}{0.1}{-135}{1}

%            \drawchain{0.5}{0}{1}{0}{0.9}{0.1}{0.1}{-135}{1}
%       \drawchain{0.5}{0}{0.9}{0.1}{0.8}{0.2}{0.1}{-135}{1}
%       \drawchain{0.4}{0.1}{0.8}{0.2}{0.7}{0.3}{0.1}{-135}{1}
%       \drawchain{0.4}{0.1}{0.7}{0.3}{0.6}{0.4}{0.1}{-135}{1}
%       \drawchain{0.4}{0.1}{0.6}{0.4}{0.5}{0.5}{0.1}{-135}{1}
%
%             \drawchain{0.1}{0.4}{0.5}{0.5}{0.4}{0.6}{0.1}{-135}{1}
%       \drawchain{0.1}{0.4}{0.4}{0.6}{0.3}{0.7}{0.1}{-135}{1}
%       \drawchain{0.1}{0.4}{0.3}{0.7}{0.2}{0.8}{0.1}{-135}{1}
%    %   \drawchain{0}{0.5}{0.2}{0.8}{0.1}{0.9}{0.1}{-135}{1}
%    %   \drawchain{0}{0.5}{0.1}{0.9}{0}{1}{0.1}{-135}{1}
%
%  %     \drawchain{0.5}{0.5}{0.5}{0}{0.4}{0.1}{0.1}{-135}{1}
%       \drawchain{0.5}{0.5}{0.4}{0.1}{0.3}{0.2}{0.1}{-135}{1}
%       \drawchain{0.5}{0.5}{0.3}{0.2}{0.2}{0.3}{0.1}{-135}{1}
%       \drawchain{0.5}{0.5}{0.2}{0.3}{0.1}{0.4}{0.1}{-135}{1}
%    %   \drawchain{0.5}{0.5}{0.1}{0.4}{0}{0.5}{0.1}{-135}{1}
%
	\draw   (0,0) node[below left]{$p$} (1,0) node[below]{$p_0$} (0,1) node[left]{$p_1$}
	%(0.9,0.1) node[above right]{$p_{1/10}$}	
	(0.8,0.2) node[above right]{$p_{1/5}$}
	%(0.1,0.9) node[above right]{$p_{9/10}$}
		(0.2,0.8) node[above right]{$p_{4/5}$}
	(0.4,0.1) node[above]{$q_{1/{5}}$}	(0.1,0.4) node[right]{$q_{4/{5}}$}
	(0.5,0) node[below]{$q_0$}	(0,0.5) node[left]{$q_1$};
    \end{tikzpicture}\subcaption{$\dya\bra{[p p_0 p_{1/5}] + [p p_{4/5} p_1]}$}\label{fig:cut-n-4}
    \end{center}
\end{minipage}
\caption{Two chains in the inductive step with $n=5$.}
\end{figure}

Collecting the chains $[p q_0 q_{\frac{1}{n}}]+ [p q_{\frac{n-1}{n} }q_1]$ and $\cut^{n-2} [p q_{\frac 1 n} q_{\frac {n-1}{n}}]$ which can be found in the expression above, we obtain $\cut^{n} [p q_0 q_1]$, hence the thesis follows if we prove that the difference between the chain
\[\begin{split} R^1 := & [q_0 p_0 p_{\frac 1 {2n}}] + [p_{\frac 1 {2n}} q_{\frac 1 n} q_0] + [q_{\frac 1 n} p_{\frac 1 {2n}} p_{\frac 1 n}] \\
& \quad + \cut^{n-2} \bra{ [q_{\frac 1 n} p_{\frac 1 {n}} p_{\frac 1 2 }] + [p_{\frac 1 2 } q_{\frac {n-1}{n} }q_{\frac 1 n}] + [q_{\frac {n-1}{n}} p_{\frac 1 2 } p_{\frac {n-1}{n}}]} \\
& \quad + [q_1  p_{\frac {2n-1} {2n}}p_1] + [p_{\frac {2n-1} {2n}} q_{\frac {n-1} n} q_1] + [q_{\frac {n-1} n} p_{\frac {2n-1} {2n}} p_{\frac {n-1} n}]
\end{split}\]
and the chain (Figure~\ref{fig:cut-n-dyadic-5})
\[ R^{n} := \cut^{n} \bra{ [q_0 p_0 p_{\frac 1 2 }] + [p_{\frac 1 2 } q_1 q_0] + [q_1 p_{\frac 1 2 } p_1]}\]
can be written in the form $\flip^\dagger C$, for some $C \in \chain^2(\O)$.

 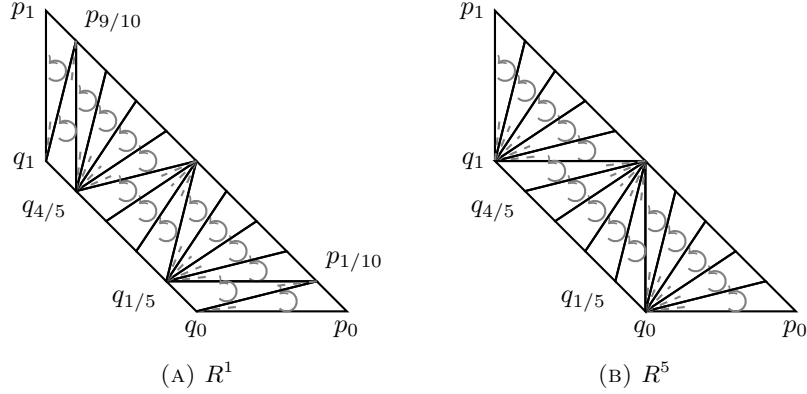
\begin{figure}[ht]
 \begin{minipage}{.45\textwidth}\begin{center}
\begin{tikzpicture}[scale=4]
  %    \drawchain{0}{0}{0.5}{0}{0.4}{0.1}{0.1}{-135}{1}
                     \drawchain{0.5}{0}{1}{0}{0.9}{0.1}{0.1}{-135}{1}
      \drawchain{0.4}{0.1}{0.9}{0.1}{0.8}{0.2}{0.1}{-135}{1}
      \drawchain{0.9}{0.1}{0.5}{0}{0.4}{0.1}{0.1}{-135}{1}

    \drawchain{0.4}{0.1}{0.8}{0.2}{0.7}{0.3}{0.1}{-135}{1}
       \drawchain{0.4}{0.1}{0.7}{0.3}{0.6}{0.4}{0.1}{-135}{1}
       \drawchain{0.4}{0.1}{0.6}{0.4}{0.5}{0.5}{0.1}{-135}{1}

             \drawchain{0.1}{0.4}{0.5}{0.5}{0.4}{0.6}{0.1}{-135}{1}
       \drawchain{0.1}{0.4}{0.4}{0.6}{0.3}{0.7}{0.1}{-135}{1}
       \drawchain{0.1}{0.4}{0.3}{0.7}{0.2}{0.8}{0.1}{-135}{1}
    %   \drawchain{0}{0.5}{0.2}{0.8}{0.1}{0.9}{0.1}{-135}{1}
    %   \drawchain{0}{0.5}{0.1}{0.9}{0}{1}{0.1}{-135}{1}

  %     \drawchain{0.5}{0.5}{0.5}{0}{0.4}{0.1}{0.1}{-135}{1}
       \drawchain{0.5}{0.5}{0.4}{0.1}{0.3}{0.2}{0.1}{-135}{1}
       \drawchain{0.5}{0.5}{0.3}{0.2}{0.2}{0.3}{0.1}{-135}{1}
       \drawchain{0.5}{0.5}{0.2}{0.3}{0.1}{0.4}{0.1}{-135}{1}

     %  \drawchain{0}{0}{0.1}{0.4}{0}{0.5}{0.1}{-135}{1}
              \drawchain{0}{0.5}{0.1}{0.9}{0}{1}{0.1}{-135}{1}
       \drawchain{0.1}{0.4}{0.2}{0.8}{0.1}{0.9}{0.1}{-135}{1}
       \drawchain{0.1}{0.9}{0.1}{0.4}{0}{0.5}{0.1}{-135}{1}

	\draw  % (0,0) node[below left]{$p_0$} (1,0)
	(1,0) node[below]{$p_0$} (0,1) node[left]{$p_1$}
	(0.9,0.1) node[above right]{$p_{1/10}$}	%(0.8,0.2) node[above right]{$p_{1/5}$}
	(0.1,0.9) node[above right]{$p_{9/10}$}	%(0.2,0.8) node[above right]{$p_{2/5}$}
	(0.4,0.1) node[below left]{$q_{1/5}$}	(0.1,0.4) node[below left]{$q_{4/5}$}
	(0.5,0) node[below]{$q_0$}	(0,0.5) node[left]{$q_1$};
    \end{tikzpicture}\subcaption{$R^1$}
   \end{center}
\end{minipage}
     \begin{minipage}{.45\textwidth}\begin{center}
\begin{tikzpicture}[scale=4]

       \drawchain{0.5}{0}{1}{0}{0.9}{0.1}{0.1}{-135}{1}
       \drawchain{0.5}{0}{0.9}{0.1}{0.8}{0.2}{0.1}{-135}{1}
       \drawchain{0.5}{0}{0.8}{0.2}{0.7}{0.3}{0.1}{-135}{1}
       \drawchain{0.5}{0}{0.7}{0.3}{0.6}{0.4}{0.1}{-135}{1}
       \drawchain{0.5}{0}{0.6}{0.4}{0.5}{0.5}{0.1}{-135}{1}

             \drawchain{0}{0.5}{0.5}{0.5}{0.4}{0.6}{0.1}{-135}{1}
       \drawchain{0}{0.5}{0.4}{0.6}{0.3}{0.7}{0.1}{-135}{1}
       \drawchain{0}{0.5}{0.3}{0.7}{0.2}{0.8}{0.1}{-135}{1}
       \drawchain{0}{0.5}{0.2}{0.8}{0.1}{0.9}{0.1}{-135}{1}
      \drawchain{0}{0.5}{0.1}{0.9}{0}{1}{0.1}{-135}{1}

       \drawchain{0.5}{0.5}{0.5}{0}{0.4}{0.1}{0.1}{-135}{1}
       \drawchain{0.5}{0.5}{0.4}{0.1}{0.3}{0.2}{0.1}{-135}{1}
       \drawchain{0.5}{0.5}{0.3}{0.2}{0.2}{0.3}{0.1}{-135}{1}
       \drawchain{0.5}{0.5}{0.2}{0.3}{0.1}{0.4}{0.1}{-135}{1}
       \drawchain{0.5}{0.5}{0.1}{0.4}{0}{0.5}{0.1}{-135}{1}

	\draw   %(0,0) node[below left]{$p_0$}
	(1,0) node[below]{$p_0$} (0,1) node[left]{$p_1$}
	%(0.9,0.1) node[above right]{$p_{1/10}$}	(0.8,0.2) node[above right]{$p_{1/5}$}
	%(0.1,0.9) node[above right]{$p_{9/10}$}	(0.2,0.8) node[above right]{$p_{2/5}$}
	(0.4,0.1) node[below left]{$q_{1/5}$}	(0.1,0.4) node[below left]{$q_{4/5}$}
	(0.5,0) node[below]{$q_0$}	(0,0.5) node[left]{$q_1$};
    \end{tikzpicture}\subcaption{$R^5$}
    \end{center}
\end{minipage}
\caption{The chains $R^1$ and $R^5$ in the case $n=5$.}\label{fig:cut-n-dyadic-5}
\end{figure}
 %Actually, we argue that the difference belongs to $\chain^\f(\R^d)$: indeed, we claim that both
% \[ [q^{2n}_1 r^{n}_{1} r^{n}_{0} ] +  \sum_{k=2}^{n} [r^n_1 q^{2n}_{k-1} q^{2n}_{k}] -   \sum_{k=2}^{n} [r^n_0 q^{2n}_{k-1} q^{2n}_{k}]  - [q^{2n}_{n} r^{n}_{1} r^n_{0}]\]
% and
% \[ [q^{2n}_{2n-1} r^{n}_{n} r^{n}_{n-1} ]+ \sum_{k=n+1}^{2n} [r^n_{n} q^{2n}_{k-1} q^{n}_{k}] - \sum_{k=n+1}^{2n} [r^n_{n} q^{2n}_{k-1} q^{n}_{k}] - [q^{2n}_{n} r^{n}_{n} r^n_{n-1}]\]
% belong to $\chain^\f(\R^d)$.

We achieve this by means of $(n-1)$ consecutive edge-flippings, i.e. building an ``interpolating'' sequence of chains $(R^k)_{k=1}^{n}$ such that $R^{k+1} - R^k =  \flip\dagger C^k$, for some $C^k \in \chain^2(\O)$, for $k = 1, \ldots, n-1$. Precisely, let (Figures~\ref{fig:cut-n-dyadic-5} and \ref{fig:cut-n-dyadic-6})
\[ \begin{split}  R^k := &  \cut^k \bra{ [ q_0 p_0 p_{\frac{k}{2n}}] +  [ q_1 p_{\frac{2n-k}{2n}} p_1 ] } +   \cut^{n-k-1} \bra{ [ q_{\frac 1 n} p_{\frac{k+1}{2n}} p_{\frac 1 2}] +  [ q_{\frac {n-1}{n}} p_{ \frac 1 2}  p_{\frac{2n-k-1}{2n}} ] }\\
& \quad + [p_{\frac{k}{2n}} q_{\frac 1 n} q_0]  + [p_{\frac{2n-k}{2n}} q_1 q_{\frac {n-1}{n}}] +  \cut^{n-2}\bra{ [p_{\frac 1 2 } q_{\frac{n-1}{n}} q_{\frac 1 n}]}.\\
\end{split}\]
and notice that $R^1$ and $R^{n}$ coincide with the one above (just recollecting some terms).

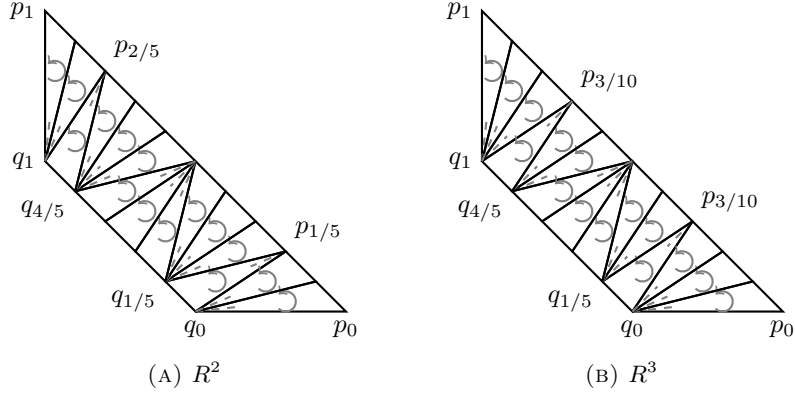
\begin{figure}[ht]
 \begin{minipage}{.45\textwidth}\begin{center}
\begin{tikzpicture}[scale=4]
  %    \drawchain{0}{0}{0.5}{0}{0.4}{0.1}{0.1}{-135}{1}
                     \drawchain{0.5}{0}{1}{0}{0.9}{0.1}{0.1}{-135}{1}
      \drawchain{0.8}{0.2}{0.4}{0.1}{0.5}{0}{0.1}{-135}{1}
      \drawchain{0.5}{0}{0.9}{0.1}{0.8}{0.2}{0.1}{-135}{1}

    \drawchain{0.4}{0.1}{0.8}{0.2}{0.7}{0.3}{0.1}{-135}{1}
       \drawchain{0.4}{0.1}{0.7}{0.3}{0.6}{0.4}{0.1}{-135}{1}
       \drawchain{0.4}{0.1}{0.6}{0.4}{0.5}{0.5}{0.1}{-135}{1}

             \drawchain{0.1}{0.4}{0.5}{0.5}{0.4}{0.6}{0.1}{-135}{1}
       \drawchain{0.1}{0.4}{0.4}{0.6}{0.3}{0.7}{0.1}{-135}{1}
       \drawchain{0.1}{0.4}{0.3}{0.7}{0.2}{0.8}{0.1}{-135}{1}
    %   \drawchain{0}{0.5}{0.2}{0.8}{0.1}{0.9}{0.1}{-135}{1}
    %   \drawchain{0}{0.5}{0.1}{0.9}{0}{1}{0.1}{-135}{1}

  %     \drawchain{0.5}{0.5}{0.5}{0}{0.4}{0.1}{0.1}{-135}{1}
       \drawchain{0.5}{0.5}{0.4}{0.1}{0.3}{0.2}{0.1}{-135}{1}
       \drawchain{0.5}{0.5}{0.3}{0.2}{0.2}{0.3}{0.1}{-135}{1}
       \drawchain{0.5}{0.5}{0.2}{0.3}{0.1}{0.4}{0.1}{-135}{1}

     %  \drawchain{0}{0}{0.1}{0.4}{0}{0.5}{0.1}{-135}{1}
              \drawchain{0}{0.5}{0.1}{0.9}{0}{1}{0.1}{-135}{1}
    \drawchain{0.2}{0.8}{0}{0.5}{0.1}{0.4}{0.1}{-135}{1}
     \drawchain{0}{0.5}{0.2}{0.8}{0.1}{0.9}{0.1}{-135}{1}

	\draw  % (0,0) node[below left]{$p_0$} (1,0)
	(1,0) node[below]{$p_0$} (0,1) node[left]{$p_1$}
	%(0.9,0.1) node[above right]{$p_{1/10}$}
		(0.8,0.2) node[above right]{$p_{1/5}$}
	%(0.1,0.9) node[above right]{$p_{9/10}$}	
	(0.2,0.8) node[above right]{$p_{2/5}$}
	(0.4,0.1) node[below left]{$q_{1/5}$}	(0.1,0.4) node[below left]{$q_{4/5}$}
	(0.5,0) node[below]{$q_0$}	(0,0.5) node[left]{$q_1$};
    \end{tikzpicture}\subcaption{$R^2$}
   \end{center}
\end{minipage}
     \begin{minipage}{.45\textwidth}\begin{center}
\begin{tikzpicture}[scale=4]
  %    \drawchain{0}{0}{0.5}{0}{0.4}{0.1}{0.1}{-135}{1}
                     \drawchain{0.5}{0}{1}{0}{0.9}{0.1}{0.1}{-135}{1}
      \drawchain{0.7}{0.3}{0.4}{0.1}{0.5}{0}{0.1}{-135}{1}
      \drawchain{0.5}{0}{0.9}{0.1}{0.8}{0.2}{0.1}{-135}{1}

  \drawchain{0.5}{0}{0.8}{0.2}{0.7}{0.3}{0.1}{-135}{1}
       \drawchain{0.4}{0.1}{0.7}{0.3}{0.6}{0.4}{0.1}{-135}{1}
       \drawchain{0.4}{0.1}{0.6}{0.4}{0.5}{0.5}{0.1}{-135}{1}

             \drawchain{0.1}{0.4}{0.5}{0.5}{0.4}{0.6}{0.1}{-135}{1}
       \drawchain{0.1}{0.4}{0.4}{0.6}{0.3}{0.7}{0.1}{-135}{1}
       \drawchain{0}{0.5}{0.3}{0.7}{0.2}{0.8}{0.1}{-135}{1}
    %   \drawchain{0}{0.5}{0.2}{0.8}{0.1}{0.9}{0.1}{-135}{1}
    %   \drawchain{0}{0.5}{0.1}{0.9}{0}{1}{0.1}{-135}{1}

  %     \drawchain{0.5}{0.5}{0.5}{0}{0.4}{0.1}{0.1}{-135}{1}
       \drawchain{0.5}{0.5}{0.4}{0.1}{0.3}{0.2}{0.1}{-135}{1}
       \drawchain{0.5}{0.5}{0.3}{0.2}{0.2}{0.3}{0.1}{-135}{1}
       \drawchain{0.5}{0.5}{0.2}{0.3}{0.1}{0.4}{0.1}{-135}{1}

     %  \drawchain{0}{0}{0.1}{0.4}{0}{0.5}{0.1}{-135}{1}
              \drawchain{0}{0.5}{0.1}{0.9}{0}{1}{0.1}{-135}{1}
    \drawchain{0.3}{0.7}{0}{0.5}{0.1}{0.4}{0.1}{-135}{1}
     \drawchain{0}{0.5}{0.2}{0.8}{0.1}{0.9}{0.1}{-135}{1}

	\draw  % (0,0) node[below left]{$p_0$} (1,0)
	(1,0) node[below]{$p_0$} (0,1) node[left]{$p_1$}
	%(0.9,0.1) node[above right]{$p_{1/10}$}
		(0.7,0.3) node[above right]{$p_{3/10}$}
	%(0.1,0.9) node[above right]{$p_{9/10}$}	
	(0.3,0.7) node[above right]{$p_{3/10}$}
	(0.4,0.1) node[below left]{$q_{1/5}$}	(0.1,0.4) node[below left]{$q_{4/5}$}
	(0.5,0) node[below]{$q_0$}	(0,0.5) node[left]{$q_1$};
    \end{tikzpicture}\subcaption{$R^3$}
    \end{center}
\end{minipage}
\caption{The chains $R^2$ and $R^3$ in the case $n=5$.}\label{fig:cut-n-dyadic-6}
\end{figure}

Then, one has
\[ \begin{split} R^{k+1} - R^k & =  [p_{\frac{k+1}{2n}} q_{\frac 1 n} q_0]  +[q_{\frac{1}{n}}p_{\frac{k}{2n}} p_{\frac {k+1}{2n}}] - [p_{\frac{k}{2n}} q_{\frac 1 n} q_0] - [q_{\frac 1 n} p_{\frac k {2n}} p_{ \frac {k+1}{2n}}] \\
& \quad + [q_{1}p_{\frac {2n-k-1}{2n}}p_{\frac{2n-k}{2n}} ] +  [p_{\frac{2n-k-1}{2n}} q_1 q_{\frac {n-1} n} ] \\
& \quad  - [p_{\frac{2n-k}{2n}} q_{1} q_{\frac {n-1}{n}}] - [q_{\frac {n-1} n} p_{\frac {2n-k-1} {2n}} p_{ \frac {2n-k}{2n}}] \\
&=  \flip ^\dagger \bra{ [q_{\frac 1 n} q_0 p_{\frac{k}{2n}} ] + [q_1 q_{\frac {n-1}{n}} p_{\frac{2n-k}{2n}}]}
\end{split}\]
and the thesis follows.
%
%
% Informally, we iteratively use edge flipping the parallelograms
%\[ [p_{1/(2n)} q_{1/n} q_0] + [q_{1/n} p_{1/(2n)} p_{1/n}], \quad [p_{(2n-1)/(2n)} q_{(n-1)/n} q_1] + [q_{(n-1)/n} p_{(2n-1)/(2n)} p_{(n-1)/n}]\]
%and then flipping the two new parallelograms ``closer to $p_{1/2}$'', thus formed and so on until flip operations are possible. In a formal way, we argue separately for the two sides, e.g.\ for the side closer to $p_0$,
% \[ \begin{split}  [p_{1/(2n)} q_{1/n} q_0] & +  \sum_{k=2}^{n} [q_{1/n} p_{(k-1)/(2n)} p_{k/(2n)}] = \\
% & = [q_0 p_{1/(2n)} p_{1/n}] + [p_{1/n} q_{1/n} q_0] + \sum_{k=3}^{n}  [q_{1/n} p_{(k-1)/(2n)} p_{k/(2n)}] + E^\f_{ [p_{1/(2n)} q_{1/n} q_0] }\\
% & = \quad \text{(iterating\ldots)}\\
% & = \sum_{k=2}^{n} [q_0 q_{(k-1)/(2n)} q_{k/(2n)}]  + [p_{1/2} q_{1/n} q_0] + \sum_{k=1}^{n-1} E^{\f}_{ [p_{k/(2n)} q_{1/n} q_0] }.
%\end{split}\]
%and similarly, for the side closer to $p_1$. %where $E^\f_{k}$ are termsSimilarly, one establishes the second claim.
\end{proof}

%
%\begin{remark}\label{rem:cut-dya-1-simplex}
%Notice that analogous properties hold, with a straightforward proof, for $\cut_{1/2}$ acting on $1$-simplices. In particular, $\cut^n \cut_{1/2} = \cut_{1/2} \cut^n$ for every $n \ge 1$.
%\end{remark}

\section{Proof of sewing results}\label{appendix-sewing}

In this section we collect the proofs of %Theorem~\ref{thm:uniqueness-regular},
Theorem~\ref{thm:sew-unique},
Theorem~\ref{thm:stokes} %, Theorem~\ref{thm:stokes-2}
and Theorem~\ref{thm:sew-existence}, together with auxiliary results. We begin with a useful uniqueness result.

\begin{lemma}[uniqueness]\label{lem:uniqueness-general}
Let $\ell \ge 1$, $\lambda_i \in [-1,1]$, $\tau_i: \simp^k(\O) \to \simp^k(\O)$, $i \in\cur{1, \ldots, \ell}$, be such that  $\tau_i S$ is  isometric to $(2^{-1})_\natural S$, for every $S \in \simp^k(\O)$ and let $\tau := \sum_{i=1}^\ell \lambda_i \tau_i$. Let $\omega \in \germ^k(\O)$ and $\v$ uniform $k$-gauge. Then, for every $n \ge 0$, $S \in \simp^k(\O)$,
\[  \abs{ \ang{S, ((\tau)^\prime)^n \omega} } \le [\omega]_\v  \ell^{n} \ang{(2^{-n})_\natural S, \v}.\]
If moreover $\v$ is $(\log_2 \ell )$-Dini and $\tau^\prime \omega = \omega$, then $\omega =0$.
\end{lemma}

\begin{proof}
For $n\ge 1$, $i=(i_1, \ldots, i_n) \in \cur{1, \ldots, \ell}^n$,  define
\[ \lambda_i := \lambda_{i_1} \ldots \lambda_{i_n} , \quad \tau_i S :=   \tau_{i_n} \tau_{i_{n-1}} \ldots \tau_{i_1} S \quad \text{for $S \in \simp^k(\O)$,}\]
so that, by induction, $\tau_i S$ is isometric to $(2^{-n})_\natural S$ and $(\tau^n)^\prime \omega = \sum_{i \in \cur{1, \ldots, \ell}^n} \lambda_i \tau_i^\prime \omega$.
%Indeed, the case $n=0$ is obvious and, assuming that the result holds for $n-1$, then
%\[ \ang{ C, \eta }   = \sum_{j \in \g{N}^{n-1}}  \ang{  C^j, \eta }  = \sum_{j \in \g{N}^{n-1}} \sum_{i_n \in \g{N}} \ang{ \bra{C^j}^{i_n}, \eta} =  \sum_{i \in \g{N}^{n}}  \ang{  C^i, \eta}.\]
Since $\v$ is a uniform gauge, for any $S \in \simp^k(\O)$, one has
\[ \abs{ \ang{\tau^n S, \omega}}  \le  \sum_{i \in \cur{1, \ldots, \ell }^{n}}\abs{\lambda_i} \abs{ \ang{ \tau_i S, \omega}}  \le  [\omega]_{\v} \sum_{i \in \cur{1, \ldots, \ell }^{n}} \ang{ (2^{-n})_\natural S, \v }  = \ell^{n} \ang{ (2^{-n})_\natural S, \v }.
\]
The last statement follows letting $n \to +\infty$, using $(\tau^n)^\prime\omega = (\tau^\prime)^n \omega = \omega$ and the $(\log_2 \ell)$-Dini assumption on $\v$.
\end{proof}

\begin{remark}[convergence rates]\label{rem:conv-rate}
In the situation of the result above, assume that $\omega \approx_\v \tilde{\omega}$ with $\tau^\prime \tilde{\omega} = \tilde{\omega}$. Then, we obtain the inequality
\[ \abs{ \ang{S, \tilde{\omega} - (\tau^n)^\prime \omega} } \le [\omega - \tilde{\omega}]_\v  \ell^{n} \ang{(2^{-n})_\natural S, \v}.\]
and in particular if $\v$ is $\log_2 \ell$-Dini then $\lim_{n \to +\infty} \ang{S, (\tau^n)^\prime \omega} = \ang{S, \tilde{\omega}}$.
\end{remark}

\subsection{Proof of sewing (uniqueness) Theorem~\ref{thm:sew-unique}}

For $k\in \cur{1,2}$, let $\omega \in \germ^k(\O)$ be regular, $\v \in \germ^k(\O)$ be $k$-Dini and $\omega\approx_\v 0$. To show that $\omega _S = 0$ for every $S \in \simp^k(\O)$, minding that $\omega$ is nonatomic, we may assume that $\vol_k(S) >0$, hence we can find $\varphi: I \subseteq \R^k \to \O$ affine injective such that $\varphi_\natural [0 e_1 e_2 \ldots e_k] = S$. Moreover, being $\varphi$ affine and injective, $\varphi^\natural \v$ is $k$-Dini and $\varphi^\natural \omega \approx_{\varphi^\natural \v} 0$ by Remark~\ref{rem:moduli-pullback}. Thus, without loss of generality, we can reduce ourselves to the case $\O \subseteq \R^k$,  $S  \in \simp^k(\O)$, and $\omega \in \germ^k(\O)$ closed and nonatomic.

%Let $(\dya^i)_{i=1}^{2k}$ be the geometric maps on $k$-simplices inducing the dyadic decomposition.

From the equivalence $\dya \equiv \Id$ (Lemma~\ref{lem:dyadic-decomposition}), i.e.,
\[ \dya - \Id = \nu + \partial \varrho, \]
we deduce $\omega = \dya^\prime \omega$ by regularity of $\omega$. Moreover, since $\dya^i S \equiv (2^{-1})_\natural S$, for $i = 1, \ldots, 2^k$ and $\v$ is  $k$-Dini, hence $\log_2 (2^k)=k$-Dini we are in a position to apply Lemma~\ref{lem:uniqueness-general} with $\ell := 2^k$, $\lambda_i := 1$,  $\tau_i := \dya^i$, $\tau := \dya$ and deduce $\omega = 0$.

\subsection{Proof of Stokes-Cartan Theorem~\ref{thm:stokes}}

If $k=1$ first, i.e.\ $\eta = f$ is a function, then $\sew f = f$ and $\sew \delta f = \delta f$, because $\delta f $ is regular. %the identity $\partial \dya =  \partial$ holds, hence, by Remark~\ref{rem:conv-rate},
%\[ \sew \delta f = \lim_{n \to +\infty} (\dya^n)^\prime \delta f = \lim_{n\to +\infty} \delta f = \delta f.\]

If $k=2$, differently from the proof of Theorem~\ref{thm:sew-unique}, it is more convenient to use $\dya^\dagger$ (introduced in Remark~\ref{rem:variants})  instead of $\dya$. Notice that, given a closed and nonatomic $\tilde\omega \in \germ^k(\O)$, from $\tilde\omega = \dya^\prime \tilde\omega$ (consequence of $\dya \equiv \Id$) and the fact that $\tilde\omega$ is alternating, given that $\dya$ and $\dya^\dagger$ differ only for composition with $- \sigma$ (and $\sigma [p_0 p_1 p_2] := [p_2 p_1 p_0]$) on one of the four isometric simplices, we also have $\tilde\omega = (\dya^\dagger)^\prime \tilde\omega$. Hence,
\[ \begin{split} \sew \delta \eta  & = \lim_{n \to +\infty} \bra{(\dya^\dagger)^n}^\prime \delta \eta \quad \quad \text{by Remark~\ref{rem:conv-rate} with $\tilde \omega:= \sew \delta \eta$, $\omega :=\delta \eta$}\\
&  = \lim_{n\to +\infty} \delta  \bra{\dya^n}^\prime \eta \quad \quad \text{for  $\partial \dya^\dagger = \dya \partial$}\\
& = \delta \lim_{n\to +\infty}  \bra{\dya^n}^\prime \eta  = \delta \sew \eta. \end{split}\]

\subsection{Proof of sewing (continuity) Theorem~\ref{thm:sew-continuity}}
%
% for any $\epsilon>0$ we choose $\bar{n} \ge 1$ such that
%\[ 2^{k \bar n } \ang{ (2^{-\bar{n}})_\natural S, \v} \le \epsilon.\]
Given $S \in \simp^k(D)$, by Remark~\ref{rem:conv-rate}, with $\omega^j$ instead of $\omega$, $\sew \omega^j$ instead of $\tilde{\omega}$,  $\ell := 2^k$, $\lambda_i = 1$, $\tau_i := \dya^i$, $\tau := \dya$ and $n \ge 1$, we have
\[ \abs{ \ang{S, \sew \omega^j - (\dya^n)^\prime \omega^j} } \le [\omega^j - \sew \omega^j]_\v  2^{kn} \ang{(2^{-n})_\natural S, \v}.\]
For any $\epsilon>0$, we choose $\bar{n} \ge 1$ such that
\[  2^{k \bar n } \ang{ (2^{-\bar{n}})_\natural S, \v} \le \epsilon /  \sup_{j \ge 1} [\omega^j - \sew \omega^j]_\v.\]
By the pointwise convergence of $\omega^j \to \omega$, we have $(\dya^{\bar n})^\prime\omega^j \to (\dya^{\bar n})^\prime \omega$ pointwise as well, hence
\[ \begin{split}
\abs{ \ang{S,  \sew \omega^i - \sew \omega^j } } & \le  \abs{ \ang{S,  \sew \omega^i - (\dya^{\bar{n}})^\prime \omega^i}} + \abs{ \ang{S,  \sew \omega^i - (\dya^{\bar{n}})^\prime \omega^i}} \\ & \quad + \abs{\ang{S,  (\dya^{\bar{n}})^\prime\omega^i -  (\dya^{\bar{n}})^\prime\omega^j}}  \\
& \le 3 \epsilon \quad \text{for for $i$, $j$ sufficiently large,}
\end{split}\]
so that $(\sew \omega^j)_{j \ge 1}$ is pointwise Cauchy. Its limit $\tilde{\omega}$ is a regular germ  by Corollary~\ref{cor:pointwise-limit-regluar}.  Since
\[ \abs{ \ang{ S, \sew \omega^j - \omega^j } } \le \sup_{\ell \ge 1} [\omega^\ell - \sew\omega^\ell]_\v \ang{S, \v},\]
letting $j \to +\infty$ in the above estimate, we obtain
\[ \abs{ \ang{ S, \tilde{\omega} - \omega } } \le  \sup_{\ell \ge 1} [\omega^\ell - \sew\omega^\ell]_\v \ang{S, \v}\]
hence $\tilde{\omega} \approx_\v  \omega$ or in other words $\tilde{\omega} = \sew \omega$.

\subsection{Proof of sewing (existence) Theorem~\ref{thm:sew-existence}}

 \emph{Step 1: definition of $\sew \omega$.} For any $S \in \simp^k(\O)$, we show that the sequence  $\bra {\ang{ \dya^n S, \omega}}_{n \ge 0}$ is Cauchy, hence the limit
\[ \ang{S, \sew \omega} := \lim_{n \to +\infty} \ang{ \dya^n S,  \omega}\]
is well-defined. Indeed, decomposing $\dya - \Id = \nu + \partial \varrho$, using the fact that $\omega$ is nonatomic, we have
\[ \ang{ (\dya - \Id)S, \omega} = \ang{ \varrho S, \delta \omega} \]
hence $[\dya^\prime \omega - \omega]_{\abs{\varrho^\natural}\u} \le [\delta \omega]_\u <\infty$. Lemma~\ref{lem:uniqueness-general} with $\tau^i = \dya^i$, $i = 1, \ldots, 2^k$, $\tau = \dya$ and $\dya^\prime \omega - \omega$ instead of $\omega$ entails, for $n \ge 0$,
\[ \abs{ \ang{ \dya^{n+1} S, \omega} -  \ang{ \dya^{n} S, \omega}  }= \abs{ \ang{ \dya^n S, \dya^\prime\omega - \omega}} \le [\delta \omega]_\u 2^{kn} \ang{ (2^{-n})_\natural S, |\varrho|' \u },\]
which is summable, being $|\varrho|' \u$ strong $k$-Dini by Lemma~\ref{lem:geo-dini}.
Hence the sequence is Cauchy and we also obtain the bound
\begin{equation}\label{eq:error-bound-existence}  \abs{ \ang{S, \omega-  \sew \omega}} \le [\delta \omega]_\u \sum_{n=0}^{\infty} 2^{kn} \ang{ (2^{-n})_\natural, |\varrho|' \u } =  [\delta \omega]_\u \ang{S, \v} \end{equation}
i.e., $\omega \approx_\v \sew \omega$, having defined (with $\tilde{\u}$ as in~\eqref{eq:v-dini})
\[ \ang{S, \v} := \sum_{n=0}^{\infty} 2^{kn} \ang{ (2^{-n})_\natural, |\varrho|' \u } = \ang{S, |\varrho|' \tilde{\u}},\]
Notice that, when $\u = \vol_1^{\gamma_1} \vol_2^{\gamma_2}$, one obtains that $\v \le \c(\gamma_1, \gamma_2) \vol_1^{\gamma_1} \vol_2^{\gamma_2}$  by Example~\ref{ex:gauge-diam}, yielding Remark~\ref{rem:explicit}.

To show that $\sew \omega$ is continuous, let $\cur{S^m}_{m\ge 0} \subseteq \simp^k(\O)$ converge to $S \in \simp^k(\O)$. For any $\varepsilon>0$, one can find $\bar{n} \ge 0$ such that
\[ \ang{ S, \v} - \sum_{n=0}^{\bar{n}-1} 2^{kn}\ang{ (2^{-n})_\natural S, |\varrho|'\u } = \sum_{ n = \bar{n}}^{+\infty} 2^{kn}\ang{ (2^{-n})_\natural S, |\varrho|'\u } < \varepsilon.\]
By continuity $\v$ and $\u$  we have that the same bound holds for any $S^m$ instead of $S$, provided that $m$ is sufficiently large. The $k$-germ $(\dya^{\bar{n}})^\prime\omega$ is also continuous, hence
\[ \abs{ \ang{ S^m - S , (\dya^{\bar n})^\prime \omega} } < \varepsilon\]
provided that $m$ is sufficiently large. Hence, for every $m \ge \bar{m}$, one has
\[
\begin{split} \abs{ \ang{ S^m - S , \sew \omega}} & \le  \abs{ \ang{ S^m - S ,  (\dya^{\bar n})^\prime \omega}}\\
& \quad +  \abs{ \ang{ S^m,  (\dya^{\bar n})^\prime \omega - \sew \omega }} + \abs{ \ang{ S,  (\dya^{\bar n})^\prime \omega - \sew \omega }} \le 3 \varepsilon.\end{split}.
\]

In the following steps, we prove that $\sew \omega$ is regular. Notice first that $\sew \omega$ is nonatomic and the identity $\dya^\prime \sew \omega = \sew \omega$ is a straightforward consequence of our definition:
\[ \dya^\prime \sew \omega = \dya^\prime \lim_{n \to \infty} (\dya^{n})^\prime \omega =  \lim_{n \to \infty}(\dya^{n+1})^\prime \omega = \sew \omega.\]

\emph{Step 2:  $\sew \omega$ is alternating.} Let $\sigma$ be a permutation of $\cur{0,1,k}$ and introduce the $k$-germ $\bra{ \sigma^\prime +(-1)^\sigma} \sew \omega$. Since $\sigma \equiv (-1)^\sigma \Id$, we can decompose (for some geometric $\nu$ and $\varrho$ different than above)
\[  \sigma - (-1)^\sigma\Id = \nu + \partial \varrho, \quad \text{hence} \ang{ \bra{\sigma - (-1)^\sigma \Id}S,  \omega} =  \ang{\varrho S, \delta \omega},\]
hence $\sigma^\prime \omega \approx_{|\varrho|' \u} (-1)^\sigma \omega$. On the other side, Lemma~\ref{lem:dya-commutator} gives that $\dya \sigma = \sigma \dya$, so that
\[  \dya \sigma^\prime \sew \omega = \sigma^\prime \dya \sew \omega =  \sigma^\prime \omega,\]
hence, letting $\tilde{\omega} := \sigma^\prime \omega - (-1)^\sigma \omega$, one has $\dya \tilde{\omega} = \tilde{\omega}$.
Finally, being $\sigma$ geometric, one has $\sigma^\prime \sew \omega \approx_{\abs{\sigma^\prime}\v} \sigma^\prime \omega$ hence
\[  (\sigma^\prime  - (-1)^\sigma ) \sew \omega \approx_{ \max\cur{ {\abs{\sigma^\prime}\v, \v} }} (\sigma^\prime - (-1)^\sigma) \omega  \approx_{|\varrho|' \u} 0,\]
and Lemma~\ref{lem:uniqueness-general} implies  $\sigma^\prime \sew \omega = (-1)^\sigma \sew \omega$.

\emph{Step 3:  $\flip^\prime \sew \omega = 0$.} This step is necessary (actually, defined) only for $k=2$. Since $\flip \equiv 0$, we can decompose (for some geometric $\nu$ and $\varrho$ different than above)
\[  \flip = \nu + \partial \varrho, \quad \text{hence} \quad \ang{ \flip S,  \omega} =  \ang{\varrho S, \delta \omega},\]
hence $\flip^\prime \omega \approx_{\abs{\varrho}' \u} 0$. On the other side,
Lemma~\ref{lem:dya-commutator} provides geometric maps $(\tau_i)_{i=1}^4$ with $\tau_i S$ isometric to $(2^{-1})_\natural S$ such that $\dya \flip  = \flip \tau$, hence
\[ \tau^\prime \flip^\prime \sew \omega  = \flip^\prime \dya^\prime \sew \omega = \flip^\prime \omega.\]
Finally, since $\flip$ is geometric, one also has  $\flip^\prime \omega \approx_{\abs{\flip^\prime} \v}  \sew \omega$. In conclusion we have
%\[ \abs{ \ang{S, \flip^\prime ( \omega - \sew \omega)}} = \abs{ \flip S, \omega - \sew \omega} \le \ang{(\pi + \tilde \pi)S, \v}= 4\ang{S, \v} \]
 $\flip^\prime \sew \omega \approx_{\abs{\flip}^\prime\v}  \flip^\prime \omega\approx_{\abs{\varrho}'\u} 0$, so that
Lemma~\ref{rem:conv-rate} (applied with $\tau = \sum_{i=1}^4 \tau_i$) gives $\flip^\prime\sew \omega = 0$. In combination with the fact that $\sew \omega$ is alternating, we also obtain that $(\flip^\dagger)' \sew \omega = 0$.

\emph{Step 4:  $\cut_t^\prime \sew \omega = \sew \omega$.} We argue in the case $k=2$ only, the case $k=1$ being completely analogous (using $1$-simplices instead of $2$-simplices) and even simpler because $\flip$ will not appear.  First, we show that $(\cut^n)^\prime \sew \omega = \sew \omega$, for every $n \ge 1$. For fixed $n\ge 1$,  by Lemma~\ref{lem:cut-n}, since $\cut^n \equiv \Id$, arguing as above, we obtain, as in the previous two steps,
\[ (\cut^n)^\prime \omega \approx_{|\varrho|' \u} \omega.\]
 On the other side, by Lemma~\ref{lem:dya-commutator} one has for every $S \in \simp^2(\O)$,
\[ \ang{  \dya \cut^n S, \sew \omega} = \ang{ \cut^n \dya S -   \flip^\dagger  C, \sew \omega }= \ang{ \cut^n S, \sew \omega },\]
since $(\flip^\dagger)'  \omega =  0$. Thus, $\dya (\cut^n)^\prime \sew \omega = (\cut^n)^\prime \sew \omega$. Finally, since $\cut^n$ is geometric, one has \[(\cut^n)^\prime\sew \omega \approx_{(\cut^n)'\v} (\cut^n)' \omega \approx_{|\varrho|' \u} \omega,\] and again by  Lemma~\ref{lem:uniqueness-general} we conclude that $(\cut^n)' \sew \omega = \sew  \omega$.

Next, given $t = \frac{k}{n} \in [0,1]$ rational, $S = [p p_0 p_1]$, $p_t := (1-t)p_0 + tp_1$, one has
\[\begin{split} \ang{ \cut_t [p p_0 p_1], \sew \omega} & = \ang{ [p p_0 p_t], \sew \omega} + \ang{ [p p_t p_1], \sew \omega }\\
& = \ang{ \cut^{n-k} [p p_0 p_t], \sew \omega} + \ang{ \cut^k [p p_t p_1], \sew \omega } \\
& = \ang{ \cut^{n} [p p_0 p_1], \sew \omega} = \ang{ [p p_0 p_1], \sew \omega}.\end{split}\]
The general case $t \in [0,1]$ follows by approximation: choose $t^n \to t$ with $t^n \in [0,1]$ rational, and since $\sew \omega$ is continuous,
\[ \begin{split} \ang{[p p_0 p_1], \sew \omega} & = \lim_{n \to \infty}  \ang{ [p p_0 p_{t^n}], \sew \omega} + \ang{ [p p_{t^n} p_1], \sew \omega} \\
& =  \ang{ [p p_0 p_{t}], \sew \omega} + \ang{ [p p_{t} p_1], \sew \omega}.\end{split}\]

\emph{Step 5: $\sew \omega$ is closed on $k$-planes.} Let $T  \in \germ^{k+1}(\O)$ with $\conv(T)$ contained in an affine $k$-plane. To show that $\ang{\partial T, \sew \omega} = 0$, we argue separately for the case $k=1$ and $k=2$. When $k=1$, up to permutations of $T=[p_0 p_1 p_2]$, we can assume that $p_1 = (1-t) p_0 + t p_2$, with $t \in [0,1]$, thus
\[\begin{split}  \ang{\partial [p_0 p_1 p_2], \sew \omega} &  = \ang{ [p_1 p_2] - [p_0 p_2] + [p_0 p_1], \sew \omega} \\
%&  = \ang{ [p_0 p_1] + [p_1 p_2], \sew \omega} - \ang{ [p_0 p_2], \sew \omega} \\
&  = \ang{ \cut_t [p_0 p_2], \sew \omega} - \ang{ [p_0 p_2], \sew \omega} = 0.
\end{split}\]
In case $k=2$, since $\sew \omega$ is nonatomic, we can assume that $\conv(T)$ is not contained in a line. Up to permutations of $T=[p_0 p_1 p_2 p_3]$, we can assume that the two straight lines, the one  between $p_0$ and $p_3$ and the one between $p_1$ and $p_2$ intersect at a single point $q=(1-s) p_0 + s p_3 = (1-t) p_1 + t p_2$, $s$, $t \in \R$. Again, up to permutations, we are reduced to one of the two cases: $s$, $t \in [0,1]$, or $s \in [0,1]$, $t >1$ (Figure~\ref{fig:closed}).

\begin{figure}[ht]
\begin{minipage}{.45\textwidth}\begin{center}
\begin{tikzpicture}[scale=4]
        \node (0) at (0, 0) {$p_0$};
	\node (1) at (1, 0)  {$p_1$};
	\node (2) at (0, 1) {$p_2$};
	\node (3) at (0.5, 0.5) [below] {$q$};
		\node (4) at (0.7, 0.7) {$p_3$};
%	\node (5) at (0, 0.5) {$q_1$};
%	\node (6) at (0.5, 0) {$q_2$};

	\draw[thick, -]   (0) edge (4);
        	\draw[thick, -]   (1) edge (2);

         %  \drawchain{0}{0}{1}{0}{1/3}{2/3}{0.4}{-135}{1}
          %   \drawchain{0}{0}{1/3}{2/3}{0}{1}{0.4}{-135}{1}

         %   \drawchain{1}{0}{0}{0}{1}{1}{0.4}{-135}{1}
%\draw[dashed] (0,0) -- ({1/3},{2/3});
      %  \drawchain{0.3}{0.8}{1}{1}{0}{0}{0.4}{-135}{1}
%	\draw   (0,0) node[below left]{$p_0$} (1,0) node[below]{$p_1$} (0,1) node[left]{$p_2$}
%	(0.2,0.4)  node{.} node[below right]{$p_{3}$} ({1/3},{2/3}) node[right]{$q$} ;
    \end{tikzpicture}\subcaption{$s$, $t \in [0,1]$}\end{center}
\end{minipage}
\begin{minipage}{.45\textwidth}\begin{center}
\begin{tikzpicture}[scale=4]
           \node (0) at (0, 0) {$p_0$};
	\node (1) at (1, 0)  {$p_1$};
	\node (2) at (0, 1) {$p_2$};
	\node (3) at (0.3, 0.3)  {$p_3$};
     \node (5) at (0.5, 0.5) [above right]{$q$};
%	\node (6) at (0.5, 0) {$q_2$};

	\draw[thick, -]   (0) edge (3);
        	\draw[thick, -]   (1) edge (2);
        	\draw[dashed]   (3) edge (5);
    \end{tikzpicture} \subcaption{$s \in [0,1]$, $t >1$}\end{center}
\end{minipage}\caption{The two cases in the proof of Step 5.}\label{fig:closed}
\end{figure}
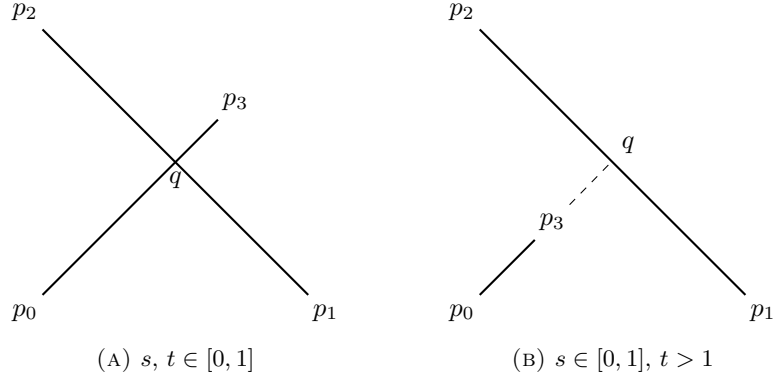

In the first case (Figure~\ref{fig:closed-proof}),
\[ \begin{split} \ang{ \partial T, \sew \omega} & = \ang{ [p_1 p_2 p_3] - [p_0 p_2p_3]+[p_0 p_1 p_3] - [p_0 p_1 p_2], \sew \omega}\\
& = - \ang{   [p_0 p_1 p_2] +  [p_3 p_2 p_1], \sew \omega} \\
& \quad + \ang{ [p_1 p_3 p_0]+[p_2 p_0 p_3], \sew \omega} \quad \text{since $\sew \omega$ is alternating}\\
& = - \ang{[p_0 p_1q ] + [p_0 q p_2] +  [p_3 p_2 q ] + [p_3 q p_1], \sew \omega}  \\
& \quad + \ang{ [p_1 p_3 q] + [p_1 qp_0]+[p_2 p_0 q ] + [p_2 q p_3], \sew \omega} \\
& \hspace{0.5\textwidth} \text{ since  $(\cut_t)^\prime \sew \omega = \sew \omega$}\\
& =  - \ang{[p_0 p_1q ] + [p_0 q p_2] +  [p_3 p_2 q ] + [p_3 q p_1], \sew \omega} \\
 & \quad   + \ang{   [p_3 q p_1 ] + [p_0p_1 q]+[ p_0 q p_2] + [p_3p_2 q ], \sew \omega} \\
 & \hspace{0.5\textwidth} \text{since $\sew \omega$ is alternating}\\
 &  = 0.
\end{split}
\]

Similarly, in the second case,
\[ \begin{split} \ang{ \partial T, \sew \omega} & = \ang{ [p_1 p_2 p_3] - [p_0 p_2p_3]+[p_0 p_1 p_3] - [p_0 p_1 p_2], \sew \omega}\\
& = - \ang{   [p_0 p_1 p_2] +  [p_3 p_2 p_1], \sew \omega} \\
& \quad + \ang{ [p_1 p_3 p_0]+[p_2 p_0 p_3], \sew \omega} \quad \text{since $\sew \omega$ is alternating}\\
& = - \ang{[p_0 p_1q ]+[p_0 q p_2] +  [p_3 p_2 q ] + [p_3 q p_1], \sew \omega}  \\
& \quad + \ang{ [p_1 qp_0] -  [p_1 qp_3 ] +  [p_2 q p_0] - [p_2 q p_3], \sew \omega} \\
 & \hspace{0.5\textwidth}  \text{ since  $(\cut_t)^\prime \sew \omega = \sew \omega$}\\
& =  - \ang{[p_0 p_1q ] + [p_0 q p_2] +  [p_3 p_2 q ] + [p_3 q p_1], \sew \omega} \\
 & \quad   + \ang{   [p_3 q p_1 ] + [p_0p_1 q]+[ p_0 q p_2] + [p_3p_2 q ], \sew \omega}\\
 & \hspace{0.5\textwidth}  \text{since $\sew \omega$ is alternating}\\
 &  = 0.
\end{split}
\]

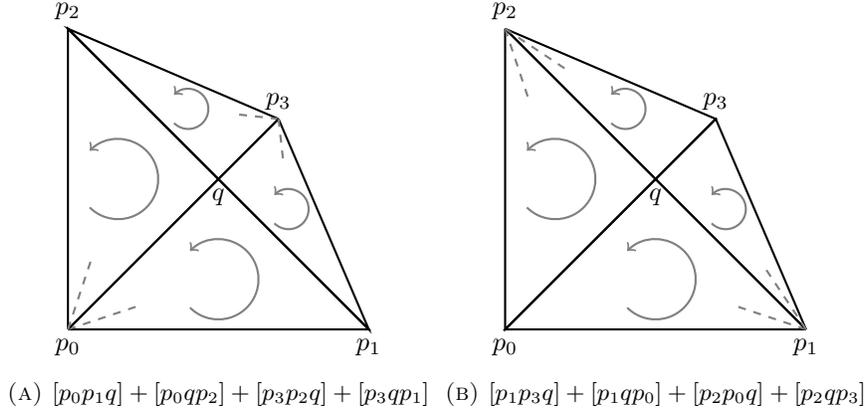
\begin{figure}[ht]
\begin{minipage}{.45\textwidth}\begin{center}
\begin{tikzpicture}[scale=4]
        \node (0) at (0, 0) [below]{$p_0$};
	\node (1) at (1, 0)  [below]{$p_1$};
	\node (2) at (0, 1) [above] {$p_2$};
	\node (3) at (0.5, 0.5) [below] {$q$};
		\node (4) at (0.7, 0.7)[above] {$p_3$};
%	\node (5) at (0, 0.5) {$q_1$};
%	\node (6) at (0.5, 0) {$q_2$};

           \drawchain{0}{0}{1}{0}{0.5}{0.5}{0.4}{-135}{1}
           \drawchain{0}{0}{0.5}{0.5}{0}{1}{0.4}{-135}{1}
               \drawchain{0.7}{0.7}{0}{1}{0.5}{0.5}{0.2}{-135}{1}
           \drawchain{0.7}{0.7}{0.5}{0.5}{1}{0}{0.2}{-135}{1}

         %   \drawchain{1}{0}{0}{0}{1}{1}{0.4}{-135}{1}
%\draw[dashed] (0,0) -- ({1/3},{2/3});
      %  \drawchain{0.3}{0.8}{1}{1}{0}{0}{0.4}{-135}{1}
%	\draw   (0,0) node[below left]{$p_0$} (1,0) node[below]{$p_1$} (0,1) node[left]{$p_2$}
%	(0.2,0.4)  node{.} node[below right]{$p_{3}$} ({1/3},{2/3}) node[right]{$q$} ;
    \end{tikzpicture}\subcaption{$[p_0 p_1q ] + [p_0 q p_2] +  [p_3 p_2 q ] + [p_3 q p_1]$ }\end{center}
\end{minipage}
\begin{minipage}{.45\textwidth}\begin{center}
\begin{tikzpicture}[scale=4]
        \node (0) at (0, 0) [below]{$p_0$};
	\node (1) at (1, 0)  [below]{$p_1$};
	\node (2) at (0, 1) [above] {$p_2$};
	\node (3) at (0.5, 0.5) [below] {$q$};
		\node (4) at (0.7, 0.7)[above] {$p_3$};
%	\node (5) at (0, 0.5) {$q_1$};
%	\node (6) at (0.5, 0) {$q_2$};

           \drawchain{1}{0}{0.5}{0.5}{0}{0}{0.4}{-135}{1}
           \drawchain{0}{1}{0}{0}{0.5}{0.5}{0.4}{-135}{1}
               \drawchain{0}{1}{0.7}{0.7}{0.5}{0.5}{0.2}{-135}{1}
           \drawchain{1}{0}{0.7}{0.7}{0.5}{0.5}{0.2}{-135}{1}

    \end{tikzpicture} \subcaption{$[p_1 p_3 q] + [p_1 qp_0]+[p_2 p_0 q ] + [p_2 q p_3]$}\end{center}
\end{minipage}\caption{Decompositions in the proof of Step 5, first case.\label{fig:closed-proof}}
\end{figure}

\section{Bound improvement for regular germs}

Aim of this section is to show the following result valid for regular $2$-germs.

\begin{lemma}\label{lem:bound-improvement}
For $\gamma \in [1,2]$ there exists  $\c = \c(\gamma)$ such that if $\omega \in \germ^2(D)$ is regular then
\[ \c^{-1} \sqa{ \omega}_{\diam^\gamma} \le [\omega]_{\diam^{2-\gamma} \vol_2^{\gamma-1} } \le \c \sqa{ \omega}_{\diam^{\gamma}}.\] % < \infty$ if Then
\end{lemma}

\begin{proof}
We already noticed in Section~\ref{sec:gauges} that $\sqa{ \omega}_{\diam^\gamma} \le\c  [\omega]_{\diam^{2-\gamma} \vol_2^{\gamma-1} }$ always holds (even if $\omega$ is not regular). To prove the other inequality, we let $S = [pqr] \in \simp^2(\O)$ and notice that it is sufficient to prove that, for some $\c = \c(\gamma)$,
\begin{equation} \label{eq:triangale-rett} \ang{S, \omega} \le \c \diam^{2-\gamma}(S) \vol_2^{\gamma-1}(S) [\omega]_{\diam^\gamma}\end{equation}
in the case that the triangle is rectangle at $q$. Indeed, to obtain the general case, it is sufficient to cut the longest side of $S$ through its height, obtaining two rectangle triangles $S^1$ and $S^2$. Using the fact that $\omega$ is regular (so  $\omega = \cut_t^\prime \omega$ for any $t \in [0,1]$), we have
\[ \begin{split} \ang{S, \omega}  & = \ang{S^1 + S^2, \omega}\\
&  \le 2 \c \max\cur{ \diam^{2-\gamma}(S^1) \vol_2^{\gamma-1}(S^1), \diam^{2-\gamma}(S^1) \vol_2^{\gamma-1}(S^2)}[\omega]_{\diam^\gamma}\\
&  \le 2 \c \diam^{2-\gamma}(S) \vol_2^{\gamma-1}(S)[\omega]_{\diam^\gamma}.
\end{split}\]
hence the thesis with $2 \c$ instead of $\c$.

To prove~\eqref{eq:triangale-rett}, we assume (without loss of generality, up to a permutation) that $\ell := |q-p| \le |q-r|=:L$ and decompose $S$ into $2n+1$ triangles $(S^i)_{i=1}^{2n+1}$ such that $\diam(S^i) \le 4 \ell$,
\[ \ang{S, \omega}  = \sum_{i=1}^{2n+1} \ang{S^i, \omega},\]
and $n  \le 2 L/\ell$. Once this decomposition is performed, we obtain
\[ \begin{split} \abs{ \ang{S, \omega}} & \le \sum_{i=1}^{2n+1} \abs{ \ang{S^i, \omega}}  \le [\omega]_{\diam^{\gamma}} (2n+1) \ell^\gamma \\
& \le 2^5 [\omega]_{\diam^{\gamma}}L \ell^{\gamma-1}  \le 2 [\omega]_{\diam^{\gamma}}L^{1-\gamma} (L\ell)^{\gamma-1}\\
& \le 8 \cdot 2^{\gamma-1} [\omega]_{\diam^{\gamma}} \diam(S)^{1-\gamma} \vol_2(S)^{\gamma-1}\end{split}\]
since $L \ell /2 = \vol_2(S)$.

Finally, to obtain such a decomposition, we argue recursively using the $\cut_t$ operation to produce a ``Christmas tree'' like decomposition of $[pqr]$ as in Figure~\ref{fig:tree}. Precisely, given any $n \ge 1$, for $k = 1, \ldots, n$, let $p_k := (kp + (n-k)r)/n$ and $q_k := (kq+ (n-k)r)/n$. Then,
\[ \begin{split} \ang{ [pqr], \omega} & = \ang{ \cut_{1/n} [pqr], \omega} \\
& =  \ang{  [q_{n-1} pq]+ [q_{n-1}rp], \omega} \\
& =  \ang{   [q_{n-1} pq]+  \cut_{(n-1)/n}[q_{n-1}rp], \omega}\\
& =  \ang{   [q_{n-1} pq] + [p_{n-1}p q_{n-1}]+ [p_{n-1} q_{n-1} r], \omega} \quad \text{(see Figure~\ref{fig:tree})}\\
& = \sum_{k=1}^n \ang{   [q_{k-1} p_{k} q_{k}] + [p_{k-1}p_k q_{k-1}], \omega} + \ang{ [p_1 q_1 r], \omega},
 \end{split}\]
 the last equality following by recursion (apply the decomposition with $n-1$ instead of $n$ to the simplex $[p_{n-1}q_{n-1} r]$). By construction, the diameter of each of the simplices $(S^i)_{i=1}^{2n+1}$ in the last line above is smaller than $\ell+ 3 L/n$. Therefore, letting $n$ be the smallest integer larger than $L/\ell$, so that in particular $L/\ell \le n \le 2 L / \ell$, we conclude that $\diam(S^i) \le 4 \ell$.

 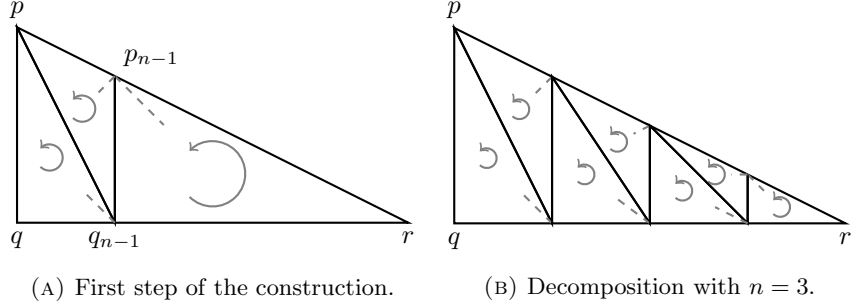
\begin{figure}[ht]
\begin{minipage}{0.45\textwidth}
\begin{center}
\begin{tikzpicture}[scale=1.3]
\drawchain{1}{1.5}{1}{0}{4}{0}{1}{-135}{1};
\drawchain{1}{1.5}{1}{0}{0}{2}{0.4}{-135}{1};
\drawchain{1}{0}{0}{0}{0}{2}{0.4}{-135}{1};
	\draw   (4,0) node[below]{$r$}  (0,0) node[below]{$q$} (0,2) node[above]{$p$} (1,0) node[below]{$q_{n-1}$} (1,1.5) node[above right]{$p_{n-1}$} ;
\end{tikzpicture} \subcaption{First step of the construction.}
\end{center}
\end{minipage}
\begin{minipage}{0.45\textwidth}
\begin{center}
\begin{tikzpicture}[scale=1.3]
\drawchain{3}{0.5}{3}{0}{4}{0}{0.3}{-135}{1};
\drawchain{1}{1.5}{1}{0}{0}{2}{0.3}{-135}{1};
\drawchain{1}{0}{0}{0}{0}{2}{0.3}{-135}{1};
\drawchain{2}{1}{2}{0}{1}{1.5}{0.3}{-135}{1};
\drawchain{2}{0}{1}{0}{1}{1.5}{0.3}{-135}{1};
\drawchain{3}{0.5}{3}{0}{2}{1}{0.3}{-135}{1};
\drawchain{3}{0}{2}{0}{2}{1}{0.3}{-135}{1};
	\draw   (4,0) node[below]{$r$}  (0,0) node[below]{$q$} (0,2) node[above]{$p$} ;
  \end{tikzpicture}\subcaption{Decomposition with $n=3$.}
\end{center}
\end{minipage}\caption{Christmas tree decomposition.}\label{fig:tree}
\end{figure}

%let $n$ the smallest integer larger than $L/\ell$,  the twice the $\cut$ operation $n$ times inductively, with In the first $n=0$ we let  $S = $

\end{proof}

\section{Proof of results on side compensators}\label{app:side-compensator}

\subsection{Proof of uniqueness of a side compensator Theorem~\ref{thm:comp-uni}}

 If $\eta \in \germ^1(\O)$ is such that $\eta\approx_\u 0$ for some $1$-Dini gauge $\u \in \germ^1(\O)$ and $\delta \eta_{prq} = 0$  for every $[prq] \in \simp^1(\O)$ such that $r = \frac{p+q}{2}$, i.e.,
 \[\eta_{pq} = \eta_{pr}+\eta_{rq},\]
 then we are in a position to apply Lemma~\ref{lem:uniqueness-general} with $\ell=2$, $\lambda_1 = \lambda_2 = 1$, $\tau_1[pq] := [pr]$ and $\tau_2[pq] := [rq]$, $\omega := \eta$ and $\v := \u$, so that $\eta = \tau \eta$ hence $\eta =0$.
% \[ \abs{\eta_{[pq]} } = \abs{ \ang{ [pq], (\tau')^n\eta}}\le  [\eta]_\u 2^n \ang{ (2^{-n})_\natural [pq], \u} \to 0 \quad  \text{as $n \to +\infty$.}\]

\subsection{Proof of existence of a  side compensator Theorem~\ref{thm:comp:ex}}

\newcommand{\bc}{\operatorname{bc}}
We define recursively $L^0(\omega)_{pq} := \omega_{prq}$ where $[pq] \in \simp^1(\O)$ and $r := \frac{p+q}{2}$, and for $n \ge 0$,
\[ L^{n+1}(\omega)_{pq} :=  L^{n}(\omega)_{pr}+L^{n}(\omega)_{rq}-\omega_{prq} .\]
Arguing by induction, we have, for every $n \ge 0$,
\begin{equation}\label{eq:recursion-boundary-corrector} \ang{[pq], L^{n+1}(\omega) - L^n(\omega)}\le [\omega]_{\u}  2^{n} \ang{ (2^{-n})_\natural [prq], \u},\end{equation}
hence, being $\u$ strong $1$-Dini, the sequence $\ang{[pq], L^n(\omega)}$ is Cauchy. The pointwise limit $\lim_{n \to + \infty} L^{n}(\omega) =: L(\omega)$ clearly satisfies~\eqref{eq:area-corrector}, i.e.,
\[  L(\omega)_{pq} =L(\omega)_{pr}+L(\omega)_{rq} -\omega_{prq} \]
as well as, by summing over $n \ge 0$ the inequalities~\eqref{eq:recursion-boundary-corrector},
\[ \abs{ \ang{[pq], L(\omega) }} \le [\omega]_{\u} \sum_{n=0}^{+\infty} 2^{n} \ang{ (2^{-n})_\natural [prq], \u} = [\omega]_{\u} \ang{ [pq], \v},\]
where we introduce the $1$-Dini gauge $\ang{  [pq], \v} := \ang{[prq], \tilde{\u}}$, with $\tilde{\u}$ given as in~\eqref{eq:v-dini}. By Example~\ref{ex:gauge-diam}, if $\u = \diam^\alpha$ (with $\alpha>1$) we obtain that $\v = \c(\alpha) \diam^\alpha$, with $\c(\alpha) = (1-2^{1-\alpha})^{-1}$.

\subsection{Proof of Theorem~\ref{thm:cancellation}}

%We begin with the following lemma.
%
%\begin{lemma}[prism]
%Let $\omega \in \germ^2(\O)$ be alternating. Then, for every $p_0
%\end{lemma}
First, we notice that Theorem~\ref{thm:comp:ex} implies that $L(\omega) \in \germ^1(\O)$ is well-defined with $L(\omega) \approx_\v 0$ for some $2$-Dini gauge $\v$ (by Remark~\ref{rem_Dini_gauge1}, because $\u$ is assumed to be strong $2$-Dini). We define $\tilde{\omega}:= \omega - \delta L(\omega)$ which is closed on $2$-planes, alternating, and $\tilde{\omega}_{prq} = 0$ whenever $r = \frac{p+q}{2}$ (by definition of $L$) hence by Remark~\ref{rem:dya-cancellations} (applied to the restriction of $\tilde\omega$ to a $2$ plane contaning $S \in \simp^2(\O)$) we have
\[ \ang{ \dya S, \tilde{\omega}} = \ang{S, \tilde{\omega}}.\]
Using that $\tilde{\omega} \approx_{\w} 0$ with the $2$-Dini gauge $\w := \max\cur{\u, \v}$,  by Lemma~\ref{lem:uniqueness-general} $m := 2^k$, $\tau_i := \dya^i$, $\tau := \dya$ we deduce  $\tilde{\omega} = 0$. % we assume that $D \subsete \R^2$

\bibliographystyle{plain}
%\bibliography{biblio-sewing-two}

\end{document}